\documentclass{article}
\usepackage{amsmath, amssymb, amsthm, stackrel, graphicx, enumerate,etoolbox,setspace,mathtools}
\usepackage{titlesec, tikz, spath3,tqft,pgf}

\usepackage{tikz-cd,units}
\usetikzlibrary{knots,tqft}
\usetikzlibrary{arrows}

\newcommand{\rar}{\rightarrow}

\newcommand{\C}{\mathbb{C}}

\newcommand{\cat}[1]{\mathcal{#1}}

\newcommand{\Aut}{\textnormal{Aut}}

\newcommand{\id}{\textnormal{id}}

\newcommand{\End}{\textnormal{End}}

\newcommand{\ev}{\textnormal{ev}}

\newcommand{\Vect}{\mathbf{Vect}}

\newcommand{\tn}[1]{\textnormal{#1}}

\newcommand{\cattens}[1]{\mathop{\boxtimes}\limits_{#1}}

\newcommand{\lincat}{\mathbf{LinCat}}

\newcommand{\Rep}{\mathbf{Rep}}

\newcommand{\bxc}[1]{\cattens{\cat{#1}}}

\newcommand{\raru}[1]{\mathop{\rar}\limits^{#1}}

\newcommand{\bx}{\boxtimes}
\newcommand{\tens}[1]{\mathop{\otimes}\limits_{#1}}
\newcommand{\rlaru}[1]{\mathop{\longrightarrow}\limits^{#1}}
\newcommand{\dcentcat}[1]{\cat{Z}(\cat{#1})}

\newcommand{\SA}{\cat{O}(\cat{A})}

\newcommand{\ZAXBT}{\dcentcat{A}\tn{-}\mathbf{XBT}}

\newcommand{\BFCA}{\mathbf{BFC}/\cat{A}}
\newcommand{\GXBT}{G\tn{-}\mathbf{XBT}}

\newcommand{\rt}{\cattens{\tn{red}}^\cat{A}}
\newcommand{\MC}[2]{\cat{Z}_2(\cat{#1},\cat{#2})}
\newcommand{\mme}{\tn{MME}}

\newcommand{\DE}{\mathbf{DE}}
\newcommand{\BTCA}{\mathbf{BTC}(\cat{A})}
\newcommand{\MME}[1]{\tn{MME}(\cat{#1},\cat{A})}
\newcommand{\etalchar}[1]{$^{#1}$}

\usetikzlibrary{shapes.geometric}
\usetikzlibrary{positioning}

\newtheorem{thm}{Theorem}
\newtheorem{cor}[thm]{Corollary}
\newtheorem{lem}[thm]{Lemma}
\newtheorem{prop}[thm]{Proposition}
\theoremstyle{definition}
\newtheorem{ex}[thm]{Example}
\newtheorem{df}[thm]{Definition}
\newtheorem{rmk}[thm]{Remark}

\newtheorem{notation}[thm]{Notation}

\title{The Reduced Tensor Product of Braided Tensor Categories containing a Symmetric Fusion Category}
\author{Thomas A. Wasserman}

\begin{document}
	\maketitle
	\begin{abstract}
		We constuct the reduced tensor product $\rt$: a symmetric monoidal structure on the 2-category $\BTCA$ of braided tensor categories containing a fixed symmetric fusion subcategory $\cat{A}$. The construction only depends on the braiding and monoidal structure of the categories involved. The main tool in the construction is an enriching 2-functor that is shown to give an equivalence between $\BTCA$ and a 2-category $\ZAXBT$ of so-called Drinfeld centre crossed braided tensor categories.
		
		As an application of the reduced tensor product we give a pairing between minimal modular extensions of braided tensor categories containing $\cat{A}$ as their transparent subcategory.
	\end{abstract}
	\tableofcontents

\section{Introduction}
\subsection{Aim}
This paper addresses the following problem: given two braided tensor categories containing a fixed symmetric fusion category $\cat{A}$, how does one form a tensor product of these categories that is braided, while condensing the two copies of $\cat{A}$ in both to a single subcategory $\cat{A}$ in the product?

As a starting point, one can form the balanced tensor product \cite{Etingof2009,Douglas2014} over $\cat{A}$, a categorification of the tensor product of modules over an algebra. The balanced tensor product gives a tensor category with the two copies of $\cat{A}$ condensed into one as desired. However, the resulting tensor category will not be braided unless the symmetric subcategory is transparent (has trivial double braiding with all other objects) in both tensor categories. The reduced tensor product\footnote{This term was used by Drinfeld \cite{Drinfeld2009a} to describe this product in the special case where the symmetric subcategory is the category of super vector spaces.} of braided tensor categories over a symmetric fusion subcategory presented here yields a braided tensor category, also when the symmetric subcategory is not transparent. Its construction is inspired by the balanced tensor product, and reduces to it when the symmetric subcategory is transparent.

\subsection{Motivation}
Motivation for this paper is work on finite gauge theory \cite{Lan2016,Lan2016a,Bruillard2016}. Summarising\footnote{A more extensive summary of the Physics background can be found in \cite[Introduction]{Wasserman2017e}.} this from a mathematical standpoint, certain physical systems are described by so-called minimal modular extensions. These are modular tensor categories containing a symmetric fusion category $\cat{A}$, associated to a braided tensor category $\cat{C}$ with M\"uger centre (or transparent subcategory) $\MC{\cat{A}}{\cat{C}}$ exactly equal to $\cat{A}$. Taking a tensor product of such physical systems does not correspond to the usual Deligne tensor product of linear categories, one would like the resulting category to contain just a single copy of $\cat{A}$. As alluded to above, the naive approach of taking a balanced tensor product over $\cat{A}$ does not yield a braided tensor category. The reduced tensor product presented here is as the appropriate tensor product in this setting, and corresponds to the product of minimal modular extensions constructed in \cite{Lan2016} . The construction done there relies on Tannaka duality and Ostrik's results on the correspondence between module categories and algebra objects \cite{Ostrik2003}. Our construction of the reduced tensor product uses the braiding and tensor products of the categories involved directly, and furthermore does not rely on a choice of fibre functor for $\cat{A}$.  

The construction done in this paper is additionally motivated by the desire to better understand braided tensor categories containing a symmetric fusion category $\cat{A}$ and their (de-)equivariantisation. In the particular case where $\cat{A}$ is Tannakian (has only positive twists) and equal to $\MC{\cat{C}}{\cat{C}}$, one can produce (see \cite{Bruguieres2000,Muger2004}) a modular tensor category $\cat{C}/\cat{A}$ called the modularisation. This category is constructed by first enriching $\cat{C}$ over $\cat{A}$ (by representing the action of $\cat{A}$ on $\cat{C}$), and then applying the fibre functor for $\cat{A}$ and idempotent completing. In the super-Tannakian case (with negative twists), this procedure is also understood, and yields a ``super-modular'' category. The intermediate step, before applying the fibre functor, is a braided $\cat{A}$-tensor category, which is ``modular'' in the sense that the braiding is non-degenerate. In the case where $\cat{A}=\Rep(G)$ and is not transparent in $\cat{C}$, the same procedure of enriching and then applying the fibre functor now yields a so-called $G$-crossed braided tensor category \cite{Drinfeld2009,Turaev2010a,KirillovJr.2002}. However, the intermediate step is not well-understood, it is still an $\cat{A}$-tensor category, but no longer braided. In this paper we fill this gap by explaining in what sense the intermediate category is a braided object.

\subsection{Overview of the construction}
To give some intuition for our construction of the reduced tensor product, let us first take a closer look at the balanced tensor product. One way of constructing the balanced tensor product over a symmetric fusion category $\cat{A}$ is by using the fact \cite{Douglas2014a} that we can enrich module categories over $\cat{A}$ to $\cat{A}$-enriched categories. For categories enriched over a symmetric monoidal category, there is a natural notion of enriched cartesian product. This is the category with as objects direct sums of pairs of objects from the categories in the product, and as hom-objects the tensor product in $\cat{A}$ of the hom-objects. The balanced tensor product of two module categories can then be formed by enriching both categories in this fashion, taking the enriched cartesian product, and then de-enriching the resulting $\cat{A}$-enriched category by applying the functor $\cat{A}(\mathbb{I}_\cat{A},-)$ to the hom-objects. Here $\mathbb{I}_\cat{A}$ denotes the monoidal unit of $\cat{A}$. If one is interested in semi-simple and idempotent complete module categories, this construction is followed by a Cauchy completion.

As we will discuss in this paper, the balanced tensor product takes braided tensor categories containing $\cat{A}$ to tensor categories, but not to braided tensor categories. In fact, this is already visible at the level of the enriched categories associated to braided tensor categories containing $\cat{A}$, these will be ``$\cat{A}$-tensor'', with monoidal structure that factors through the $\cat{A}$-enriched cartesian product, but not braided. Consequently one should have no hope that the enriched cartesian product of two categories obtained in this way is braided. This in turn means the de-enrichment will not be braided in general.

To address this problem, we introduce the novel idea of taking our enrichment further. We will construct for each braided tensor category containing $\cat{A}$ an associated $\dcentcat{A}$-enriched category, where $\dcentcat{A}$ denotes the Drinfeld centre of $\cat{A}$. We show that this enriched category carries in some sense a braided monoidal structure. Similarly to the classical $\cat{A}$-enriched case, this enrichment allows an inverse $\DE$. We then define the reduced tensor product in terms of the enriched cartesian product of the $\dcentcat{A}$-enriched categories.

Underpinning this construction is the fact that $\dcentcat{A}$ is a 2-fold monoidal category. Recall that the Drinfeld centre, as introduced by Drinfeld and first written down by Majid \cite{Majid1991}, gives a braided monoidal category $\dcentcat{M}$ associated to any monoidal category $\cat{M}$, where the objects are pairs consisting of an object of $\cat{M}$ and a half-braiding between tensoring on the right and tensoring on the left with that object. The monoidal structure $\otimes_c$ on $\dcentcat{M}$ is induced from the one on $\cat{M}$, in particular we tensor the objects in pairs together as objects in $\cat{M}$. When taking the Drinfeld centre of a symmetric fusion category, the Drinfeld centre carries another, symmetric, tensor product $\otimes_s$, constructed in \cite{Wasserman2017}. These two tensor products are laxly compatible \cite{Wasserman2017a}, making $\dcentcat{A}$ into a 2-fold monoidal category.

The composition in our $\dcentcat{A}$-enriched category factors through this symmetric tensor product. That is, the resulting category is enriched over $(\dcentcat{A},\otimes_s)$, and we will denote $(\dcentcat{A},\otimes_s)$ by $\dcentcat{A}_s$ for short. The monoidal structure, however, factors through the usual tensor product $\otimes_c$, a phenomenon that we will refer to as having a $\dcentcat{A}$-crossed tensor structure. We will additionally show that the categories obtained from braided tensor categories containing $\cat{A}$ by enriching are braided in the appropriate sense. These $\dcentcat{A}$-crossed braided categories are introduced in \cite{Wasserman2017b}. The appropriate notion of enriched cartesian product $\cattens{s}$ of such categories uses the symmetric tensor product $\otimes_s$ on hom-objects.

The reduced tensor product is then defined to be the category obtained by de-enriching the product $\cattens{s}$ of the $\dcentcat{A}_s$-enrichments of two braided tensor categories containing $\cat{A}$.

\subsection{Context}
By Tannaka duality \cite{Deligne1990,Deligne2002} we have that $\cat{A}$ is equivalent to the representation category of a finite (super-)group. In \cite{Wasserman2017b}, it is shown that the 2-category $\ZAXBT$ of $\dcentcat{A}$-crossed tensor categories, equipped with the enriched cartesian product, is, for $\cat{A}=\Rep(G)$, where $G$ is a finite group, equivalent to the 2-category $\GXBT$ of $G$-crossed braided tensor categories, equipped with the degree-wise Deligne product $\cattens{G}$, and the corresponding statement for the super-group case is also spelled out. 

Let us now focus on the case $\cat{A}=\Rep(G)$ for simplicity. In \cite{Drinfeld2009}, it is established that there is an equivalence, along mutually inverse functors called equivariantisation ($\mathbf{Eq}$) and de-equivariantisation ($\mathbf{DeEq}$), between the 2-categories $\BTCA$ and $\GXBT$. The enrichment procedure given here, together with the equivalence from \cite{Wasserman2017b}, gives a factorisation of these functors, and further shows monoidality of the steps in the factorisation. In summary, we have a commutative triangle
\begin{center}\label{trianglediag}
	\begin{tikzcd}[column sep= large]
	&			(\ZAXBT,\cattens{s}) \arrow[ddl,shift left,"\DE\overline{(-)}"]	\arrow[ddr,shift left,"\overline{ \overline{(-)}}"]		&\\
	& & \\
	(\BTCA,\rt) \arrow[ruu,shift left,"\underleftarrow{\underline{(-)}}"] \arrow[rr,shift left,"\mathbf{DeEq}"] & & (\GXBT,\cattens{G})\arrow[ll,shift left, "\mathbf{Eq}"]\arrow[luu,shift left,"\mathbf{Fix}"].
	\end{tikzcd}
\end{center}
of mutually inverse symmetric monoidal equivalences of 2-categories. The proof of commutativity of this triangle is beyond the scope of this work, and can be found in \cite{Wasserman2017e}. In this article we will concern ourselves with the left hand side of this triangle.

In the present work, we do not need to refer to Tannaka duality, and all our constructions are independent of a choice of fibre functor on $\cat{A}$. Indeed, our construction only depends on the monoidal structure and the braiding of the categories involved.

The bottom side of the above triangle is the algebraic counterpart to a construction in the study of Topological Quantum Field Theories (TQFTs) called orbifoldisation. This construction takes a TQFT with values in the 2-category of linear categories $\lincat$ defined on the once-extended three-dimensional bordism category of manifolds equipped with principal $G$-bundles ($G$-TQFTs) to an oriented once-extended three-dimensional TQFT. For the latter, a full classification result is known \cite{Bartlett2015a}: these TQFTs are determined by their value on the circle. This value is a linear category equipped with the structure of a modular tensor category. A similar classification result is conjectured for the former, but here it is only known that the values on the different principal $G$-bundles over $S^1$ combine into a so-called $G$-(multi)modular category \cite{Schweigert2018}. The converse, that every such TQFT arises from a $G$-modular category, is expected to be true, but no proof is known. Orbifoldisation of TQFTs is studied in detail in \cite{Schweigert2018}, and on the value of the circle is shown to correspond to the functor $\mathbf{Eq}$. This means that one possible strategy for addressing the classification of $G$-TQFTs is then to find a geometric counterpart to the algebraic inverse $\mathbf{DeEq}$ of $\mathbf{Eq}$. As is explained in \cite{Schweigert2018}, orbifoldisation is a two-step procedure. One first constructs an oriented TQFT with values in the bicategory of $\lincat$ bundles over spans of groupoids, this step is invertible. Then one takes sections of these bundles to produce an oriented $\lincat$-valued TQFT, this step is not known to be invertible. The above triangle should be an algebraic counterpart to this factorisation, and the author expects it to be helpful in addressing the classification conjecture for $G$-TQFTs.

\subsection{Organisation}
The paper is organised as follows. We start by explaining the enriching procedure, this is summarised in Theorem \ref{cunderlliscrossedtensor}. The next section proceeds to show that this construction has an inverse, this is Theorem \ref{bfcazaxbtev}. In the final section, we construct the reduced tensor product and show it indeed defines a symmetric monoidal structure (Theorem \ref{rttheorem}). After this, we compute the reduced tensor product in specific examples, and show how it defines a pairing between minimal modular extensions.

\subsection{Acknowledgements}
The author would like to thank Chris Douglas and Andr\'e Henriques for their guidance during the project. An earlier version of the manuscript benefited greatly from Michael M\"uger's and Ulrike Tillmann's comments. Thanks go to Marina Logares for further helpful comments.

The research presented here was made possible through financial support from the Engineering and Physical Sciences Research Council for the United Kingdom, the Prins Bernhard Cultuurfonds, the Hendrik Mullerfonds, the Foundation Vrijvrouwe van Renswoude and the Vreedefonds. The author has received support from the by the Centre for
Symmetry and Deformation at the University of Copenhagen (CPH-SYM-DNRF92) and Nathalie Wahl's European Research Council Consolidator Grant (772960).

\section{Enriching}\label{enrichingsection}
Throughout this paper, we will denote by $\cat{A}$ a fixed symmetric fusion category. This first part of the paper is devoted to enriching braided tensor categories containing $\cat{A}$ as a braided subcategory over the Drinfeld centre $\dcentcat{A}$ of $\cat{A}$. We want to do this construction functorially, so we will first set up a 2-category of braided tensor categories containing $\cat{A}$. After this, the enrichment is done in two steps. First we enrich over $\cat{A}$, and see why the resulting category is $\cat{A}$-tensor, but not braided. Then we pick appropriate half-braidings for the hom-objects, lifting them to objects of $\dcentcat{A}$. This is done in a way that ensures the resulting category is braided, in the sense of being a so-called $\dcentcat{A}$-crossed braided tensor category. We will recall the relevant definitions in the next subsection.

\subsection{Preliminaries}
In this section we will introduce the basic objects of study: braided tensor categories containing a fixed symmetric fusion subcategory, and define a 2-category of such categories. After this, we recall the definition of the Drinfeld centre of $\cat{A}$, and the two laxly compatible tensor products on this. Throughout this paper we will work over an algebraically closed field of characteristic zero.

\subsubsection{Braided tensor categories containing a symmetric fusion subcategory}
We assume the reader is familiar with the theory of tensor categories and fusion categories. To avoid any confusion, we will briefly recall the basic definitions here.

By a \emph{tensor category} we will mean a category enriched in the category $\Vect$ of \emph{finite} dimensional vector spaces, is abelian for this enrichment (in particular: it is a \emph{linear category}), idempotent complete, and carries a monoidal structure that factors through the Deligne tensor product of linear categories and is right exact in both slots. Such a category is called \emph{fusion} if it is rigid, semi-simple with finitely many isomorphism classes of simple objects, and has a simple unit object.

We will often suppress the tensor product symbol and write $cc'$ for $c\otimes c'$. 

\begin{df}
	Let $\cat{A}$ be a symmetric fusion category and let $\cat{C}$ be a braided tensor category. We say that $\cat{C}$ \emph{contains} $\cat{A}$ if $\cat{C}$ comes equipped with a braided tensor functor $\cat{A}\rightarrow \cat{C}$.
\end{df}

As $\cat{A}$ is semi-simple, any tensor functor on $\cat{A}$ is automatically faithful. We are not asking that $\cat{A}$ is embedded in $\cat{C}$, ie. the inclusion need not be full.

Throughout this paper $\cat{C}$ will be used to denote a braided tensor category containing $\cat{A}$.

\subsubsection{The 2-category of braided tensor categories containing $\cat{A}$}
In defining the 2-category $\BTCA$ of braided tensor categories containing $\cat{A}$, there are several choices to be made, we use the following definition:

\begin{df}\label{bfcdef}
	The \emph{2-category $\BTCA$ of braided tensor categories containing the symmetric fusion category $\cat{A}$} is the 2-category with
	\begin{itemize}
		\item objects: braided tensor categories $\cat{C}$ containing $\cat{A}$, such that tensoring by $\cat{A}$ is exact;
		\item morphisms: triples $(F,\mu_{-1},\mu_{0},\mu_1)$, where $(F,\mu_{-1},\mu_1)$ is a braided tensor functor (that is, $F:\cat{C}\rar \cat{C}'$ is a right exact linear functor, $\mu_{-1}\colon F(\mathbb{I}_\cat{C})\xrightarrow{\cong}\mathbb{I}_{\cat{C}'}$ an isomorphism compatible with the unitors and $\mu_1\colon F(-\otimes_{\cat{C}}-)\Rightarrow F(-)\otimes_{\cat{C}'}F(-)$ a natural isomorphism compatible with the associators), that preserves the inclusion of $\cat{A}$ up to a chosen monoidal natural isomorphism $\mu_{0}$;
		\item 2-morphisms: monoidal natural transformations $\eta$ between $(F,\mu_{-1},\mu_0,\mu_1)$ and $(G,\nu_{-1},\nu_0,\nu_1)$ that satisfy $\nu_{0}\circ\eta|_\cat{A}=\mu_{0}$.
	\end{itemize}
\end{df}

The goal of this paper is to define a symmetric monoidal structure on this 2-category.

\subsubsection{Enriched cartesian product}
We will need the definition of the enriched cartesian product of categories enriched and tensored over a symmetric category. By tensored we mean:

\begin{df}
	Let $\cat{V}$ be a symmetric finite tensor category, and let $\cat{U}$ be a $\cat{V}$-enriched category. Then \emph{a $\cat{V}$-tensoring for $\cat{U}$} is a monoidal functor $\cat{V}\rar \Aut(\cat{U})$ sending $v\mapsto v\cdot -$ such that
	\begin{equation}\label{tensoringchar}
	\cat{U}(v\cdot u, u') \cong \cat{V}(v, \cat{U}(u,u'))
	\end{equation}
	naturally.
\end{df}

Given a symmetric tensor category $\cat{V}$ with right exact monoidal structure there is a $\cat{V}$-enriched and tensored category $\underleftarrow{\cat{V}}$ obtained by finding representing objects $\underleftarrow{\cat{V}}(v,v')$ for $\cat{V}(-\otimes v,v')$. These objects will satisfy Equation \eqref{tensoringchar} by definition.

The enriched cartesian product is defined as follows, for more details see \cite[Section A.1.2]{Wasserman2017b}.

\begin{df}\label{enrichedcartprod}
	Let $\cat{V}$ be a symmetric fusion category. The \emph{enriched cartesian product $\cat{U}\bxc{V}\cat{W}$} of two $\cat{V}$-enriched and tensored categories $\cat{U}$ and $\cat{W}$ is the Cauchy completion of the $\cat{V}$-enriched category whose objects are symbols $u\boxtimes w$ with $u\in \cat{U}$ and $w\in \cat{W}$, and whose hom-objects are given by
	\begin{equation}
		\cat{U}{\bxc{V}\cat{W}}\left(u\boxtimes w,u'\boxtimes w'\right)=\cat{U}(u,u')\otimes_\cat{V}\cat{W}(w,w').
	\end{equation}
	Composition is given by first applying the symmetry in $\cat{V}$ and then the compositions in $\cat{U}$ and $\cat{W}$.
	
	The unit for $\bxc{V}$ is $\underleftarrow{\cat{V}}$, with unitor $\underleftarrow{\cat{V}}\bxc{V}\cat{U}\rar \cat{U}$ defined by $v\boxtimes u\mapsto v\cdot u$, where $\cdot$ denotes the $\cat{V}$-tensoring of $\cat{U}$.
\end{df}

\subsubsection{The Drinfeld centre of a symmetric fusion category}
We recall the definition of the Drinfeld centre of a monoidal category for convenience.
\begin{df}\label{drinfeldcenterdef}
	Let $\cat{M}$ be a monoidal category. The \emph{Drinfeld centre $\cat{Z}(\cat{M})$ of $\cat{M}$} is the braided monoidal category with objects pairs $(m,\beta)$ where $m$ is an object of $\cat{M}$ and $\beta$ is a natural transformation
	$$
	\beta\colon  - \otimes m \Rightarrow m\otimes -,
	$$
	called a half-braiding. The $\beta$ are further required to satisfy 
	\begin{equation}\label{halfbraidreq}
	\beta_{nn'}= ( \beta_{n}\otimes\id_{n'})\circ( \id_{n}\otimes\beta_{n'}),
	\end{equation}
	for all $n,n'\in\cat{M}$.
	
	The morphisms in $\cat{Z}(\cat{M})$ are those morphisms in $\cat{M}$ that commute with the half-braidings in the obvious way. The tensor product $\otimes_c$ is the one on $\cat{M}$ with consecutive half-braiding, and the braiding is the one specified by the half-braidings. 
	
	The Drinfeld centre comes with a monoidal \emph{forgetful functor} $\Phi\colon \dcentcat{A}\rar \cat{A}$, which forgets the half-braiding.
\end{df}

In the string diagram calculus for $\dcentcat{M}$ we will use
$$
\beta_z=
\hbox{
	\begin{tikzpicture}[baseline=(current  bounding  box.center)]
		\node (c) at (0.5,0){$z$};
		\node (a) at (1, 0){$m$};

		\coordinate (co) at (1,1);
		\coordinate (ao) at ( 0.5,1);
		\begin{knot}[clip width=4]
			\strand [thick] (a) to [out=90,in=-90] (ao);
			\strand [thick] (c) to [out=90,in=-90] (co);
		\end{knot}
	\end{tikzpicture}
}
$$
to depict the half-braiding of $(m,\beta)\in\dcentcat{M}$ with $z\in\dcentcat{M}$.

The braiding on the Drinfeld centre of a tensor category is non-degenerate, in the sense that no object other than the unit braids trivially with all other objects. In particular, if $\cat{R}$ is a spherical fusion category, then $(\dcentcat{R},\otimes_c,\beta)$ is a modular tensor category. The Drinfeld centre $\dcentcat{B}$ of a braided monoidal category $(\cat{B},\otimes, \beta)$ comes equipped with an inclusion $\cat{B}\hookrightarrow\dcentcat{B}$ taking $b\mapsto \beta_{-,b}$. 

It was shown in \cite{Wasserman2017} that the Drinfeld centre $\dcentcat{A}$ of a symmetric fusion category carries a second, symmetric, tensor product $\otimes_s$. To recall this, let us first a set of representatives $\SA$ of the isomorphism classes of simple objects of $\cat{A}$. Denote by $d_i$ the quantum dimension of the simple object $i$, and by $D=\sum_{i\in\cat{O}(\cat{A})}d_i^2$ the global dimension of $\cat{A}$. The symmetric tensor product is defined as follows:

\begin{df}[\cite{Wasserman2017}]\label{symtensdef}
	Let $z=(a,\beta)$ and $z'=(a',\beta')$ be objects of $\dcentcat{A}$ and let $\Pi_{z,z'}$ be the idempotent
	$$
	\Pi_{z,z'}=
	\hbox{
		\begin{tikzpicture}[baseline=(current  bounding  box.center)]
			\coordinate (west) at (-1,0);
			\coordinate (north) at (0,0.5);
			\coordinate (east) at (1,0);
			\coordinate (south) at (0,-0.4);
			\node (a) at (-0.5,-1.2) {$z$};
			\node (b) at (0.5,-1.2) {$z'$};
			\coordinate (ao) at (-0.5,1);
			\coordinate (bo) at (0.5,1);
			
			\begin{knot}[clip width=4]
				\strand [blue, thick] (west)
				to [out=90,in=-180] (north)
				to [out=0,in=90] (east)
				to [out=-90,in=0] (south)
				to [out=-180,in=-90] (west);
				\strand [thick] (a) to (ao);
				\strand [thick] (b) to (bo);
				\flipcrossings{1,4}
			\end{knot}
		\end{tikzpicture}
	}
	=\sum_{i\in\cat{O}(\cat{A})}\frac{d_i}{D}
	\hbox{
		\begin{tikzpicture}[baseline=(current  bounding  box.center)]
			\coordinate (west) at (-1,0);
			\coordinate (north) at (0,0.5);
			\coordinate (east) at (1,0);
			\coordinate (south) at (0,-0.4);
			\node (a) at (-0.5,-1.2) {$z$};
			\node (b) at (0.5,-1.2) {$z'$};
			\coordinate (ao) at (-0.5,1);
			\coordinate (bo) at (0.5,1);
			
			\node (i) at (-1.25,0){$i$};
			
			\begin{knot}[clip width=4]
				\strand [thick] (west)
				to [out=90,in=-180] (north)
				to [out=0,in=90] (east)
				to [out=-90,in=0] (south)
				to [out=-180,in=-90] (west);
				\strand [thick] (a) to (ao);
				\strand [thick] (b) to (bo);
				\flipcrossings{1,4}
			\end{knot}
		\end{tikzpicture}
	},
	$$
	where the last equation introduces the notation that an unlabelled ring in a string diagram represents a weighted sum over simple objects of $\cat{A}\subset\dcentcat{A}$.	Write $z\otimes_\Pi z'$ for the object associated to this idempotent, then the \emph{symmetric tensor product} of $z$ and $z'$ is defined to be
	$$
	z\otimes_s z'=(\Phi(z\otimes_\Pi z'), \mathfrak{b} ),
	$$ 
	where the half-braiding $\mathfrak{b}$ is the natural isomorphism with components
	\begin{equation}\label{symtenshalfbraid}
	\mathfrak{b}_a=
	\hbox{
		\begin{tikzpicture}[baseline=(current  bounding  box.center)]
			
			\node (cf) at (0,-0.5){$z\otimes_s z'$};
			\node (inc) at (-0.75,-0.5){$a$};
			
			\node (tr) at (0,0.1) {$\bigtriangledown$};
			\coordinate (trd) at (0,0);
			\coordinate (truw) at (-0.1,00.2);
			\coordinate (true) at (0.1,0.2);

			\node (tr1) at (0,2.2) {$\bigtriangleup$};
			\coordinate (trd1w) at (-0.1,2.12);
			\coordinate (trd1e) at (0.1,2.12);
			\coordinate (tru1) at (0,2.35);
			\node (c1) at (0,3) {};
			
			\coordinate (outg) at (0.75,2.8);
			
			\begin{knot}[clip width=4]
				\strand [thick] (cf) to (trd);
				\strand [thick] (true) to [out=70, in=-70] (trd1e);
				\strand [thick] (tru1) to (c1);
				\strand [thick] (inc)  to [out=90, in =-90] (outg);
				\strand [thick] (truw) to [out=110,in=-110] (trd1w);
			\end{knot}
		\end{tikzpicture}
	}.
	\end{equation}
	Here the triangles denote the inclusion and projection morphisms between $z\otimes_c z'$ and $z\otimes_\Pi z'$. The \emph{unit $\mathbb{I}_s\in \dcentcat{A}$ for $\otimes_s$} is the object $\sum_{i\in \SA} ii^*$, equipped with the half-braiding:
		\begin{equation}\label{unithalfbraid}
			\hbox{
				\begin{tikzpicture}[baseline=(current  bounding  box.center)]
					\node (i) at (0,0) {$\mathbb{I}_s$};
					\coordinate (ip) at (-0.1,0.27);
					\coordinate (o) at (0,3);
					\coordinate (op) at (-0.1,3);
					
					\node (a) at (-0.5,0){$a$};
					\coordinate (ao) at (0.5,3);

					\begin{knot}[clip width=4]
						\strand [thick,blue] (i) to (o);
						\strand [thick,blue] (ip) to (op);
						\strand [thick] (a) to [out=90,in=-90] (ao);
					\end{knot}
					
				\end{tikzpicture}
			}
			:=
			\sum_{i,j\in\cat{O}(\cat{A})}\sum_{\phi\in B(ai,j)}
			\hbox{
				\begin{tikzpicture}[baseline=(current  bounding  box.center)]
					\node (i) at (0,0) {$i$};
					\node (o) at (0,3){$j$};
					
					\node (id) at (0.5,0){$i^*$};
					\node (do) at (0.5,3){$j^*$};
					
					\node (a) at (-0.5,0){$a$};
					\node (phi) at (0,1) [draw]{$\phi$};
					
					\node (phid) at (0.5,2)[draw]{$\phi^*$};
					\node (ao) at (1,3){$a$};
					
					\begin{knot}[clip width=4]
						\strand [thick] (i) to (phi);
						\strand [thick] (phi) to (o);
						\strand [thick] (a) to [out=90,in=-110] (phi);
						\strand [thick] (id) to (phid);
						\strand [thick] (phid) to (do);
						\strand [thick] (phid) to [out=-60,in=-90] (ao);
					\end{knot}
					
				\end{tikzpicture}
			}.
		\end{equation}
		The double strand is used to denote the identity on $\mathbb{I}_s$, $B(ai,j)$ denotes a basis for $\cat{A}(ai,j)$, and $\phi^*$ is the adjoint of the dual $\phi^t \in \cat{A}(i,ai)$ with respect to composition along $i$ of $\phi \in \cat{A}(ai,j)$ (see \cite{Wasserman2017} for details). 
		
		The right unitor is built from evaluation morphisms
		\begin{equation}\label{rightunitor}
			\hbox{
				\begin{tikzpicture}[baseline=(current  bounding  box.center)]
					\node (inc) at (0,0.5){$ z\otimes_s\mathbb{I}_s$};
					\node (bo) at (0,3){$z$};
					\node (tr) at (0,1.1) {$\bigtriangledown$};
					\coordinate (trd) at (0,1);
					\coordinate (truw) at (0.15,1.2);
					\coordinate (true) at (-0.1,1.2);
					\coordinate (trum) at (0.1,1.2);
					
					\coordinate (ctrl) at (-0.175,2.3);
					
					\begin{knot}[clip width=4,clip radius=3pt]
						\strand [thick] (inc) to (trd);
						\strand [thick,blue] (trum) to [out=60, in=180] (ctrl) to [out=0,in=60] (truw);
						\strand [thick] (true) to [out=120,in =-90] (bo);
						\flipcrossings{2}
					\end{knot}
				\end{tikzpicture}
			}
			:=
			\sum\limits_{i\in\cat{O}(\cat{A})}
			\hbox{
				\begin{tikzpicture}[baseline=(current  bounding  box.center)]
					\node (inc) at (0,0.5){$z\otimes_s \mathbb{I}_s$};
					\node (bo) at (0,3){$z$};
					\node (il) at (0.35,1.4){$i$};
					\node (tr) at (0,1.1) {$\bigtriangledown$};
					\coordinate (trd) at (0,1);
					\coordinate (truw) at (0.15,1.2);
					\coordinate (true) at (-0.1,1.2);
					\coordinate (trum) at (0.1,1.2);
					
					\coordinate (ctrl) at (-0.175,2.3);
					
					\begin{knot}[clip width=4,clip radius=3pt]
						\strand [thick] (inc) to (trd);
						\strand [thick] (trum) to [out=60, in=180] (ctrl) to [out=0,in=60] (truw);
						\strand [thick] (true) to [out=120,in =-90] (bo);
						\flipcrossings{2}
					\end{knot}
				\end{tikzpicture}
			},
		\end{equation}
		where on the left hand side the double strand coming out of the inclusion denotes the identity on $\mathbb{I}_s$ (as above). The left unitor is obtained from this by reflecting in a vertical line.
\end{df}

For more details, see \cite{Wasserman2017}. The tensor product $\otimes_s$ is indeed symmetric, with symmetry coming from the symmetry in $\cat{A}$. It is distinct from $\otimes_c$, and in \cite{Wasserman2017a} it was shown that the two tensor products are laxly compatible, in the sense that there are natural transformations between $(-\otimes_c - )\otimes_s(-\otimes_c-)$ and $(-\otimes_s - )\otimes_c(-\otimes_s-)$, that are compatible with the associators, unitors, symmetry, and braiding. This makes $\dcentcat{A}$ into a braided two-fold monoidal category. We will write $\dcentcat{A}_s$ for the symmetric tensor category $(\dcentcat{A},\otimes_s)$. 

\subsubsection{$\dcentcat{A}$-crossed braided categories}
In \cite{Wasserman2017b} the two-fold monoidal structure is used to define the notion of a $\dcentcat{A}$-crossed braided tensor category. In short, a $\dcentcat{A}$-crossed tensor category is a category which has hom-objects in $\dcentcat{A}$, with a composition that factors through $\otimes_s$, and a monoidal structure that (on morphisms) factors through $\otimes_c$. A braiding for such a category is natural transformation between the two possible orders of taking the tensor product, where we take into account that changing the order for $\otimes_c$ requires using the braiding for $\otimes_c$. To recall the formal definition, we need:

\begin{df}\label{ZAprodsdef}
	Let $\cat{K}$ and $\cat{L}$ be $(\dcentcat{A},\otimes_s)$-enriched and tensored categories. We will denote by $\cat{K}\cattens{s}\cat{L}$ their \emph{$\dcentcat{A}_s$-enriched cartesian product} from Definition \ref{enrichedcartprod}.	Their \emph{convolution product $\cat{K}\cattens{c}\cat{L}$} is the category with pairs $k\boxtimes l$ for $k\in \cat{K}$ and $l\in\cat{L}$ as its objects and 
	$$
	\cat{K}\cattens{c}\cat{L}(k\boxtimes l, k'\boxtimes l')= \cat{K}(k,k')\otimes_c \cat{L}(l,l')
	$$
	as its morphisms. The \emph{braiding functor $B$ for $\cattens{c}$} is the functor
	\begin{center}
	\begin{tikzcd}[column sep= small, row sep = 0ex,/tikz/column 1/.append style={anchor=base west},/tikz/column 2/.append style={anchor=base east},/tikz/column 3/.append style={anchor=base west}]
		B\colon& \cat{K}\cattens{c}\cat{L}	\arrow[r]&\cat{L}\cattens{c}\cat{K}\\ 
		&k \boxtimes l 				\arrow[r,mapsto]& l \boxtimes k \\
		&\cat{K}\cattens{c}\cat{L}(k\boxtimes l, k'\boxtimes l')\arrow[r,"\beta"] &\cat{L}\cattens{c}\cat{K}(l\boxtimes k, l'\boxtimes k'),
	\end{tikzcd}	
	\end{center}
	where $\beta$ denotes the braiding for $\otimes_c$. The \emph{unit $\cat{A}_\cat{Z}$ for $\cattens{c}$} is the category $\cat{A}$ enriched over itself, viewed as a $\dcentcat{A}_s$ enriched category (via the inclusion $\cat{A}\subset \dcentcat{A}$). If we denote by $\cdot$ the action of $\dcentcat{A}_s$ on $\cat{K}$, the unitor $\cat{A}_\cat{Z}\cattens{c}\cat{K}\rar \cat{K}$ is determined by $a\boxtimes k\mapsto (a\otimes_c \mathbb{I}_s)\cdot k$. 
\end{df}

We will use $\cattens{s}$ to obtain the reduced tensor product. More details, including a proof that $\cat{A}_\cat{Z}$ is indeed the unit for $\cattens{c}$, can be found in \cite{Wasserman2017b}. We now have what we need to define $\dcentcat{A}$-crossed braided tensor categories.

\begin{df}\label{zacrossedtensordef}
	A $\dcentcat{A}$-crossed tensor category is a $\dcentcat{A}_s$-enriched and tensored category $\cat{K}$, together with a functor 
	$$
	\otimes_\cat{K}\colon \cat{K}\cattens{c}\cat{K} \rar \cat{K}
	$$
	and a functor
	$$
	\mathbb{I}_\cat{K}\colon \cat{A}_\cat{Z} \rar \cat{K},
	$$
	and natural transformations expressing associativity and unitality of $\otimes_\cat{K}$ in the usual way. A braiding for $(\cat{K},\otimes_{\cat{K}})$ is a monoidal natural isomorphism
	$$
	\beta\colon \otimes_{\cat{K}}\Rightarrow \otimes_\cat{K} \circ B. 
	$$
\end{df} 

\subsubsection{The 2-category of $\dcentcat{A}$-crossed braided categories}
The $\dcentcat{A}$-crossed braided categories can be organised into a symmetric monoidal 2-category. The morphisms and 2-morphisms in this category are:

\begin{df}\label{tensorfunctsnattrafos}
	Let $(\cat{K},\otimes_\cat{K},\mathbb{I}_\cat{K})$ and $(\cat{L},\otimes_\cat{L},\mathbb{I}_\cat{L})$ be $\dcentcat{A}$-crossed tensor categories. A \emph{$\dcentcat{A}$-crossed tensor functor} from $\cat{K}$ to $\cat{L}$ is a triple $(F,\mu_0,\mu_1)$ consisting of a $\dcentcat{A}_s$-enriched functor $F\colon \cat{K}\rar \cat{L}$, and $\dcentcat{A}_s$-enriched natural isomorphisms
	\begin{align*}
		\mu_0:& F\circ\mathbb{I}_\cat{K}\Rightarrow \mathbb{I}_\cat{L},\\
		\mu_1:& F(-\otimes_{\cat{K}}-)\Rightarrow F(-)\otimes_{\cat{L}}F(-).
	\end{align*}
	This data should satisfy the usual compatibility with the associators and the unitors. 
	
	If $\cat{K}$ and $\cat{L}$ carry braidings $\beta_K$ and $\beta_L$, respectively, then $(F,\mu_0,\mu_1)$ is called \emph{braided} if the following natural transformations agree:
	\begin{align*}
	F(-\otimes_{\cat{K}}-)\xRightarrow{F(\beta_K)} F(-\otimes_{\cat{K}}-)B\xRightarrow{\mu_1}& (-\otimes_\cat{L}-)(F\cattens{c}F)B,\\
	F(-\otimes_{\cat{K}}-)\xRightarrow{ F\cattens{c} F( \mu_1)} (-\otimes_{\cat{L}}-)(F\cattens{c}F) \xRightarrow{\beta_\cat{L}}& (-\otimes_\cat{L}-)B(F\cattens{c}F).
	\end{align*}
	
	A \emph{monoidal natural transformation} between two such functors $(F,\mu_0,\mu_1)$ and $(G,\nu_0,\nu_1)$ between $\cat{K}$ and $\cat{L}$ is a $\dcentcat{A}_s$-enriched natural transformation $\eta$ that makes the following diagrams commute
	\begin{center}
		\begin{tikzcd}
			F(-\otimes_{\cat{K}}-) \arrow[r,"\mu_1",Rightarrow]\arrow[d,"\eta\circ(-\otimes -)",Rightarrow]	& F(-)\otimes_{\cat{L}}F(-)\arrow[d,"\eta\otimes \eta",Rightarrow]\\
			G(-\otimes_\cat{K} -)	\arrow[r,"\nu_{1}",Rightarrow]								& G(-)\otimes G(-)
		\end{tikzcd}
		\quad
		\begin{tikzcd}[row sep=tiny]
			F\circ \mathbb{I}_\cat{K}	\arrow[dd,"\eta\circ\mathbb{I}_\cat{K}",Rightarrow] \arrow[rd,"\mu_0",Rightarrow]	&						\\
			&\mathbb{I}_\cat{L}\\
			G\circ \mathbb{I}_\cat{K}	\arrow[ru,Rightarrow,"\nu_0"']
		\end{tikzcd}.
	\end{center}
	Here the composite of $\eta$ with a functor is to be understood as whiskering.
\end{df}

For more details, we refer the reader to \cite{Wasserman2017b}. The 2-category is then:

\begin{df}\label{zaxbtdef}
	The \emph{symmetric monoidal 2-category $\ZAXBT$ of $\dcentcat{A}$-crossed braided tensor categories} is the 2-category with
	\begin{itemize}
		\item objects $\dcentcat{A}$-crossed braided tensor categories,
		\item morphisms braided $\dcentcat{A}$-crossed tensor functors,
		\item 2-morphisms monoidal natural transformations,
		\item monoidal structure $\cattens{s}$, with swap map $S$.
	\end{itemize}
\end{df}

\subsection{Enriching over a symmetric subcategory}
It is well known (see for example \cite{Ostrik2003}, see also \cite[Proposition 2.15]{Douglas2014a}) that from a module category over a fusion category one can obtain a category enriched over the acting category. We will recall this construction here, and spell out the special features this has when the acting category is a symmetric category.

\subsubsection{The enriched category}
This section applies to a tensor category with a tensor inclusion of a (symmetric) fusion category, we will only need the braiding later.

\begin{df}\label{enrich}
	Let $\cat{C}$ be a tensor category containing a symmetric fusion category $\cat{A}$. The \emph{left-associated $\cat{A}$-enriched category $\underleftarrow{\cat{C}}$} has the same objects as $\cat{C}$ and the hom-object $\underleftarrow{\cat{C}}(c,c')$ is defined by
	\begin{equation}\label{enrichedhomdef}
	\cat{A}(a,\underleftarrow{\cat{C}}(c,c'))=\cat{C}(ac,c').
	\end{equation}
	An $\cat{A}$-point $f\colon a \rar \underleftarrow{\cat{C}}(c,c')$ of $\underleftarrow{{\cat{C}}}(c,c)'$ will be denoted by $f\colon c\rar_{a} c'$, and is called a \emph{morphism of degree $a$} from $c$ to $c'$. We will refer to morphisms mapped to each other by the isomorphism from Equation \eqref{enrichedhomdef} as \emph{mates}. Given $f\colon c\rar_a c'$ we will write $\bar{f}\colon ac\rar c'$ for its mate, and the mate of $g\colon ac\rar c'$ will be denoted by $\underline{g}\colon c\rar_a c'$.
	
	The composition morphisms, 
	$$
	\circ\colon  \underleftarrow{\cat{C}}(c',c'')\otimes\underleftarrow{\cat{C}}(c,c')\rar \underleftarrow{\cat{C}}(c,c''),
	$$
	are defined by observing that we have the following string of canonical isomorphisms
	\begin{align}\label{stringofiso}
	\begin{split}
	\cat{A}(a,\underleftarrow{\cat{C}}(c',c'')\otimes\underleftarrow{\cat{C}}(c,c'))&\cong \cat{A}(\underleftarrow{\cat{C}}(c',c'')^*\otimes a,\underleftarrow{\cat{C}}(c,c'))\\
	&\cong \cat{C}(\underleftarrow{\cat{C}}(c',c'')^*\otimes ac,c')\\
	&\cong \cat{C}(ac,\underleftarrow{\cat{C}}(c',c'')\otimes c')\\
	&\raru{\tn\ev}\cat{C}(ac,c'')\\
	&\cong \cat{A}(a,\underleftarrow{\cat{C}}(c,c'')).
	\end{split}
	\end{align}
	Here $\tn{ev}\colon \underleftarrow{\cat{C}}(c,c')\cdot c \rar c'$ is the unit of the adjunction from \eqref{enrichedhomdef}.

	Similarly, we define the \emph{right-associated $\cat{A}$-enriched category} $\underrightarrow{\cat{C}}$ by representing $a\mapsto\cat{C}(ca,c')$.
\end{df}

Note that the mate $\bar{f}$ of $f\colon c\rar_a c'$ is a morphism in $\cat{C}$. In terms of mates and the composition in $\cat{C}$, the composition of $f\colon c\rar_a c'$ and $f'\colon c'\rar_{a'} c''$ in $\underleftarrow{\cat{C}}$ is given by
\begin{equation}\label{matescomp}
f'\circ f= \underline{\bar{f}'(\id_{a'}\otimes \bar{f})},
\end{equation}
which in string diagrams reads as
\begin{equation}\label{matescompstring}
\hbox{
	\begin{tikzpicture}[baseline=(current  bounding  box.center)]
	\node (ap) at (-0.5,0){$a'$};
	\node (a) at (0,0){$a$};
	\node (c) at (0.5,0){$c$};
	
	\node (f) at (0.25,1.2)[draw,minimum width=20pt,minimum height=10pt,thick]{$\bar{f}$};
	\node (fp) at (-0.25,2.6)[draw,minimum width=20pt,minimum height=10pt,thick]{$\bar{f}'$};
	
	\node (cp) at (-0.25,3.5){$c''$};
	
	\begin{knot}[clip width=4]
	\strand [thick] (ap)
	to [out=90,in=-110] (fp);
	\strand [thick] (a)
	to [out=90,in=-110] (f);
	\strand [thick] (c)
	to [out=90,in=-75] (f);
	\strand [thick] (f)
	to [out=90,in=-60] (fp);
	\strand [thick] (fp)
	to [out=90,in=-90] (cp);
	\end{knot}
	\end{tikzpicture}
}.
\end{equation}

The object $\underleftarrow{{\cat{C}}}(c,c')$ can be represented as follows. Pick a set of representatives $\SA$ of the isomorphism classes of simple objects in $\cat{A}$. Then 
$$
\underleftarrow{{\cat{C}}}(c,c')\cong \bigoplus_{i\in\SA} \cat{C}(ic,c')i,
$$
as is easily checked using Equation \eqref{enrichedhomdef}. We also have (see \cite[Lemma A.13]{Wasserman2017b}) a canonical isomorphism
\begin{equation}\label{aintohomiso}
		a  \underleftarrow{\cat{C}}(c,c')\xrightarrow{\cong} \underleftarrow{\cat{C}}(c,ac').
\end{equation}

\begin{rmk}
	Both $\underleftarrow{\cat{C}}$ and $\underrightarrow{\cat{C}}$ are tensored over $\cat{A}$. For $\underrightarrow{\cat{C}}$, the tensoring induces a functor $\cat{A}^{\tn{mop}}\rar \End(\underrightarrow{\cat{C}})$, where $\cat{A}^\tn{mop}$ denotes the monoidal opposite of $\cat{A}$.
\end{rmk}

The $\cat{A}$-product $\bxc{A}$ (see Definition \ref{enrichedcartprod}) of an $\cat{A}$-enriched category $\underleftarrow{\cat{C}}$ obtained from Definition \ref{enrich} with itself has some nice features. First of all, the action of $\cat{A}$ on $\underleftarrow{\cat{C}}$ by tensoring on the left gives an action of $\cat{A}$ on $\underleftarrow{\cat{C}}\bxc{A}\underleftarrow{\cat{C}}$ defined by
\begin{equation}\label{tensoronprod}
	a (c\boxtimes c') = (ac)\boxtimes c'\cong c\boxtimes (ac'),
\end{equation}
for $a\in \cat{A}$ and $c,c\in\underleftarrow{\cat{C}}$, and this action is an $\cat{A}$-tensoring for the $\cat{A}$ enrichment: it admits natural isomorphisms similar to Equation \eqref{enrichedhomdef}. For more details see \cite[Proposition A.18]{Wasserman2017b}.

Corresponding to the product $f_1\boxtimes f_2$ of $f_1\colon c_1\rar_{a_1}c_1'$ and $f_2\colon c_2\rar_{a_2}c_2'$ there is, by using the $\cat{A}$-tensoring on the product $\underleftarrow{\cat{C}}\bxc{A}\underleftarrow{\cat{C}}$ and the tensor product in $\cat{C}$, a map $\bar{f}_1\otimes_\cat{C}\bar{f}_2$. It is tempting to represent this in string diagrams as
\begin{equation*}
\begin{tikzpicture}[baseline=(current  bounding  box.center)]
\node (a1) at (-1,-1){$a_1$};
\node (c1) at (-0.5,-1){$c_1$};
\node (a2) at (0.5,-1){$a_2$};
\node (c2) at (1,-1){$c_2$};

\node (f1) at (-0.75,0)[draw,minimum width=20pt,minimum height=10pt,thick]{$\bar{f}_1$};
\node (f2) at (0.75,0)[draw,minimum width=20pt,minimum height=10pt,thick]{$\bar{f}_2$};

\node (c1p) at (-0.75,1){$c_1'$};
\node (c2p) at (0.75,1){$c_2'$};

\begin{knot}
\strand[thick] (a1)
to [out=90,in=-105] (f1);
\strand[thick] (c1)
to [out=90,in=-75] (f1);
\strand[thick] (a2)
to [out=90,in=-105] (f2);
\strand[thick] (c2)
to [out=90,in=-75] (f2);

\strand[thick] (f1)
to [out=90,in=-90] (c1p);
\strand[thick] (f2)
to [out=90,in=-90] (c2p);
\end{knot}
\end{tikzpicture}.
\end{equation*}
Care should be taken, however, that, by Equation \eqref{tensoronprod}, the position of the $a$'s is immaterial. To avoid confusion we will therefore always keep the objects of $\cat{A}$ to the left when we are dealing with left enrichments. In drawing string diagrams this does mean that we need to cross $\cat{A}$-strands past $\cat{C}$-strands. To emphasise such crossings are not actual braidings in $\cat{C}$, we will draw them unresolved as follows
\begin{equation}\label{problemdiagrams}
\hbox{
	\begin{tikzpicture}[baseline=(current  bounding  box.center)]
	\node (a1) at (-1.5,-1){$a_1$};
	\node (a2) at (-1,-1){$a_2$};
	\node (c1) at (-0.5,-1){$c_1$};
	\node (c2) at (0.25,-1){$c_2$};

	\node (f1) at (-0.75,0.5)[draw,minimum width=20pt,minimum height=10pt,thick]{$\bar{f}_1$};
	\node (f2) at (0.25,0.5)[draw,minimum width=20pt,minimum height=10pt,thick]{$\bar{f}_2$};
	
	\node (c1p) at (-0.75,1.25){$c_1'$};
	\node (c2p) at (0.25,1.25){$c_2'$};
	
	\begin{knot}[clip width=4]
	\strand[thick] (a1)
	to [out=90,in=-105] (f1);
	\strand[thick] (c1)
	to [out=90,in=-75] (f1);
	\strand[thick] (c2)
	to [out=90,in=-75] (f2);
	\strand[thick] (f1)
	to [out=90,in=-90] (c1p);
	\strand[thick] (f2)
	to [out=90,in=-90] (c2p);
	\end{knot}
	\draw[thick] (a2)
	to [out=90,in=-105] (f2);
	\end{tikzpicture}
}
.
\end{equation}

When considering a morphism $f\colon c_1\bx c_2 \rar_a c_1'\bx c_2'$, we will give a string diagram presentation by first picking a factorisation $(t,f_1,f_2)$
\begin{equation}\label{factoraprodmorph}
f\colon a\raru{t} a_1a_2\raru{f_1f_2} \underleftarrow{\cat{C}}(c_1,c_1')\underleftarrow{\cat{C}}(c_2,c_2'),
\end{equation}
and then using the tensor isomorphism to find mates for $f_1$ and $f_2$. There are many different choices of factorisation for a given $f$. In terms of the triples,  we have the equivalence relation
$$
(t,f_1\circ g_1,f_2\circ g_2)\sim(g_1g_2\circ t, f_1,f_2),
$$
for $g_i:a_i'\rar a_i$ for $i=1,2$, with $t:a\rar a_1'a_2'$ and $f_i: a_ic_i\rar c_i'$ with $i=1,2$. A factorisation $(t,f_1,f_2)$ can be presented in string diagrams by
\begin{equation}\label{morphinaprod}
	\hbox{
		\begin{tikzpicture}[baseline=(current  bounding  box.center)]
			\node (c1) at (-0.5,-1.5){$c_1$};
			\node (c2) at (0.25,-1.5){$c_2$};
			\node (a) at (-1.25,-1.5){$a$};
			
			\coordinate (merge) at (-1.25,-1.25);
			
			\node (f1) at (-0.75,0.25)[draw,minimum width=20pt,minimum height=10pt,thick]{$\bar{f}_1$};
			\node (f2) at (0.25,0.25)[draw,minimum width=20pt,minimum height=10pt,thick]{$\bar{f}_2$};
			
			\node (c1p) at (-0.75,1.25){$c_1'$};
			\node (c2p) at (0.25,1.25){$c_2'$};
			
			\begin{knot}[clip width=4]
				\strand[thick] (f1)
				to [out=-90,in=90] (merge)
				to [out=-90,in=90] (a)
				to [out=90,in=-90] (merge)
				to [out=90,in=-110] (f2);
				\strand[thick] (c2)
				to [out=90,in=-75] (f2);
				\strand[thick] (f1)
				to [out=90,in=-90] (c1p);
				\strand[thick] (f2)
				to [out=90,in=-90] (c2p);
				\flipcrossings{2}
			\end{knot}
			\draw[thick](c1)
			to [out=90,in=-75] (f1);
		\end{tikzpicture}
	}.
\end{equation}
Here the trivalent vertex represents the morphism $t\colon a\rar a_1a_2$ from Equation \eqref{factoraprodmorph}. 

\subsubsection{Functors between the associated $\cat{A}$-categories}
We want to extend $\cat{C}\mapsto \underleftarrow{{\cat{C}}}$ to a 2-functor, so far we have only defined it on the objects of $\BTCA$. 

\begin{df}\label{associatedAfunctor}
	Let $(F,\mu_{-1},\mu_0,\mu_1)\colon \cat{C}\rar \cat{C}'$ be a morphism in $\BTCA$, then \emph{the associated $\cat{A}$-enriched functor} 
	$$
	\underleftarrow{F}\colon \underleftarrow{\cat{C}}\rar \underleftarrow{\cat{C}}',
	$$
	is the functor which acts as $F$ on objects. On morphisms, we define the morphisms $\underleftarrow{F}_{c,c'}$ in $\cat{A}$ by observing that the composite
	$$
	\cat{C}(ac,c')\xrightarrow{F_{ac,c'}} \cat{C}'(F(ac),F(c'))\cong\cat{C}'(aF(c),F(c')),
	$$
	gives for each $c,c'\in\cat{C}$ a natural transformation from $\cat{C}(-c,c')\colon \cat{A}\rar \Vect$ to $\cat{C}'(-F(c),F(c'))\colon \cat{A}\rar \Vect$. In this composite the last isomorphism is induced by the composite of $\mu_1$ and $\mu_{0}$. Using the Yoneda embedding, the natural transformation defined in this way induces a morphism
	$$
	\underleftarrow{F}_{c,c'}\colon \underleftarrow{\cat{C}}(c,c')\rar \underleftarrow{\cat{C}}'(Fc,Fc').
	$$
	This morphism takes the mate $\bar{f}\colon ac\rar c'$ for a morphism $f\colon c\rar_a c'$ to $\overline{\underleftarrow{F}(f)}\colon aF(c)\xrightarrow{\cong} F(ac)\xrightarrow{\bar{f}} F(c')$, where the first map is the composite of $\mu_1$ and $\mu_{0}$. 
\end{df}

We need to check that the functor $\underleftarrow{F}$ defined in this way is indeed a $\cat{A}$-enriched functor:

\begin{lem}
	The map $\underleftarrow{F}$ defined above respects composition.
\end{lem}
\begin{proof}
	We need to show that for all $c,c',c''\in\cat{C}$:
	\begin{center}
		\begin{tikzcd}
		\underleftarrow{\cat{C}}(c,c')\underleftarrow{\cat{C}}(c',c'')\arrow[r,"\circ"] \arrow[d,"\underleftarrow{F}_{c,c'}\otimes \underleftarrow{F}_{c',c''}"]	&	\underleftarrow{\cat{C}}(c,c'') \arrow[d,"\underleftarrow{F}_{c,c''}"]\\
		\underleftarrow{\cat{C}}'(Fc,Fc')\underleftarrow{\cat{C}}'(Fc',Fc'') \arrow[r,"\circ"] & \underleftarrow{\cat{C}}'(Fc,Fc'').
		\end{tikzcd}
	\end{center}
	On mates for $f\colon c\rar_{a} c'$ and $f'\colon c'\rar_{a'}c''$, the top route computes as
	$$
	a'aFc \xrightarrow{\cong} F(a'ac)\xrightarrow{\bar{f}'\circ (\id_{a'}\otimes \bar{f})} F(c''),
	$$
	whereas the bottom route becomes
	$$
	a'aFc \xrightarrow{\cong} a'F(ac) \xrightarrow{\id_{a'}F(\bar{f})} a'F(c') \xrightarrow{\cong} F(a'c') \xrightarrow{F(\bar{f}')} F(c''). 
	$$
	Using the fact that the structure isomorphisms for $F$ are natural, we can exchange the middle two morphisms to get
	$$
	a'a F(c) \xrightarrow{\cong} F(a'ac) \xrightarrow{F(\bar{f}'\circ (\id_{a} \otimes \bar{\underleftarrow{F}(f)}))} F(c''),
	$$
	where we have also used the fact that $F$ preserves composition, and that the monoidality isomorphisms for $aF(c)$ and $a'F(ac)$ compose to the monoidality isomorphism for $F(a'ac)$.
\end{proof}

For natural transformations, we use the following.

\begin{df}
		Let $\kappa\colon F\Rightarrow G$ be a 2-morphism between two morphisms in $\BTCA$ between $\cat{C}$ and $\cat{C}'$. Then the \emph{associated $\cat{A}$-enriched natural transformation $\underleftarrow{{\kappa}}\colon \underleftarrow{{F}}\Rightarrow \underleftarrow{{G}}$} is given by the mate to $\kappa$.
\end{df}

As we have added additional morphisms when defining $\underleftarrow{\cat{C}}$, we need to check this definition makes sense:

\begin{lem}\label{assnattrafonatural}
	The associated $\cat{A}$-enriched natural transformation $\underleftarrow{{\kappa}}$ for a 2-morphism $\kappa:(F,\mu_{-1},\mu_{0},\mu_1)\Rightarrow(G,\nu_{-1},\nu_{0},\nu_1)$ is indeed natural.
\end{lem}
\begin{proof}
	In shorthand notation, we need to check that for any $f\colon c\rar_{a}d$ we have that $\underleftarrow{{G}}(f)\underleftarrow{{\kappa}}_c=\underleftarrow{{\kappa}}_d\underleftarrow{{F}}(f)$. In string diagrams, the left hand side is, in terms of mates,
	$$
	
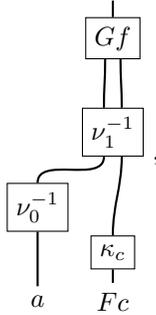
\hbox{
		\begin{tikzpicture}[baseline=(current  bounding  box.center)]
		
		\node (a) at (-0.5,0){$a$};
		\node (Fc) at (0.5,0){$Fc$};
		
		\node (k) at (0.5,0.65)[draw]{$\kappa_c$};
		\node (n0) at (-0.5,1.25)[draw]{$\nu_0^{-1}$};
		
		\node (n1) at (0.5,2.25)[draw,minimum width=1]{$\nu_1^{-1}$};
		
		\node (Gf) at (0.5,3.5)[draw]{$Gf$};
		
		\coordinate (out) at (0.5,4);

		\begin{knot}[clip width=4]
			\strand [thick] (Fc) to (k);
			\strand [thick] (a) to (n0);
			\strand [thick] (n0) to[out=90,in=-90] (n1.-110);
			\strand [thick] (k) to[out=90,in=-90] (n1.-70);
			\strand [thick] (n1.110) to[out=90,in=-90] (Gf.-110);
			\strand [thick] (n1.70) to[out=90,in=-90] (Gf.-70);
			\strand [thick] (Gf) to (out);
		\end{knot}
	\end{tikzpicture}
	},
	$$
	whereas the right hand side gives
	$$
	\hbox{
	\begin{tikzpicture}[baseline=(current  bounding  box.center)]
	
	\node (a) at (-0.5,0){$a$};
	\node (Fc) at (0.5,0){$Fc$};

	\node (n0) at (-0.5,0.6)[draw]{$\mu_0^{-1}$};
	
	\node (n1) at (0.5,1.75)[draw,minimum width=1]{$\mu_1^{-1}$};
	
	\node (Gf) at (0.5,2.75)[draw]{$Ff$};
	
	\node (kap) at (0.5,3.5)[draw]{$\kappa_{d}$};
	
	\coordinate (out) at (0.5,4);

	\begin{knot}[clip width=4]
		\strand [thick] (Fc) to[out=90,in=-90] (n1.-70);
		\strand [thick] (a) to (n0);
		\strand [thick] (n0) to[out=90,in=-90] (n1.-110);
		\strand [thick] (n1.110) to[out=90,in=-90] (Gf.-110);
		\strand [thick] (n1.70) to[out=90,in=-90] (Gf.-70);
		\strand [thick] (Gf) to[out=90,in=-90] (kap);
		\strand [thick] (kap) to (out);
	\end{knot}
	\end{tikzpicture}
	}
	=
	\hbox{
		\begin{tikzpicture}[baseline=(current  bounding  box.center)]
		
		\node (a) at (-0.5,0){$a$};
		\node (Fc) at (0.5,0){$Fc$};

		\node (n0) at (-0.5,0.6)[draw]{$\mu_0^{-1}$};
		
		\node (n1) at (0.5,1.75)[draw,minimum width=1]{$\mu_1^{-1}$};
		
		\node (kap) at (0.5,2.75)[draw]{$\kappa_{ac}$};
		
		\node (Gf) at (0.5,3.5)[draw]{$Gf$};
		
		\coordinate (out) at (0.5,4);

		\begin{knot}[clip width=4]
			\strand [thick] (Fc) to[out=90,in=-90] (n1.-70);
			\strand [thick] (a) to (n0);
			\strand [thick] (n0) to[out=90,in=-90] (n1.-110);
			\strand [thick] (n1.110) to[out=90,in=-90] (kap.-110);
			\strand [thick] (n1.70) to[out=90,in=-90] (kap.-70);
			\strand [thick] (kap.110) to[out=90,in=-90] (Gf.-110);
			\strand [thick] (kap.70) to[out=90,in=-90] (Gf.-70);
			\strand [thick] (Gf) to (out);
		\end{knot}
		\end{tikzpicture}
	}
	=
	\hbox{
		\begin{tikzpicture}[baseline=(current  bounding  box.center)]
		
		\node (a) at (-0.5,0){$a$};
		\node (Fc) at (0.5,0){$Fc$};
		
		\node (k) at (0.5,0.65)[draw]{$\kappa_c$};
		\node (n0) at (-0.5,0.6)[draw]{$\mu_0^{-1}$};
		
		\node (ka) at (-0.5,1.25)[draw]{$\kappa_a$};
		
		\node (n1) at (0.5,2.25)[draw,minimum width=1]{$\nu_1^{-1}$};
		
		\node (Gf) at (0.5,3.5)[draw]{$Gf$};
		
		\coordinate (out) at (0.5,4);

		\begin{knot}[clip width=4]
			\strand [thick] (Fc) to (k);
			\strand [thick] (a) to (n0);
			\strand [thick] (n0) to (ka);
			\strand [thick] (ka) to[out=90,in=-90] (n1.-110);
			\strand [thick] (k) to[out=90,in=-90] (n1.-70);
			\strand [thick] (n1.110) to[out=90,in=-90] (Gf.-110);
			\strand [thick] (n1.70) to[out=90,in=-90] (Gf.-70);
			\strand [thick] (Gf) to (out);
		\end{knot}
	\end{tikzpicture}
}
	=
	\hbox{
		\begin{tikzpicture}[baseline=(current  bounding  box.center)]
		
		\node (a) at (-0.5,0){$a$};
		\node (Fc) at (0.5,0){$Fc$};
		
		\node (k) at (0.5,0.65)[draw]{$\kappa_c$};
		\node (n0) at (-0.5,1.25)[draw]{$\nu_0^{-1}$};
		
		\node (n1) at (0.5,2.25)[draw,minimum width=1]{$\nu_1^{-1}$};
		
		\node (Gf) at (0.5,3.5)[draw]{$Gf$};
		
		\coordinate (out) at (0.5,4);

		\begin{knot}[clip width=4]
			\strand [thick] (Fc) to (k);
			\strand [thick] (a) to (n0);
			\strand [thick] (n0) to[out=90,in=-90] (n1.-110);
			\strand [thick] (k) to[out=90,in=-90] (n1.-70);
			\strand [thick] (n1.110) to[out=90,in=-90] (Gf.-110);
			\strand [thick] (n1.70) to[out=90,in=-90] (Gf.-70);
			\strand [thick] (Gf) to (out);
		\end{knot}
	\end{tikzpicture}
	},
	$$
	where we have used naturality of $\kappa$, monoidality of $\kappa$, and the relation $\nu_{0}\kappa|_\cat{A}=\mu_{0}$, consecutively.
\end{proof}

\subsubsection{Enriched monoidal structure}
	In this section we start making use of the braiding on $\cat{C}$, a braided tensor category containing $\cat{A}$. The tensor product on $\cat{C}$ together with the braiding between the objects of $\cat{A}$ and those of $\cat{C}$ induces an associated $\cat{A}$-monoidal structure on $\underleftarrow{\cat{C}}$ (and similarly on $\underrightarrow{\cat{C}}$). This $\cat{A}$-monoidal structure is defined as follows.
	
\begin{df}\label{inducedatensor}
	 The \emph{induced $\cat{A}$-tensor product on $\underleftarrow{\cat{C}}$} is given by $\otimes_\cat{C}$ on objects. On morphisms it is given by the map
	\begin{equation}\label{proofgoal2}
	\mathop{\otimes}\limits_{\underleftarrow{\cat{C}}}\colon  \underleftarrow{\cat{C}}(c_1,c_1')\tens{\cat{A}}\underleftarrow{\cat{C}}(c_2,c_2')\rar \underleftarrow{\cat{C}}(c_1c_2,c_1'c_2'),
	\end{equation}
	which is obtained as follows. Given $f\colon c_1 \boxtimes c_2 \rar_a c'_1 \boxtimes c'_2$, pick a factorisation $(t,f_1,f_2)$ with $t\colon a \rar a_1 a_2$ and $f_i \colon c_i \rar_{a_i} c_i'$ for $i=1,2$. Map $f_1\otimes f_2$ along the composite
	\begin{align}
	\begin{split}\label{monfirstdef}
	\cat{A}(a_1,\underleftarrow{\cat{C}}(c_1,c_1'))\tens{\Vect} \cat{A}(a_2,\underleftarrow{\cat{C}}(c_2,c_2'))	&\cong \cat{C}(a_1c_1,c_1')\tens{\Vect}\cat{C}(a_2c_2,c_2')\\
	&\rlaru{\otimes_\cat{C}} \cat{C}(a_1c_1a_2c_2,c_1'c_2')\\
	&\rlaru{(\beta_{a_2,c_1})^*}\cat{C}(a_1a_2c_1c_2,c_1'c_2'),
	\end{split}
	\end{align}
	where in the first line we used the tensor product of the tensor structure on $\underleftarrow{\cat{C}}$ with itself, the monoidal structure in $\cat{C}$ in the second line and the braiding between $a_2$ and $c_1$ in the last line. The image of $f\colon c_1 \boxtimes c_2 \rar_a c'_1 \boxtimes c'_2$ under the map \eqref{proofgoal2} is defined to be the mate of the image of $f_1 \otimes f_2$ precomposed with $t \otimes \id_{c_1c2}$.
\end{df}

	Explicitly, this translates to the following. Let $f_1\colon c_1\rar_{a_1} c_1'$ and $f_2\colon c_2\rar_{a_2} c_2'$, following the above recipe we find:
	\begin{equation}\label{matestens}
	f_{1}\tens{\underleftarrow{\cat{C}}}f_2=\underline{\bar{f_1}\tens{\cat{C}}\bar{f_2}(\id_{a_1}\tens{\cat{C}} \beta_{a_2,c_1} \tens{\cat{C}} \id_{c_2})}.
	\end{equation}
	In string diagrams, this becomes
	\begin{equation}\label{matestensstring}
	\hbox{
		\begin{tikzpicture}[baseline=(current  bounding  box.center)]
		\node (C) at (0,1)[draw,minimum width=20pt,minimum height=10pt,thick]{$\bar{f}_1$};
		\node (D) at (1.5,1)[draw,minimum width=20pt,minimum height=10pt,thick]{$\bar{f}_2$};
		\node (b3) at (0,-0.5){$a_1$};
		\node (b4) at (0.5,-0.5){$a_2$};
		\node (b5) at (1,-0.5){$c_1$};
		\node (b6) at (1.5,-0.5){$c_2$};
		\node (c1) at (0,2){$c_1'$};
		\node (c2) at (1.5,2){$c_2'$};
		\begin{knot}[clip width=4]
		\strand [thick] (c1)
		to (C);
		\strand [thick] (c2)
		to (D);
		\strand [thick] (C)
		to [out=-90,in=90] (b3);
		\strand [thick] (D.-120)
		to [out=-90,in=90] (b4);
		\strand [thick] (C.-55)
		to [out=-90,in=90] (b5);
		\strand [thick] (D)
		to [out=-90,in=90] (b6);
		\end{knot}
		\end{tikzpicture}
	}.
	\end{equation}

\begin{rmk}
	The constructions up to this point only assume that $\cat{C}$ is a tensor category equipped with a central functor $\cat{A}\rar \dcentcat{C}\rar \cat{C}$. 
\end{rmk}

\begin{lem}\label{indeedatensor}
	The categories $\underleftarrow{\cat{C}}$ and $\underrightarrow{\cat{C}}$ are $\cat{A}$-monoidal for the monoidal structure from Definition \ref{inducedatensor}.
\end{lem}
\begin{proof}
	We will only provide a proof for $\underleftarrow{\cat{C}}$, the case of $\underrightarrow{\cat{C}}$ is similar. We need to prove the structure above satisfies the interchange law, i.e. that the proposed $\cat{A}$-monoidal structure is indeed a functor. Checking functoriality boils down to checking that the following diagram commutes
	\begin{center}
		\begin{tikzcd}
		\underleftarrow{\cat{C}}(c_1\bx c_2,c_1'\bx c_2')\otimes\underleftarrow{\cat{C}}(c_1'\bx c_2',c_1''\bx c_2'')\arrow[d,"\tens{\underleftarrow{\cat{C}}}\bxc{A}\tens{\underleftarrow{\cat{C}}}"]\arrow[r,"\circ_{\left(\underleftarrow{\cat{C}}\bxc{A}\underleftarrow{\cat{C}}\right)}"]&\underleftarrow{\cat{C}}(c_1,c_1'')\otimes\underleftarrow{\cat{C}}(c_2,c_2'')\arrow[d,"\mathop{\otimes}\limits_{\underleftarrow{\cat{C}}}"]\\
		\underleftarrow{\cat{C}}(c_1c_2,c_1' c_2')\otimes\underleftarrow{\cat{C}}(c_1' c_2',c_1'' c_2'')\arrow[r,"\circ_{\underleftarrow{\cat{C}}}"]& \underleftarrow{\cat{C}}(c_1 c_2,c_1''c_2'').
		\end{tikzcd}
	\end{center}
	
	We will do this by checking that the precomposition of the two routes in this diagram with
	$$
	f_1\otimes f_2 \otimes f'_1\otimes f_2'\colon  a_1a_2a_1'a_2'\rar \underleftarrow{\cat{C}}(c_1,c_1')\underleftarrow{\cat{C}}(c_2,c_2')\underleftarrow{\cat{C}}(c_1',c_1'')\underleftarrow{\cat{C}}(c_2',c_2'')
	$$
	are the same. This will be the case if and only if their mates are equal. Using Equations \eqref{matescompstring} and \eqref{matestensstring} we see that we need to check
	\begin{align}
	\hbox{
		\begin{tikzpicture}[baseline=(current  bounding  box.center)]
		\node (A) at (-1,2.5)[draw,minimum width=20pt,minimum height=10pt,thick]{$\bar{f}'_1$};
		\node (B) at (1.5,2.5)[draw,minimum width=20pt,minimum height=10pt,thick]{$\bar{f}_2'$};
		\node (C) at (0,1)[draw,minimum width=20pt,minimum height=10pt,thick]{$\bar{f}_1$};
		\node (D) at (1.5,1)[draw,minimum width=20pt,minimum height=10pt,thick]{$\bar{f}_2$};
		\node (b1) at (-1,-0.5){$a_1'$};
		\node (b2) at (-0.5,-0.5){$a_2'$};
		\node (b3) at (0,-0.5){$a_1$};
		\node (b4) at (0.5,-0.5){$a_2$};
		\node (b5) at (1,-0.5){$c_1$};
		\node (b6) at (1.5,-0.5){$c_2$};
		\begin{knot}[clip width=4]
		\strand [thick] (B.-120)
		to [out=-90,in=90] (-0.5,1.4)
		to [out=-90,in=90] (b2);
		\strand [thick] (A)
		to [out=-60,in=90] (C);
		\strand [thick] (C)
		to [out=-90,in=90] (b3);
		\strand [thick] (-1,3)
		to  (A)
		to  [out=-90,in=90] (b1);
		\strand [thick] (D.-120)
		to [out=-90,in=90] (b4);
		\strand [thick] (C.-55)
		to [out=-90,in=90] (b5);
		\strand [thick] (1.5,3)
		to (B)
		to [out=-90,in=90] (D);
		\strand [thick] (D)
		to [out=-90,in=90] (b6);
		\end{knot}
		\end{tikzpicture}
	}
	=
	\hbox{
		\begin{tikzpicture}[baseline=(current  bounding  box.center)]
		\node (A) at (-1,2.5)[draw,minimum width=20pt,minimum height=10pt,thick]{$\bar{f}'_1$};
		\node (B) at (1.5,2.5)[draw,minimum width=20pt,minimum height=10pt,thick]{$\bar{f}_2'$};
		\node (C) at (-0.5,1)[draw,minimum width=20pt,minimum height=10pt,thick]{$\bar{f}_1$};
		\node (D) at (1.5,1)[draw,minimum width=20pt,minimum height=10pt,thick]{$\bar{f}_2$};
		\node (b1) at (-1,-0.5){$a_1'$};
		\node (b2) at (-0.5,-0.5){$a_2'$};
		\node (b3) at (0,-0.5){$a_1$};
		\node (b4) at (0.5,-0.5){$a_2$};
		\node (b5) at (1,-0.5){$c_1$};
		\node (b6) at (1.5,-0.5){$c_2$};
		\begin{knot}[clip width=4]
		\strand [thick] (B.-120)
		to [out=-90,in=90] (b2);
		\strand [thick] (A)
		to [out=-60,in=90] (C);
		\strand [thick] (C.-120)
		to [out=-90,in=90] (b3);
		\strand [thick] (-1,3)
		to  (A)
		to  [out=-90,in=90] (b1);
		\strand [thick] (D.-120)
		to [out=-90,in=90] (b4);
		\strand [thick] (C.-55)
		to [out=-90,in=90](b5);
		\strand [thick] (1.5,3)
		to (B)
		to [out=-90,in=90] (D);
		\strand [thick] (D)
		to [out=-90,in=90] (b6);
		\end{knot}
		\end{tikzpicture}
	},\end{align}
	and this equation holds by naturality of the braiding in $\cat{C}$.

	The associators in $\cat{C}$ will descend to morphisms in $\underleftarrow{\cat{C}}$ and still satisfy the pentagon equations. We have to convince ourselves that these morphisms define a natural isomorphism, with respect to the extra morphisms in the enriched hom-objects $\underleftarrow{\cat{C}}(c,c')$ for $c,c'\in\cat{C}$. Looking at Definition \ref{enrich}, all these extra morphisms are just morphisms $ac\rar c'$ for some $a\in\cat{A}$. Using the pentagon equations on these morphisms, this means the associators from $\cat{C}$ will also be natural for these extra morphisms. 
\end{proof}

\subsubsection{A second $\cat{A}$-monoidal structure}
Since we made a choice to use $\beta$ rather than $\beta^{-1}$ in Definition \ref{inducedatensor}, we also have:

\begin{df}\label{othertensor}
	We define $-\bar{\otimes} -\colon \underleftarrow{\cat{C}}\bxc{A}\underleftarrow{\cat{C}}\rar \underleftarrow{\cat{C}}$, by taking it to be $-\otimes_\cat{C} -$ on objects and on morphisms the map from Definition \ref{inducedatensor} with $\beta_{a_2,c_1}$ replaced by $\beta^{-1}_{c_1,a_2}$.
\end{df}

The proof that this indeed specifies an $\cat{A}$-monoidal structure is analogous to the proof of Lemma \ref{indeedatensor}. In string diagrams for the mates of $f_1\colon c_1\rar_{a_1} c_1'$ and $f_2\colon c_2\rar_{a_2} c_2'$ this monoidal structure gives:
\begin{equation}\label{matesothertensstring}
\hbox{
	\begin{tikzpicture}[baseline=(current  bounding  box.center)]
	\node (C) at (0,1)[draw,minimum width=20pt,minimum height=10pt,thick]{$\bar{f}_1$};
	\node (D) at (1.5,1)[draw,minimum width=20pt,minimum height=10pt,thick]{$\bar{f}_2$};
	\node (b3) at (0,-0.5){$a_1$};
	\node (b4) at (0.5,-0.5){$a_2$};
	\node (b5) at (1,-0.5){$c_1$};
	\node (b6) at (1.5,-0.5){$c_2$};
	\node (c1) at (0,2){$c_1'$};
	\node (c2) at (1.5,2){$c_2'$};
	\begin{knot}[clip width=4]
	\strand [thick] (c1)
	to (C);
	\strand [thick] (c2)
	to (D);
	\strand [thick] (C)
	to [out=-90,in=90] (b3);
	\strand [thick] (D.-120)
	to [out=-90,in=90] (b4);
	\strand [thick] (C.-55)
	to [out=-90,in=90] (b5);
	\strand [thick] (D)
	to [out=-90,in=90] (b6);
	\flipcrossings{1}
	\end{knot}
	\end{tikzpicture}
}.
\end{equation}

This second monoidal structure will be useful below when studying how the braiding on $\cat{C}$ behaves on $\underleftarrow{\cat{C}}$.

\subsection{Braiding for the associated $\cat{A}$-enriched category}
In the previous section, we only used the half-twists $\beta_{a,c}$ for $a\in \cat{A}$ and $c\in\cat{C}$, and the braiding in $\cat{A}$, this amounts to the data of a central functor $\cat{A}\rar \dcentcat{C}\rar\cat{C}$. From here onward, we will use that $\cat{C }$ is itself braided, and that $\cat{A}$ is a symmetric subcategory of $\cat{C}$.

\subsubsection{A problem with the braiding}
Naively, one might expect $\underleftarrow{\cat{C}}$ to be braided if $\cat{C}$ is, with braiding induced by the braiding in $\cat{C}$. We now pause to show that $\underleftarrow{\cat{C}}$ is not a braided $\cat{A}$-monoidal category. The failure of the braiding on $\cat{C}$ to induce a braiding on $\underleftarrow{\cat{C}}$ will motivate the next step in our construction, where we further enrich $\underleftarrow{\cat{C}}$ to a category enriched over the Drinfeld centre of $\cat{A}$.

When attempting to lift the braiding of $\cat{C}$ to a braiding on $\underleftarrow{\cat{C}}$, one encounters the following problem: the braiding will no longer be natural with respect to the additional morphisms. We will show that the naturality diagram
	\begin{center}
		\begin{tikzcd}[column sep=large]
		c_1c_2\arrow[r,"\beta_{c_1,c_2}"]\arrow[d,"f_1\otimes f_2", "a_1a_2"' very near end]& c_2c_1 \arrow[d,"f_2\otimes f_1", "a_2a_1"' very near end]\\
		c_1'c_2'\arrow[r,"\beta_{c_1',c_2'}"']& c_2'c_1'
		\end{tikzcd}
	\end{center}
	fails to commute in general. Its failure to commute can be seen as follows. In terms of mates, the naturality diagram becomes the outside of:
	\begin{center}
		\begin{tikzcd}[column sep=large]
		a_1a_2c_1c_2	\arrow[r,"\beta_{a_1,a_2}\otimes\beta_{c_1,c_2}"]		\arrow[d,"\beta_{a_2,c_1}"]				& a_2a_1c_2c_1 \arrow[d,"\beta_{a_1,c_2}"]\\
		a_1c_1a_2c_2	\arrow[r,"\beta_{a_1c_1,a_2c_2}"]						\arrow[d,"\bar{f}_1\otimes \bar{f}_2"]	& a_2c_2a_1c_1 \arrow[d,"\bar{f}_2\otimes \bar{f}_1"]\\
		c_1'c_2'		\arrow[r,"\beta_{c_1',c_2'}"']																	& c_2'c_1'
		\end{tikzcd}.
	\end{center}
	Here the braiding $\beta_{a_1,a_2}$ in the top row comes from the switch map for the $\cat{A}$-product that was implicit in the previous diagram. The map in the middle will help us understand the failure of commutativity. Note that, by naturality of the braiding in $\cat{C}$, the lower square of the diagram does commute. It therefore suffices to consider the top square, in string diagrams the top and bottom routes read
	\begin{equation}\label{failuremonodromy}
	\hbox{
		\begin{tikzpicture}[baseline=(current  bounding  box.center)]
		\node (a1) at (-0.5,-0.5){$a_1$};
		\node (a2) at (0,-0.5){$a_2$};
		\node (a3) at (0.5,-0.5){$c_1$};
		\node (a4) at (1,-0.5){$c_2$};
		
		\node (b1) at (-0.5,2){$a_2$};
		\node (b2) at (0,2){$c_2$};
		\node (b3) at (0.5,2){$a_1$};
		\node (b4) at (1,2){$c_1$};
		
		\begin{knot}[clip width=4]
		\strand [thick] (b4)
		to [out=-90,in=90] (1,1)
		to [out=-90,in=90] (a3);
		\strand [thick] (b3)
		to [out=-90,in=90] (0,1)
		to [out=-90,in=90] (a1);
		\strand [thick] (b2)
		to [out=-90,in=90] (0.5,1)
		to [out=-90,in=90] (a4);
		\strand [thick] (b1)
		to [out=-90,in=90] (-.5,1)
		to [out=-90,in=90] (a2);
		
		\end{knot}
		\end{tikzpicture}
	}
	\tn{ and }
	\hbox{
		\begin{tikzpicture}[baseline=(current  bounding  box.center)]
		\node (a1) at (-0.5,-0.5){$a_1$};
		\node (a2) at (0,-0.5){$a_2$};
		\node (a3) at (0.5,-0.5){$c_1$};
		\node (a4) at (1,-0.5){$c_2$};

		\node (b1) at (-0.5,2){$a_2$};
		\node (b2) at (0,2){$c_2$};
		\node (b3) at (0.5,2){$a_1$};
		\node (b4) at (1,2){$c_1$};
		
		\begin{knot}[clip width=4]
		\strand [thick] (b4)
		to [out=-90,in=90](0,0.5)
		to [out=-90,in=90] (a3);
		\strand [thick] (b3)
		to [out=-90,in=90] (-0.5,0.5)
		to [out=-90,in=90] (a1);
		\strand [thick] (b2)
		to [out=-90,in=90] (1,0.5)
		to [out=-90,in=90] (a4);
		\strand [thick] (b1)
		to [out=-90,in=90] (0.5,0.5)
		to [out=-90,in=90] (a2);
		\flipcrossings{3}
		\end{knot}
		\end{tikzpicture}
	},
	\end{equation}
	respectively. We see that these diagrams differ from each other by a precomposition with the \emph{braiding monodromy} $\beta_{c_1,a_2}\beta_{a_2,c_1}$ between $a_2$ and $c_1$. 

\subsubsection{Braiding between the two monoidal structures}
In this section, we show how the braiding on $\cat{C}$ can be used instead to relate the two different monoidal structures $\otimes$ (Definition \ref{inducedatensor}) and $\bar{\otimes}$ (Definition \ref{othertensor}) on $\underleftarrow{\cat{C}}$. We will do this in two ways, the first is along a functor that encodes the braiding monodromy between objects of $\cat{A}$ and objects of $\cat{C}$, the second is showing that the braiding gives a natural isomorphism between $\otimes^{\beta}$ and $\otimes \circ S$, where $S$ is the swap functor for $\cattens{\cat{A}}$. Later, we will use these results to show we can enrich $\cat{C}$ to a braided object in some sense.

\begin{df}\label{betasquaredinv}
	The \emph{inverse monodromy functor} $\beta^{-2}\colon \underleftarrow{\cat{C}}\bxc{A}\underleftarrow{\cat{C}}\rar \underleftarrow{\cat{C}}\bxc{A}\underleftarrow{\cat{C}}$ is defined as follows. The functor $\beta^{-2}$ is the identity on objects. On morphisms, 
	we take mates for $f\colon c_1\bx c_2\rar_a c_1'\bx c_2'$ factored into the tensor product between $f_1\colon c_1 \rar_{a_1} c_1'$ and $f_2\colon c_2\rar_{a_2} c_2'$, with $a_1a_2=a$, and assign
	\begin{equation}\label{inversemonodromymorph}
	\hbox{
		\begin{tikzpicture}[baseline=(current  bounding  box.center)]
		\node (a1) at (-1,-2)[minimum width=20pt,minimum height=10pt]{$a_1$};
		\node (c1) at (0.5,-2)[minimum width=20pt,minimum height=10pt]{$c_1$};
		\node (a2) at (-0.5,-2)[minimum width=20pt,minimum height=10pt]{$a_2$};
		\node (c2) at (1,-2)[minimum width=20pt,minimum height=10pt]{$c_2$};
		
		\node (f1) at (-0.75,1.5)[draw,minimum width=20pt,minimum height=10pt,thick]{$\bar{f}_1$};
		\node (f2) at (0.75,1.5)[draw,minimum width=20pt,minimum height=10pt,thick]{$\bar{f}_2$};
		
		\node (cp1) at (-0.75,2.5){$c_1'$};
		\node (cp2) at (0.75,2.5){$c'_2$};
		
		\begin{knot}[clip width=4]
		\strand [thick] (a1)
		to [out=90,in=-110] (f1);
		\strand [thick] (c1)
		to [out=90,in=-75] (f1);
		\strand [thick] (c2)
		to [out=90,in=-75] (f2);
		\strand [thick] (f1)
		to [out=90,in=-90] (cp1);
		\strand [thick] (f2)
		to [out=90,in=-90] (cp2);
		\end{knot}
		\draw [thick] (a2)
		to [out=90,in=-110] (f2);
		\end{tikzpicture}
	}
	\mapsto
		\hbox{
		\begin{tikzpicture}[baseline=(current  bounding  box.center)]
		\node (a1) at (-1,-2)[minimum width=20pt,minimum height=10pt]{$a_1$};
		\node (a2) at (-0.5,-2)[minimum width=20pt,minimum height=10pt]{$a_2$};
		\node (c1) at (0.5,-2)[minimum width=20pt,minimum height=10pt]{$c_1$};
		\node (c2) at (1,-2)[minimum width=20pt,minimum height=10pt]{$c_2$};
		\coordinate (upper) at (-0.5,1);
		\coordinate (firstcross) at (0,0.5);
		\coordinate (secondcross) at (-0.3,-0.3);
		\coordinate (unit) at (0,-1.5);
		\coordinate (backup) at (0.5,0);
		
		\node (f1) at (-0.75,2)[draw,minimum width=20pt,minimum height=10pt,thick]{$\bar{f}_1$};
		\node (f2) at (0.75,2)[draw,minimum width=20pt,minimum height=10pt,thick]{$\bar{f}_2$};
		
		\node (cp1) at (-0.75,3){$c_1'$};
		\node (cp2) at (0.75,3){$c'_2$};
		
		\begin{knot}[clip width=4]
		\strand [thick] (a1)
		to [out=90,in=-110] (f1);
		\strand [thick] (c1)
		to [out=90,in=-75] (f1);
		\strand [thick] (c2)
		to [out=90,in=-75] (f2);
		\strand [thick] (f1)
		to [out=90,in=-90] (cp1);
		\strand [thick] (f2)
		to [out=90,in=-90] (cp2);
		\strand [thick] (a2)
		to [out=90,in=-180](upper)
		to [out=0,in=90](firstcross)
		to [out=-90,in=90] (secondcross)
		to [out=-90,in=-180] (unit);
		\flipcrossings{2}
		\end{knot}
		\draw[thick] (unit)
		to [out=0,in=-90](backup)
		to [out=90,in=-110] (f2);
		\end{tikzpicture}
	}.
	\end{equation}
	We remind the reader of the convention discussed around Equation \eqref{problemdiagrams}, and emphasise that the double braiding in this diagram really is a double braiding, whereas the unresolved crossings indicate a crossing used to bring all objects of $\cat{A}$ to the left. To justify that this assignment really defines a morphism $\beta^{-2}(f)\colon a\rar \underleftarrow{\cat{C}}\cattens{A}\underleftarrow{\cat{C}}(c_1\boxtimes c_2, c_1'\boxtimes c_2')$ we compare with Equation \eqref{morphinaprod}. We can interpret the right hand side of Equation \eqref{inversemonodromymorph} as the mate for the tensor product of a morphism $a_1a_2a_2^*\rar \underleftarrow{\cat{C}}(c_1,c_1')$ and $f_2\colon c_2\rar_{a_2}c_2'$, precomposed with the with trivalent vertex $a_1a_2\rar a_1a_2a_2^*a_2$ coming from the coevaluation for $a_2^*$. 	
\end{df}

\begin{lem}
	The assignment $\beta^{-2}$ is an autofunctor of $\underleftarrow{\cat{C}}\bxc{A}\underleftarrow{\cat{C}}$.
\end{lem}

\begin{proof}
	If $\beta^{-2}$ is indeed a functor, then it is clearly invertible with inverse given by using the opposite crossings in Equation \eqref{inversemonodromymorph}. So we need to check that $\beta^{-2}$ preserves composition. That is, we need to check that the following diagram commutes:
	\begin{center}
		\begin{tikzcd}
		\underleftarrow{\cat{C}}(c_1',c_1'')\underleftarrow{\cat{C}}(c_2',c_2'')\underleftarrow{\cat{C}}(c_1,c_1')\underleftarrow{\cat{C}}(c_2,c_2')	\arrow[d,"\beta^{-2}\otimes \beta^{-2}"]	\arrow[r,"\circ_{\underleftarrow{\cat{C}}\cattens{a}\underleftarrow{\cat{C}}}"]	& \underleftarrow{\cat{C}}(c_1,c_1'')\underleftarrow{\cat{C}}(c_2,c_2'')	\arrow[d,"\beta^{-2}"]	\\
		\underleftarrow{\cat{C}}(c_1',c_1'')\underleftarrow{\cat{C}}(c_2',c_2'')\underleftarrow{\cat{C}}(c_1,c_1')\underleftarrow{\cat{C}}(c_2,c_2')						\arrow[r,"\circ_{\underleftarrow{\cat{C}}\cattens{a}\underleftarrow{\cat{C}}}"]	& \underleftarrow{\cat{C}}(c_1,c_1'')\underleftarrow{\cat{C}}(c_2,c_2'').
		\end{tikzcd}
	\end{center}
		Recall that composition in terms of mates is given by Equation \eqref{matescompstring}. In terms of mates for $f_1\colon c_1\rar_{a_1} c_1'$, $f_1'\colon c_1'\rar_{a_1'} c_1''$, $f_2\colon c_2\rar_{a_2} c_2'$ and $f_2'\colon c_2'\rar_{a_2} c_2''$, the top route becomes
	
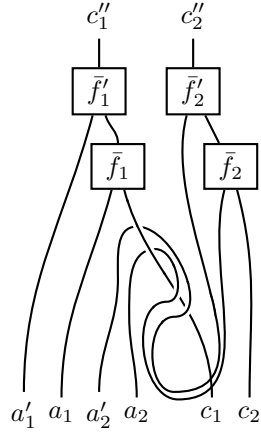
\begin{equation*}
	\hbox{
		\begin{tikzpicture}[baseline=(current  bounding  box.center)]
		\node (a1p) at (-2.5,0.7){$a_1'$};
		\node (a1) at (-2,0.7){$a_1$};
		\node (c1) at (0,0.7){$c_1$};
		\node (a2p) at (-1.5,0.7){$a_2'$};
		\node (a2) at (-1,0.7){$a_2$};
		\node (c2) at (0.5,0.7){$c_2$}; 
		
		\node (f1) at (-1.25,4)[draw,minimum width=20pt,minimum height=10pt,thick]{$\bar{f}_1$};
		\node (f1p) at (-1.5,5)[draw,minimum width=20pt,minimum height=10pt,thick]{$\bar{f}_1'$};
		
		\node (f2) at (0.25,4)[draw,minimum width=20pt,minimum height=10pt,thick]{$\bar{f}_2$};
		\node (f2p) at (-0.25,5)[draw,minimum width=20pt,minimum height=10pt,thick]{$\bar{f}_2'$};
		
		\node (c1pp) at (-1.5,6){$c_1''$};
		\node (c2pp) at (-0.25,6){$c_2''$};
		
		\begin{knot}[clip width=4]
		\strand [thick] (a1p)
		to [out=90,in=-105] (f1p);
		\strand [thick] (a1)
		to [out=90,in=-105] (f1);
		\strand [thick] (c1)
		to [out = 90, in =-75] (f1);
		\strand [thick] (f1)
		to [out=90,in=-75] (f1p);
		\strand [thick] (f1p)
		to [out=90,in=-90] (c1pp);
		\strand [thick] (a2p) to [out=90,in=-90] (-1.25,2.5) to [out=90,in=-180] (-1,3.2) to [out=0,in=0] (-0.5,2) to [out=180,in=180] (-0.5,1) to [out=0,in=-105] (f2p);
		\strand [thick] (a2) to [out=90,in=-90] (-1.1,2) to [out=90,in=180] (-0.75,2.9) to [out=0,in=0] (-0.6,2.2) to [out=180,in=180] (-0.45,0.9)	to [out=0,in=-105] (f2);
		\strand [thick] (c2)
		to [out=90,in=-75] (f2);
		\strand [thick] (f2)
		to (f2p);
		\strand [thick] (f2p)
		to [out=90,in=-90] (c2pp);
		\flipcrossings{2,4}
		\end{knot}
		\end{tikzpicture}
	},
	\end{equation*}
	while the bottom route becomes, first applying $\beta^{-2}$ then the composition
	$$
	
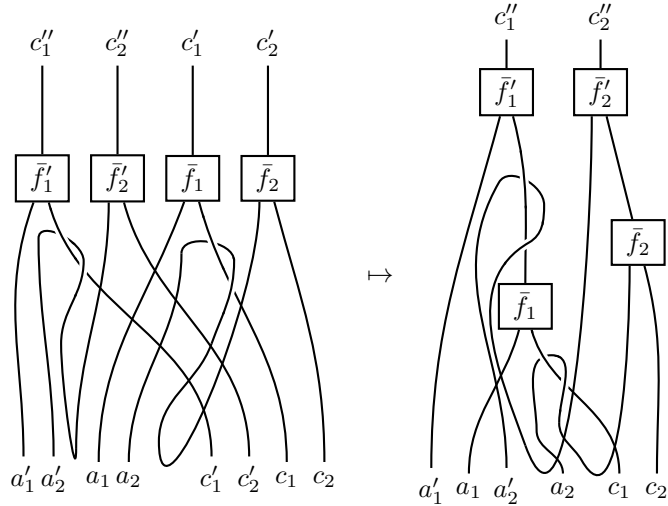
\hbox{
	\begin{tikzpicture}[baseline=(current  bounding  box.center)]
	\node (a1) at (-1,-2)[minimum width=20pt,minimum height=10pt]{$a_1'$};
	\node (a2) at (-0.6,-2)[minimum width=20pt,minimum height=10pt]{$a_2'$};
	\node (c1) at (1.5,-2)[minimum width=20pt,minimum height=10pt]{$c_1'$};
	\node (c2) at (2,-2)[minimum width=20pt,minimum height=10pt]{$c_2'$};
	\coordinate (upper) at (-0.8,1.2);
	\coordinate (firstcross) at (-0.2,0.9);
	\coordinate (secondcross) at (-0.5,-0.2);
	\coordinate (unit) at (-0.3,-1.6);
	\coordinate (backup) at (0.5,0);
	
	\node (f1) at (-0.75,2)[draw,minimum width=20pt,minimum height=10pt,thick]{$\bar{f}_1'$};
	\node (f2) at (0.25,2)[draw,minimum width=20pt,minimum height=10pt,thick]{$\bar{f}_2'$};
	
	\node (cp1) at (-0.75,3.8){$c_1''$};
	\node (cp2) at (0.25,3.8){$c''_2$};
	
	\node (sa1) at (0,-2)[minimum width=20pt,minimum height=10pt]{$a_1$};
	\node (sa2) at (0.4,-2)[minimum width=20pt,minimum height=10pt]{$a_2$};
	\node (sc1) at (2.5,-2)[minimum width=20pt,minimum height=10pt]{$c_1$};
	\node (sc2) at (3,-2)[minimum width=20pt,minimum height=10pt]{$c_2$};
	\coordinate (supper) at (1.1,1);
	\coordinate (sfirstcross) at (1.8,0.9);
	\coordinate (ssecondcross) at (1.5,-0.2);
	\coordinate (sunit) at (0.8,-1.5);
	\coordinate (sbackup) at (0.5,0);
	
	\node (sf1) at (1.25,2)[draw,minimum width=20pt,minimum height=10pt,thick]{$\bar{f}_1$};
	\node (sf2) at (2.25,2)[draw,minimum width=20pt,minimum height=10pt,thick]{$\bar{f}_2$};
	
	\node (scp1) at (1.25,3.8){$c_1'$};
	\node (scp2) at (2.25,3.8){$c'_2$};
	
	\begin{knot}[clip width=4]
	\strand [thick] (a1)
	to [out=90,in=-110] (f1);
	\strand [thick] (c1)
	to [out=90,in=-75] (f1);
	\strand [thick] (c2)
	to [out=90,in=-75] (f2);
	\strand [thick] (f1)
	to [out=90,in=-90] (cp1);
	\strand [thick] (f2)
	to [out=90,in=-90] (cp2);
	\strand [thick] (a2)
	to [out=90,in=-90](upper)
	to [out=90,in=90](firstcross)
	to [out=-90,in=90] (secondcross)
	to [out=-90,in=-90] (unit);
	\strand [thick] (sc1)
	to [out=90,in=-75] (sf1);
	\strand [thick] (sc2)
	to [out=90,in=-75] (sf2);
	\strand [thick] (sf1)
	to [out=90,in=-90] (scp1);
	\strand [thick] (sf2)
	to [out=90,in=-90] (scp2);
	\strand [thick] (supper)
	to [out=90,in=90](sfirstcross)
	to [out=-90,in=70] (ssecondcross);
	\flipcrossings{2,4}
	\end{knot}
	\draw[thick] (unit)
	to [out=90,in=-90] (f2.-110);
	\draw [thick] (sa1)
	to [out=90,in=-110] (sf1);
	\draw [thick] (sa2) to [out=90,in=-90] (supper);
	\draw [thick] (ssecondcross)
	to [out=-110,in=90] (sunit) to [out=-90,in=-90] (sf2.-110);
	\end{tikzpicture}
	}
	\mapsto
		\hbox{
		\begin{tikzpicture}[baseline=(current  bounding  box.center)]
		\node (a1p) at (-2.5,0.7){$a_1'$};
		\node (a1) at (-2,0.7){$a_1$};
		\node (c1) at (0,0.7){$c_1$};
		\node (a2p) at (-1.5,0.7){$a_2'$};
		\node (a2) at (-0.75,0.7){$a_2$};
		\node (c2) at (0.5,0.7){$c_2$}; 
		
		\node (f1) at (-1.25,3.15)[draw,minimum width=20pt,minimum height=10pt,thick]{$\bar{f}_1$};
		\node (f1p) at (-1.5,6)[draw,minimum width=20pt,minimum height=10pt,thick]{$\bar{f}_1'$};
		
		\node (f2) at (0.25,4)[draw,minimum width=20pt,minimum height=10pt,thick]{$\bar{f}_2$};
		\node (f2p) at (-0.25,6)[draw,minimum width=20pt,minimum height=10pt,thick]{$\bar{f}_2'$};
		
		\node (c1pp) at (-1.5,7){$c_1''$};
		\node (c2pp) at (-0.25,7){$c_2''$};
		
		\begin{knot}[clip width=4]
			\strand [thick] (a1p)
			to [out=90,in=-105] (f1p);
			\strand [thick] (c1)
			to [out = 90, in =-75] (f1);
			\strand [thick] (f1)
			to [out=90,in=-75] (f1p);
			\strand [thick] (f1p)
			to [out=90,in=-90] (c1pp);
			\strand [thick] (a2p) to [out=90,in=-120] (-1.6,4.8) to [out=60,in=90] (-1,4.5) to [out=-90,in=30] (-1.2,4) to [out=-150,in=110] (-1.25,1.3);
			\strand [thick] (a2) to [out=90,in=-90] (-1.1,1.5) to [out=90,in=180] (-0.9,2.5) to [out=0,in=90] (-0.8,1.6);
			\strand [thick] (c2)
			to [out=90,in=-75] (f2);
			\strand [thick] (f2)
			to (f2p);
			\strand [thick] (f2p)
			to [out=90,in=-90] (c2pp);
			\flipcrossings{2,4}
		\end{knot}
			\draw [thick] (a1) to [out=90,in=-105] (f1);
			\draw [thick] (-0.8,1.6) to [out=-90,in=120] (-0.5,1.1) to [out=-60,in=-90] (f2.-110);
			\draw [thick] (-1.25,1.3) to [out=-70,in=-90] (f2p.-110);
	\end{tikzpicture}
	}.
	$$
	Using the naturality of the braiding, we see that the string diagrams corresponding to the top and bottom routes are indeed equal. 
\end{proof}

The inverse monodromy functor can be used to obtain the two monoidal structures on $\underleftarrow{\cat{C}}$ from each other:

\begin{lem}\label{otimesbetaisotimesafterbeta}
	The functor $-\bar{\otimes} -\colon\underleftarrow{\cat{C}}\cattens{\cat{A}}\underleftarrow{\cat{C}}\rar \underleftarrow{\cat{C}}$ is equal to the functor obtained by precomposing $-\otimes_{\underleftarrow{\underline{\cat{C}}}}-$ with $\beta^{-2}$.
\end{lem}
\begin{proof}
	The functors agree on objects by definition, so we only need to check the functors agree on morphisms. Let $f_1\colon c_1\rar_{a_1} c_1'$ and $f_2\colon c_2\rar_{a_2} c_2'$ be morphisms in $\underleftarrow{\cat{C}}$. The image of their mates under $-\bar{\otimes} -$ is shown in Equation \eqref{matesothertensstring}. Applying the composite of $\beta^{-2}$ and $-\otimes -$ to these mates is given, in string diagrams, by
	\begin{equation}
		\hbox{
		\begin{tikzpicture}[baseline=(current  bounding  box.center)]
		\node (a1) at (-1,-0.5)[minimum width=20pt,minimum height=10pt]{$a_1$};
		\node (a2) at (-0.5,-0.5)[minimum width=20pt,minimum height=10pt]{$a_2$};
		\node (c1) at (0.5,-0.5)[minimum width=20pt,minimum height=10pt]{$c_1$};
		\node (c2) at (1,-0.5)[minimum width=20pt,minimum height=10pt]{$c_2$};
		\coordinate (upper) at (-0.2,1.2);
		\coordinate (firstcross) at (-0.1,0.5);
		\coordinate (unit) at (0,-0.1);
		\coordinate (backup) at (0.5,0);
		
		\node (f1) at (-0.75,2)[draw,minimum width=20pt,minimum height=10pt,thick]{$\bar{f}_1$};
		\node (f2) at (0.75,2)[draw,minimum width=20pt,minimum height=10pt,thick]{$\bar{f}_2$};
		
		\node (cp1) at (-0.75,3){$c_1'$};
		\node (cp2) at (0.75,3){$c'_2$};
		
		\begin{knot}[clip width=4]
		\strand [thick] (a1)
		to [out=90,in=-110] (f1);
		\strand [thick] (c1)
		to [out=90,in=-75] (f1);
		\strand [thick] (c2)
		to [out=90,in=-75] (f2);
		\strand [thick] (f1)
		to [out=90,in=-90] (cp1);
		\strand [thick] (f2)
		to [out=90,in=-90] (cp2);
		\strand [thick] (a2)
		to [out=90,in=-180](upper)
		to [out=0,in=90](firstcross)
		to [out=-90,in=-180] (unit) to [out=0,in=-110] (f2);
		\flipcrossings{2,3}
		\end{knot}
		\end{tikzpicture}
	}
	,
	\end{equation}
	which is indeed equal to Equation \eqref{matesothertensstring}.
\end{proof}

We will now show that the braiding is a natural transformation between these two monoidal structures on $\underleftarrow{\cat{C}}$, after we compose one with the switch map.

\begin{lem}\label{braidingbetaotimes}
	The braiding in $\cat{C}$ induces a natural isomorphism between the functors $-\bar{\otimes} -\colon \underleftarrow{\cat{C}}\bxc{A}\underleftarrow{\cat{C}}\rar\underleftarrow{\cat{C}}$ and the composite of $-\otimes -\colon \underleftarrow{\cat{C}}\bxc{A}\underleftarrow{\cat{C}}\rar \underleftarrow{\cat{C}}$ with the switch map for the $\cat{A}$-product. This isomorphism satisfies the hexagon equations.
\end{lem}
\begin{proof}
	We want to show the diagram
	\begin{center}
		\begin{tikzcd}[column sep=large]
		c_1c_2\arrow[r,"\beta_{c_1,c_2}"]\arrow[d,"f_1\bar{\otimes} f_2", "a_1a_2"' very near end]& c_2c_1 \arrow[d,"f_2\otimes f_1", "a_2a_1"' very near end]\\
		c_1'c_2'\arrow[r,"\beta_{c_1',c_2'}"']& c_2'c_1'
		\end{tikzcd}
	\end{center}
	commutes for all $f_1\colon c_1\rar_{a_1} c_1'$ and $f_2\colon c_2\rar_{a_2} c_2'$. In terms of the mates, this diagram becomes
	\begin{center}
		\begin{tikzcd}[column sep=large]
		a_1a_2c_1c_2	\arrow[r,"\beta_{a_1,a_2}\otimes\beta_{c_1,c_2}"]		\arrow[d,"\beta_{a_2,c_1}^{-1}"]				& a_2a_1c_2c_1 \arrow[d,"\beta_{c_2,a_1}"]\\
		a_1c_1a_2c_2															\arrow[d,"\bar{f}_1\otimes \bar{f}_2"]	& a_2c_2a_1c_1 \arrow[d,"\bar{f}_2\otimes \bar{f}_1"]\\
		c_1'c_2'		\arrow[r,"\beta_{c_1',c_2'}"']																	& c_2'c_1'
		\end{tikzcd}.
	\end{center}
	Writing this in terms of string diagrams
	
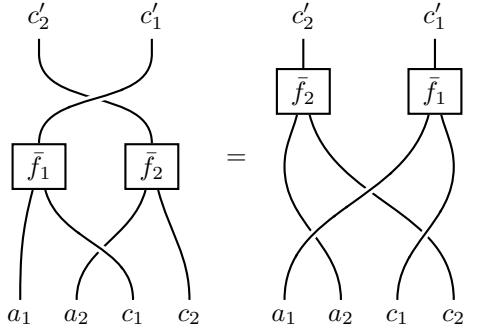
\begin{equation*}
	\hbox{
		\begin{tikzpicture}[baseline=(current  bounding  box.center)]
		
		\node (a1) at (-1.5,-2){$a_1$};
		\node (a2) at (-0.75,-2){$a_2$};
		\node (c1) at (0,-2){$c_1$};
		\node (c2) at (0.75,-2){$c_2$}; 
		
		\node (f1) at (-1.25,0)[draw,minimum width=20pt,minimum height=10pt,thick]{$\bar{f}_1$};
		\node (f2) at (0.25,0)[draw,minimum width=20pt,minimum height=10pt,thick]{$\bar{f}_2$};
		
		\node (c1p) at (0.25,2){$c_1'$};
		\node (c2p) at (-1.25,2){$c_2'$};
		
		\begin{knot}[clip width=4]
		\strand [thick] (a1)
		to [out=90,in=-105] (f1); 
		\strand [thick] (a2)
		to [out=90,in=-105] (f2);
		\strand [thick] (c1)
		to [out=90,in=-75] (f1);
		\strand [thick] (c2)
		to [out=90,in=-75] (f2);
		\strand [thick] (f1)
		to [out=90,in=-90] (0.25,1.5)
		to [out=90,in=-90] (c1p);
		\strand [thick] (f2)
		to [out=90,in=-90] (-1.25,1.5)
		to [out=90,in=-90] (c2p);
		\flipcrossings{1}
		\end{knot}
		\end{tikzpicture}
	}=
	\hbox{
		\begin{tikzpicture}[baseline=(current  bounding  box.center)]
		
		\node (a1) at (-1.5,-2){$a_1$};
		\node (a2) at (-0.75,-2){$a_2$};
		\node (c1) at (0,-2){$c_1$};
		\node (c2) at (0.75,-2){$c_2$}; 
		
		\node (f1) at (0.5,1)[draw,minimum width=20pt,minimum height=10pt,thick]{$\bar{f}_1$};
		\node (f2) at (-1.25,1)[draw,minimum width=20pt,minimum height=10pt,thick]{$\bar{f}_2$};
		
		\node (c1p) at (0.5,2){$c_1'$};
		\node (c2p) at (-1.25,2){$c_2'$};
		
		\begin{knot}[clip width=4]
		\strand [thick] (a1)
		to [out=90,in=-105] (f1); 
		\strand [thick] (a2)
		to [out=90,in=-90] (-1.5,0)
		to [out=90,in=-105] (f2);
		\strand [thick] (c1)
		to [out=90,in=-90] (0.75,0)
		to [out=90,in=-75] (f1);
		\strand [thick] (c2)
		to [out=90,in=-75] (f2);
		\strand [thick] (f1)
		to [out=90,in=-90] (c1p);
		\strand [thick] (f2)
		to [out=90,in=-90] (c2p);
		\end{knot}
		\end{tikzpicture}
	}.
	\end{equation*}
	The hexagon equations follow from the hexagon equations for the braiding in $\cat{C}$.
\end{proof}

\subsection{Towards $\dcentcat{A}$-crossed braided categories}
Recall that our goal is to give a tensor product of braided tensor categories containing $\cat{A}$. To achieve this, we want a construction that produces from such categories enriched categories that are braided, and will take a product of these that takes braided categories to braided categories. To ensure that our construction gives a braided object, we will enrich our category $\underleftarrow{\cat{C}}$ further to a $\dcentcat{A}$-enriched category $\underleftarrow{\underline{\cat{C}}}$, where we take care to encode the braiding monodromy from Equation \eqref{failuremonodromy} into the half-braidings we pick on our hom-objects. As the swap map for the convolution product of $\dcentcat{A}$-enriched categories uses these half-braidings, we can use this to cancel the failure of the naturality of the braiding.

\subsubsection{The $\dcentcat{A}_s$-enrichment}\label{dcentcataenrichment}

We will now show that the $\cat{A}$-enrichment from the previous sections can be lifted to an enrichment over $\dcentcat{A}_s=(\dcentcat{A},\otimes_s)$. We will need to define an enriched hom-functor with values in $\dcentcat{A}$, the composition, and the identity morphisms. We will denote the resulting $\dcentcat{A}_s$-enriched category by $\underleftarrow{\underline{\cat{C}}}$.

The first step towards enriching $\cat{C}$ over $\dcentcat{A}_s$ is: 

\begin{df}\label{halfbraidequip}
	Let $\cat{C}$ be a braided tensor category containing $\cat{A}$ as a braided subcategory. The \emph{$\dcentcat{A}_s$-enriched hom-object $\underleftarrow{\underline{\cat{C}}}(c,c')$} between $c, c' \in \cat{C}$ is defined as follows. We set
	$$
	\underleftarrow{\underline{\cat{C}}}(c,c')=(\underleftarrow{\cat{C}}(c,c'),\mathfrak{b}),
	$$
	where the half-braiding $\mathfrak{b}$ is defined by:
	\begin{equation}\label{halfbraidingonhomobjects}
	a\underleftarrow{\cat{C}}(c,c')\xrightarrow{\cong}\underleftarrow{\cat{C}}(c,ac')\xrightarrow{(\beta^{-1}_{a,c'}\beta_{c',a}^{-1})_*} \underleftarrow{\cat{C}}(c,ac')\xrightarrow{\cong}a\underleftarrow{\cat{C}}(c,c')\xrightarrow{s}\underleftarrow{\cat{C}}(c,c')a.
	\end{equation}
	Here $s$ denotes the symmetry in $\cat{A}$. 
\end{df}

For this definition to make sense, we need to show that the half-braiding $\mathfrak{b}$ is monoidal (cf. Equation \eqref{halfbraidreq}):

\begin{lem}
	The half-braiding $ -\otimes_{\cat{A}}\underleftarrow{\cat{C}}(c,c')\Rightarrow\mathfrak{b}\colon \underleftarrow{\cat{C}}(c,c')\otimes_{\cat{A}}- $ is a monoidal natural isomorphism between functors $\cat{A}\rar \cat{A}$. 
\end{lem}
\begin{proof}
	Using that $a=\underleftarrow{\cat{C}}(\mathbb{I},a)$, we can unpack the half-braiding from Equation \eqref{halfbraidingonhomobjects} in terms of the mates for $\underline{\id_a}\colon \mathbb{I}\rar_{a} a$ and $f_2\colon c\rar_{a'} c'$ as
	\begin{equation}\label{braidingonmates}
	\hbox{
		\begin{tikzpicture}[baseline=(current  bounding  box.center)]
		
		\node (a1) at (-1.5,0){$a$};
		\node (a2) at (-0.75,0){$a'$};
		\node (c1) at (0,0){$c$};

		\coordinate (f1) at (-1.5,1.5);
		\node (f2) at (-0.25,1.5)[draw,minimum width=20pt,minimum height=10pt,thick]{$\bar{f}_2$};
		
		\node (a) at (-1.5,3){$a$};
		\node (c2p) at (-0.25,3){$c'$};
		
		\begin{knot}[clip width=4]
		\strand [thick] (a1)
		to [out=90,in=-90] (f1); 
		\strand [thick] (a2)
		to [out=90,in=-105] (f2);
		\strand [thick] (c1)
		to [out=90,in=-75] (f2);
		\strand [thick] (f1)
		to [out=90,in=-90] (a);
		\strand [thick] (f2)
		to [out=90,in=-90] (c2p);
		\end{knot}
		\end{tikzpicture}
	}
	\mapsto
	\hbox{
		\begin{tikzpicture}[baseline=(current  bounding  box.center)]
		
		\node (a1) at (-1.5,0){$a$};
		\node (a2) at (-0.4,0){$a'$};
		\node (c1) at (0.4,0){$c$};

		\coordinate (f1) at (-0.6,2.8);
		\node (f2) at (-0.0,1.5)[draw,minimum width=20pt,minimum height=10pt,thick]{$\bar{f}_2$};
		
		\node (a) at (0.4,3.5){$a$};
		\node (c2p) at (-1.4,3.5){$c'$};
		
		\coordinate (upper) at (-0.8,1.3);
		\coordinate (unit) at (0,0.3);
		\coordinate (uctrl) at (0.4,1.25);

		\begin{knot}[clip width=4]
		\strand [thick] (a1)
		to [out=90,in=-180] (f1) to [out=0,in=90] (upper)
		to [out=-90, in=-180] (unit); 
		\strand [thick] (a2)
		to [out=90,in=-105] (f2);
		\strand [thick] (c1)
		to [out=90,in=-75] (f2);
		\strand [thick] (f2)
		to [out=90,in=-90] (c2p);
		\flipcrossings{1,3}
		\end{knot}
		\draw[thick] (unit)
		to [out=0,in=-90] (uctrl)
		to [out=90,in=-90] (a);
		\end{tikzpicture}
	}
	=
	\hbox{
		\begin{tikzpicture}[baseline=(current  bounding  box.center)]
		
		\node (a1) at (-1.5,0){$a$};
		\node (a2) at (-0.85,0){$a'$};
		\node (c1) at (0.4,0){$c$};

		\coordinate (f1) at (-0.6,2);
		\node (f2) at (-1.4,2.6)[draw,minimum width=20pt,minimum height=10pt,thick]{$\bar{f}_2$};
		
		\node (a) at (0.4,3.5){$a$};
		\node (c2p) at (-1.4,3.5){$c'$};

		\coordinate (unit) at (-0.2,0.3);
		\coordinate (uctrl) at (0.4,1.25);

		\begin{knot}[clip width=4]
		\strand [thick] (a1)
		to [out=90,in=-180] (f1); 
		\strand [thick] (a2)
		to [out=90,in=-105] (f2);
		\strand [thick] (c1)
		to [out=90,in=-75] (f2);
		\strand [thick] (f1)
		to [out=0, in=-180] (unit);
		\strand [thick] (f2)
		to [out=90,in=-90] (c2p);
		\flipcrossings{2,3}
		\end{knot}
		\draw[thick] (unit)
		to [out=0,in=-90] (uctrl)
		to [out=90,in=-90] (a);
		\end{tikzpicture},
	}
	\end{equation}
	where in the equality we have used the naturality of the braiding that the fact that objects of $\cat{A}$ are transparent to each other.The last diagram is the mate to the tensor product of a morphism $aa'a^*\rar\underleftarrow{\cat{C}}(c, c')$ and $\underline{\id}_a\colon \mathbb{I}\rar a$. 
	Now one uses that the braiding monodromy between $a,a'\in \cat{A}$ and $c\in\cat{C}$ has the property
	$$
			\hbox{
			\begin{tikzpicture}[baseline=(current  bounding  box.center)]
			
			\node (a1) at (-0.6,-0.5){$a$};
			\node (a2) at (-0.3,-0.5){$a'$};
			\node (c1) at (0,-0.5){$c$};

			\coordinate (ao1) at (-0.6,2.5);
			\coordinate (ao2) at (-0.3,2.5);
			\coordinate (co) at (0,2.5);
			
			\begin{knot}[clip width=4]
			\strand [thick] (c1) to (co);
			\strand [thick] (a1) to [out=90,in=-90] (0.3,1) to [out=90,in=-90] (ao1);
			\strand [thick] (a2) to [out=90,in=-90] (0.6,1) to [out=90,in=-90] (ao2);
			\flipcrossings{1,3}
			\end{knot}
		\end{tikzpicture}
	}
	=
		\hbox{
		\begin{tikzpicture}[baseline=(current  bounding  box.center)]
		
		\node (a1) at (-0.6,-0.5){$a$};
		\node (a2) at (-0.3,-0.5){$a'$};
		\node (c1) at (0,-0.5){$c$};

		\coordinate (ao1) at (-0.6,2.5);
		\coordinate (ao2) at (-0.3,2.5);
		\coordinate (co) at (0,2.5);
		
		\begin{knot}[clip width=4,clip radius=2pt]
			\strand [thick] (c1) to (co);
			\strand [thick] (a1)to [out=90,in=-90] (-0.6,1) to [out=90,in=-90] (0.3,1.5) to [out=90,in=-90] (-0.6,2) to [out=90,in=-90] (ao1);
			\strand [thick] (a2) to [out=90,in=-90] (-0.3,0) to [out=90,in=-90] (0.3,0.6)to [out=90,in=-90] (-0.3,1) to [out=90,in=-90] (ao2);
			\flipcrossings{3,1}
		\end{knot}
	\end{tikzpicture}
},
	$$
	as the objects in $\cat{A}$ are transparent to each other.
\end{proof}

	What we have shown so far is that every hom-object can be viewed as an object in the Drinfeld centre of $\cat{A}$. The next step is to define a composition morphism. This composition morphism will factor through the symmetric tensor product $\otimes_s$ for $\dcentcat{A}$, see Definition \ref{symtensdef}.

\begin{df}\label{compzadef}
	Let $c,c',c''\in\cat{C}$, and let $\underleftarrow{\underline{\cat{C}}}(c,c')=(\underleftarrow{\cat{C}}(c,c'),\mathfrak{b}')$ and $\underleftarrow{\underline{\cat{C}}}(c',c'')=(\underleftarrow{\cat{C}}(c',c''),\mathfrak{b}'')$ denote the lifts of the $\cat{A}$-enriched hom-objects to $\dcentcat{A}$ from Definition \ref{halfbraidequip}. Then we define the \emph{composition morphism for the $\dcentcat{A}_s$-enrichment} to be the composite
	\begin{equation}\label{ZAcomp}
	\Phi(\underleftarrow{\underline{\cat{C}}}(c',c'')\otimes_s\underleftarrow{\underline{\cat{C}}}(c,c'))\hookrightarrow \underleftarrow{\cat{C}}(c',c'')\underleftarrow{\cat{C}}(c,c') \raru{\circ} \underleftarrow{\cat{C}}(c,c''),
	\end{equation}
	on the underlying objects in $\cat{A}$.
\end{df}

In order for this definition to make sense, we need the composite from Equation \eqref{ZAcomp} to define a morphism in $\dcentcat{A}$:

\begin{lem}
	The composition from Equation \eqref{ZAcomp} defines a morphism in $\dcentcat{A}$.
\end{lem}

\begin{proof}
	We need to check that the morphism commutes with the braiding. That is, we need to show that the outside of the diagram
	$$
	\begin{tikzcd}
	a \Phi(\underleftarrow{\underline{\cat{C}}}(c',c'')\otimes_s\underleftarrow{\underline{\cat{C}}}(c,c')) \arrow[r,hook] \arrow[d,"\mathfrak{b}_a"'] 	& a \underleftarrow{\cat{C}}(c',c'')\underleftarrow{\cat{C}}(c,c') \arrow[r,"\circ"]\arrow[d,"s\circ(\id \otimes(\beta_{a,c''}^{-1}\circ\beta_{c'',a}^{-1})_*)"'] 	& a\underleftarrow{\cat{C}}(c,c'')\arrow[d,"s\circ (\beta_{a,c''}^{-1}\circ\beta_{c'',a}^{-1})_*"]\\
	\Phi(\underleftarrow{\underline{\cat{C}}}(c',c'')\otimes_s\underleftarrow{\underline{\cat{C}}}(c,c'))a  \arrow[r,hook]					& \underleftarrow{\cat{C}}(c',c'')\underleftarrow{\cat{C}}(c,c') a \arrow[r,"\circ"]																& \underleftarrow{\cat{C}}(c,c'')a
	\end{tikzcd}
	$$
	commutes. Here, the leftmost square is the definition of the half-braiding $\mathfrak{b}_a$ Equation \eqref{symtenshalfbraid} on the symmetric tensor product $\underleftarrow{\underline{\cat{C}}}(c',c'')\otimes_s\underleftarrow{\underline{\cat{C}}}(c,c')$ in terms of the half-braiding on $\underleftarrow{\underline{\cat{C}}}(c',c'')$ (Equation \eqref{halfbraidingonhomobjects}). Note that we have used monoidality of the symmetry in $\cat{A}$ to compose two instances of the symmetry here, and suppressed the isomorphisms between $a\underleftarrow{\cat{C}}(c',c'')$ and $\underleftarrow{\cat{C}}(c',ac'')$, here, as well as in the rightmost morphism. As the leftmost square commutes by definition, it suffices to check that the rightmost square commutes. For this we can ignore the symmetry $s$, as this is a natural transformation, hence commutes with the morphism $\circ$ in $\cat{A}$. The top route is then a composition followed by precomposition, while the bottom is precomposition followed by composition, so they agree by associativity of the composition in $\cat{C}$. 
\end{proof}

The associativity of this $\dcentcat{A}_s$-enriched composition is immediate from the associativity of the composition in $\underleftarrow{\cat{C}}$. To finish setting up the structure of the $\dcentcat{A}_s$-enriched category $\underleftarrow{\underline{\cat{C}}}$, we need to provide, for each $c\in\cat{C}$, an identity morphism $1_c\colon \mathbb{I}_s\rar \underleftarrow{\underline{\cat{C}}}(c,c)$, where $\mathbb{I}_s$ denotes the unit object for $\otimes_s$ as introduced in Definition \ref{symtensdef}. (Recall that the underlying object in $\cat{A}$ of $\mathbb{I}_s$ is $\oplus_{i\in\SA}i^*i$.) We then need to check that this morphism indeed specifies an identity in the sense that the following diagram commutes
\begin{equation}\label{unitmorphequation}
\begin{tikzcd}
\underleftarrow{\underline{\cat{C}}}(c,c') \otimes_s \mathbb{I}_s \arrow[r,"\rho"] \arrow[d,"\id\otimes_s1_c "']		&	\underleftarrow{\underline{\cat{C}}}(c',c)\\
\underleftarrow{\underline{\cat{C}}}(c,c')\otimes_s \underleftarrow{\underline{\cat{C}}}(c,c) \arrow[ru,"\circ"]&,
\end{tikzcd}
\end{equation}
where $\rho$ denotes the right unitor for $\otimes_s$, as well as the corresponding diagram for the left unitor.

\begin{df}\label{unitzadef}
	The \emph{$\dcentcat{A}$-identity morphism} $1_c\colon \mathbb{I}_s\rar \underleftarrow{\underline{\cat{C}}}(c,c)$ for $c\in\cat{C}$ is the mate for the morphism
	$$
		\hbox{
		\begin{tikzpicture}[baseline=(current  bounding  box.center)]
		\coordinate (i) at (-0.6,0.25);
		\node (id) at (-0.55,0){$\mathbb{I}_s$};
		\node (c) at (0,0){$c$};
		
		\coordinate (ctrl) at (0.25,1.25);
		
		\node (cp) at (0,2.2){$c$};
		
		\begin{knot}[clip width=4]
			\strand [thick, blue] (i) to [out=90,in=100] (ctrl) to [out=-80,in=90] (id);
			\strand [thick] (c) to (cp);
			\flipcrossings{2}
		\end{knot}
	\end{tikzpicture}
}:=
	\sum\limits_{i\in \SA}
	\hbox{
		\begin{tikzpicture}[baseline=(current  bounding  box.center)]
		\node (i) at (-0.75,0){$i$};
		\node (id) at (-0.5,0){$i^*$};
		\node (c) at (0,0){$c$};
		
		\coordinate (ctrl) at (0.25,1.2);
		
		\node (cp) at (0,2){$c$};
		
		\begin{knot}[clip width=4]
			\strand [thick] (i) to [out=90,in=140] (ctrl) to [out=-40,in=90] (id);
			\strand [thick] (c) to (cp);
			\flipcrossings{2}
		\end{knot}
	\end{tikzpicture}
}.
$$
\end{df}

We need to check that this indeed specifies a morphism in $\dcentcat{A}$, and that it satisfies Equation \eqref{unitmorphequation}.

\begin{lem}
	The $\dcentcat{A}$-identity morphism is a morphism in $\dcentcat{A}$, that is
	$$
	\begin{tikzcd}
	a \mathbb{I}_s \arrow[r,"1_c"]\arrow[d,"\beta_{a,\mathbb{I}_s}"]	&	a \underleftarrow{\underline{\cat{C}}}(c,c) \arrow[d,"\mathfrak{b}_{a}"]\\
	\mathbb{I}_s a \arrow[r,"1_c"]										&	 \underleftarrow{\underline{\cat{C}}}(c,c) a,
	\end{tikzcd}
	$$	
	commutes.
\end{lem}

\begin{proof}
	Recalling that the braiding $\mathfrak{b}_{a}$ was computed in terms of mates in Equation \eqref{braidingonmates}, the top and bottom routes compute as
	$$
	\hbox{
		\begin{tikzpicture}[baseline=(current  bounding  box.center)]
		\coordinate (i) at (-0.6,-1.75);
		\node (id) at (-0.5,-2){$\mathbb{I}_s$};
		\node (c) at (0,-2){$c$};
		
		\coordinate (ctrl) at (0.25,-0.5);
		\coordinate (tra) at (0,-1.3);
		\node (cp) at (0,2){$c$};
		
		\node (a) at (-1.25,-2){$a$};
		\coordinate (ac1) at (0.25,1);
		\coordinate (ac2) at (-0.75,-1);
		\coordinate (ac3) at (0.5, -1.2);
		\coordinate (ao) at (0.5,2);
		
		\begin{knot}[clip width=4, clip radius = 4pt]
			\strand [thick,blue] (i) to [out=90,in=80] (ctrl) to [out=-100,in=90] (id);
			\strand [thick] (c) to (cp);
			\strand [thick] (a) to [out=90,in=90] (ac1) to [out=-90,in=90] (ac2) to [out=-90,in=-90] (ac3) to [out=90,in=-90] (ao);
			\flipcrossings{1,4}
		\end{knot}
	\end{tikzpicture}
}
\quad\tn{and }
\hbox{
	\begin{tikzpicture}[baseline=(current  bounding  box.center)]
	\coordinate (i) at (-0.6,-1.75);
	\node (id) at (-0.5,-2){$\mathbb{I}_s$};
	\node (c) at (0,-2){$c$};
	
	\coordinate (ctrl) at (0.25,1);

	\coordinate (tra) at (0,-0.75);
	\node (cp) at (0,2){$c$};
	
	\node (a) at (-1.25,-2){$a$};
	\coordinate (ac1) at (0.25,-0.5);
	\coordinate (ac2) at (-0.25,-1);
	\coordinate (ac3) at (0.5, -0.5);
	\coordinate (ao) at (0.5,2);
	
	\begin{knot}[clip width=4]
	\strand [thick,blue] (i) to [out=90,in=80] (ctrl) to [out=-100,in=90] (id);
	\strand [thick] (c) to (tra);
	\strand [thick] (tra) to (cp);
	\strand [thick] (a) to [out=90,in=-90] (ac3) to [out=90,in=-90] (ao);
	\flipcrossings{1}
	\end{knot}
	\end{tikzpicture}
},
$$
respectively. The latter has summands (using the definition of the half-braiding on $\mathbb{I}_s$ from Equation \eqref{unithalfbraid})
$$
\hbox{
	\begin{tikzpicture}[baseline=(current  bounding  box.center)]
	\node (i) at (-0.7,-2){$i$};
	\node (id) at (-0.4,-2){$i^*$};
	\node (c) at (0,-2){$c$};
	
	\coordinate (ctrl) at (0.25,1);
	
	\node (cp) at (0,2){$c$};
	
	\node (a) at (-1,-2){$a$};
	\coordinate (ac1) at (0.25,-0.5);
	\coordinate (ac2) at (-0.25,-1);
	\coordinate (ac3) at (0.5, -1.2);
	\coordinate (ao) at (0.5,2);
	
	\coordinate (tra) at (0,-0.75);
	
	\node (phi) at (-0.9,-1)[draw]{$\phi$};
	\node (phid) at (-0.4,-0.4)[draw]{$\phi^*$};
	
	\begin{knot}[clip width=4]
	\strand [thick] (phi) to [out=90,in=80] (ctrl) to [out=-100,in=90] (phid);
	\strand [thick] (c) to (tra);
	\strand [thick] (tra) to (cp);
	\strand [thick] (a) to [out=90,in=-110] (phi);
	\strand [thick] (i) to [out=90,in=-70] (phi);
	\strand [thick] (id) to [out=90,in=-110] (phid);
	\strand [thick] (phid) to [out=-70,in=-90] (ac3) to [out=90,in=-90]  (ao);
	\flipcrossings{1,3,4,5,7}
	\end{knot}
	\end{tikzpicture}
}
=
\hbox{
	\begin{tikzpicture}[baseline=(current  bounding  box.center)]
	\node (i) at (-0.7,-2){$i$};
	\node (id) at (-0.4,-2){$i^*$};
	\node (c) at (0,-2){$c$};
	
	\coordinate (ctrl) at (0.25,1);
	\coordinate (ctrl2) at (0.4,1.1);
	
	\node (cp) at (0,2){$c$};
	
	\node (a) at (-1,-2){$a$};
	\coordinate (ac1) at (-0.4,-0.25);
	\coordinate (ac2) at (-0.25,-0.75);
	\coordinate (ac3) at (0.5, -1.2);
	\coordinate (ao) at (0.5,2);
	
	\coordinate (tra) at (0,-0.85);
	
	\node (phi) at (-0.9,-1)[draw]{$\phi$};
	\node (phid) at (-0.9,-0.2)[draw]{$\phi^t$};
	
	\begin{knot}[clip width=4,clip radius=3pt]
	\strand [thick] (phi) to [out=90,in=-90] (phid);
	\strand [thick] (phid) to [out=70,in=80] (ctrl) to [out=-100,in=90] (ac1) to [out=-90,in=90] (id);
	\strand [thick] (c) to (tra);
	\strand [thick] (tra) to (cp);
	\strand [thick] (a) to [out=90,in=-110] (phi);
	\strand [thick] (i) to [out=90,in=-70] (phi);
	\strand [thick] (phid) to [out=110,in=80](ctrl2) to [out=-100,in=90] (ac2) to [out=-90,in=-90] (ac3) to [out=90,in=-90]  (ao);
	\flipcrossings{1,4,5,7}
	\end{knot}
	\end{tikzpicture}
},
$$
where the $\phi$ give a resolution of the identity on $ai$. The terms specified by this last diagram sum to the top route, remembering that the objects in $\cat{A}$ are transparent to each other.
\end{proof}

\begin{lem}
	The identity morphisms satisfy the triangle equality from Equation \eqref{unitmorphequation}.
\end{lem}

\begin{proof}
	The unitor for $\dcentcat{A}_s$ is given in Equation \eqref{rightunitor}. Let $z\in\dcentcat{A}$ and let $f\colon z\rar \underleftarrow{\underline{\cat{C}}}(c,c')$ be a morphism. The mate for the image of $f$ under $\rho$ is:
	$$
	
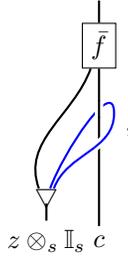
\hbox{
		\begin{tikzpicture}[baseline=(current  bounding  box.center)]
		
		\node (cf) at (0,-0.5){$z\otimes_s\mathbb{I}_s$};
		
		\node (inc) at (0.7,-0.5){$c$};
		
		\node (tr) at (0,0.1) {$\bigtriangledown$};
		\coordinate (trd) at (0,0);
		\coordinate (truw) at (-0.1,00.2);
		\coordinate (true) at (0.1,0.2);
		\coordinate (truee) at (0.05,0.2);
		
		\coordinate (rc) at (0.9,1.2);
		
		\node (c1) at (0.7,2.1)[draw]{$\bar{f}$};
		
		\coordinate (outg) at (0.7,2.7);
		
		\begin{knot}[clip width=4,clip radius=3pt]
			\strand [thick] (cf) to (trd);
			\strand [thick] (truw) to [out=110,in=-110] (c1);
			\strand [thick] (inc)  to [out=90, in=-90] (c1);
			\strand [thick] (c1) to [out=90,in=-90] (outg);
			\strand [thick,blue] (true) to [out=70, in=-90] (rc) to [out=90,in=70] (truee);
			\flipcrossings{1}
		\end{knot}
	\end{tikzpicture}
} ,
$$
where we simplified a double symmetry between $z$ and the summand of the strand, coming from the definition of the braiding on $\underleftarrow{\underline{\cat{C}}}(c,c')$.
On the other hand, the bottom route is the composite of $f$ with $1_c$, which in terms of mates is represented by the same diagram. This shows that the identity morphism indeed satisfies Equation \eqref{unitmorphequation}.
\end{proof}

We have now gathered the ingredients to define:

\begin{df}
	Let $\cat{C}$ be a braided tensor category and let $\cat{A}$ be a symmetric subcategory of $\cat{C}$. Then the \emph{left associated $\dcentcat{A}_s$-enriched category $\underleftarrow{\underline{\cat{C}}}$ for $\cat{C}$} is the $\dcentcat{A}_s$-enriched category with objects those of $\cat{C}$, hom-objects from Definition \ref{halfbraidequip}, composition from Definition \ref{compzadef} and identity morphisms from Definition \ref{unitzadef}.
\end{df}

\subsubsection{$\dcentcat{A}_s$-Tensoring}

The category $\underleftarrow{\underline{\cat{C}}}$ produced above is also tensored over $\dcentcat{A}_s$:

\begin{prop}\label{zatensoring}
	Let $\cat{C}$ be a braided tensor category containing $\cat{A}$. Then for all $c,c'\in \cat{C}$ and $(a,\beta)\in \dcentcat{A}$, the subobject $\pi((a,\beta),c)$ associated to the idempotent
	$$
	\Pi_{(a,\beta),c}
	=
	\hbox{
		\begin{tikzpicture}[baseline=(current  bounding  box.center)]
			\coordinate (west) at (-0.5,-0.3);
			\coordinate (north) at (0,0);
			\coordinate (east) at (0.5,-0.3);
			\coordinate (south) at (0,-0.6);
			\node (a) at (-0.25,-1.5) {$a$};
			\node (b) at (0.25,-1.5) {$c$};
			\coordinate (ao) at (-0.25,1);
			\coordinate (bo) at (0.25,1);

			\begin{knot}[clip width=4]
				\strand [blue, thick] (west)
				to [out=90,in=-180] (north)
				to [out=0,in=90] (east)
				to [out=-90,in=0] (south)
				to [out=-180,in=-90] (west);
				\strand [thick] (a) to [out=90,in=-90] (ao);
				\strand [thick] (b) to (bo);
				\flipcrossings{2,3}
			\end{knot}
		\end{tikzpicture}
	}
	:=\sum_{i\in\cat{O}(\cat{A})}\frac{1}{d_i}
	\hbox{
		\begin{tikzpicture}[baseline=(current  bounding  box.center)]
		
		\coordinate (east) at (1,0);
		
		\node (sym) at (-0.5,0.8)[draw,minimum height=15pt]{$s_{i,a}$};
		\node (br) at (-0.5,-0.6)[draw,minimum height=15pt]{$\beta_{i}$};
		
		\node (a) at (-0.5,-1.8) {$a$};
		\node (b) at (0.5,-1.8) {$c$};
		\coordinate (ao) at (-0.5,1.8);
		\coordinate (bo) at (0.5,1.8);
		
		\node (il) at (1.2,0){$i$};
		
		\begin{knot}[clip width=4]
			\strand [thick] (br.120) to [out=90,in=-90] (sym.-120);
			\strand [thick] (sym.60) to [out=90,in=90] (east) to [out=-90,in=-90] (br.-60);
			\strand [thick] (a) to (br);
			\strand [thick] (sym) to (br);
			\strand [thick] (sym) to (ao);
			\strand [thick] (b) to (bo);
			\flipcrossings{1}
		\end{knot}
	\end{tikzpicture}
}
$$
satisfies
$$
\dcentcat{A}((a,\beta),\underleftarrow{\underline{\cat{C}}}(c,c'))\cong \cat{C}(\pi((a,\beta),c),c'). 
$$
\end{prop}

In the diagram on the left for $\Pi_{(a,\beta),c}$, care should be taken to remember that the braidings on the $a$ strand take place in $\dcentcat{A}$, viewing $i$ as an object in $\dcentcat{A}$ equipped with the half-braiding coming from the symmetry.

\begin{proof}
	By formally the same proof as for \cite[Lemma 11]{Wasserman2017} the morphism $\Pi_{(a,\beta),c}$ is indeed an idempotent on $(a,\beta)c$. 
	
	To see the isomorphism, we notice that the hom-space $\cat{C}(\pi((a,\beta),c),c')$ can be viewed as the equaliser (in $\Vect$) for precomposition with $\Pi_{(a,\beta),c}$ and the identity on $\cat{C}(ac,c')$. Note that precomposition with $\Pi_{(a,\beta),c}$ takes a morphism $\bar{f}\colon ac\rar c'$ to the morphism
\begin{equation}\label{idemprecomp}
	\hbox{
		\begin{tikzpicture}[baseline=(current  bounding  box.center)]
		\coordinate (west) at (-0.5,-0.3);
		\coordinate (north) at (0,0);
		\coordinate (east) at (0.5,-0.3);
		\coordinate (south) at (0,-0.6);
		\node (a) at (-0.25,-1.5) {$a$};
		\node (b) at (0.25,-1.5) {$c$};
		\coordinate (ao) at (-0.25,0.5);
		\coordinate (bo) at (0.25,0.5);

		\node (f) at (0,1)[draw]{$\bar{f}$};
		
		\coordinate (out) at (0,2);
		
		\begin{knot}[clip width=4]
			\strand [blue, thick] (west)
			to [out=90,in=-180] (north)
			to [out=0,in=90] (east)
			to [out=-90,in=0] (south)
			to [out=-180,in=-90] (west);
			\strand [thick] (a) to [out=90,in=-90] (ao) to [out=90,in=-120] (f);
			\strand [thick] (b) to (bo) to [out=90,in=-60] (f);
			\strand [thick] (f) to (out);
			\flipcrossings{2,3}
		\end{knot}
	\end{tikzpicture}
}.
\end{equation}
On the other hand in $\dcentcat{A}$ the hom-spaces between $(a,\beta)$ and $(a',\beta')$ are the equalisers for the morphism $\cat{A}(a,a')\rar \cat{A}(a,a')$ given by

\begin{equation*}\label{braidingcomp}
	\hbox{
	\begin{tikzpicture}[baseline=(current  bounding  box.center)]
	\coordinate (a) at (0,-1.8);
	\node (f) at (0,-.25)[draw]{$f$};
	\coordinate (ao) at (0,1.2);

	\begin{knot}[clip width=4]
	\strand [thick] (a) to [out=90,in=-90] (f);
	\strand [thick]	(f) to [out=90,in=-90] (ao);
	\end{knot}
	\end{tikzpicture}
}
\mapsto
	\hbox{
		\begin{tikzpicture}[baseline=(current  bounding  box.center)]
		\coordinate (west) at (-0.5,-0.3);
		\coordinate (north) at (0,0.6);
		\coordinate (east) at (0.5,-0.3);
		\coordinate (south) at (0,-1.1);
		\coordinate (a) at (0,-1.8);
		\node (f) at (0,-.25)[draw]{$f$};
		\coordinate (ao) at (0,1.2);

		\begin{knot}[clip width=4]
		\strand [blue, thick] (west)
		to [out=90,in=-180] (north)
		to [out=0,in=90] (east)
		to [out=-90,in=0] (south)
		to [out=-180,in=-90] (west);
		\strand [thick] (a) to [out=90,in=-90] (f);
		\strand [thick]	(f) to [out=90,in=-90] (ao);
		\flipcrossings{1,2}
		\end{knot}
		\end{tikzpicture}
	}
	:=\sum_{i\in\cat{O}(\cat{A})}\frac{1}{d_i}
		\hbox{
		\begin{tikzpicture}[baseline=(current  bounding  box.center)]
		\coordinate (west) at (-0.5,-0.3);
		\coordinate (north) at (0,0.6);
		\coordinate (east) at (0.5,-0.3);
		\coordinate (south) at (0,-1.1);
		\coordinate (a) at (0,-1.8);
		\node (f) at (0,-.25)[draw]{$f$};
		\coordinate (ao) at (0,1.2);
		
		\node (i) at (-0.7,-0.3){$i$};

		\begin{knot}[clip width=4]
		\strand [thick] (west)
		to [out=90,in=-180] (north)
		to [out=0,in=90] (east)
		to [out=-90,in=0] (south)
		to [out=-180,in=-90] (west);
		\strand [thick] (a) to [out=90,in=-90] (f);
		\strand [thick]	(f) to [out=90,in=-90] (ao);
		\flipcrossings{1,2}
		\end{knot}
		\end{tikzpicture}
	}
	\end{equation*}
	and the identity on $\cat{A}(a,a')$. This description of the hom-spaces in $\dcentcat{A}$ is an easy consequence of cloaking (\cite[Lemma 10]{Wasserman2017}). Setting $a'=\underleftarrow{\underline{\cat{C}}}(c,c')$ and unpacking the definition of the braiding on $\underleftarrow{\underline{\cat{C}}}(c,c')$ in terms of mates, using the right hand side of Equation \eqref{braidingonmates}, we see that this agrees with Equation \eqref{idemprecomp}.
	This shows that the objects in the proposition are equalisers for the same morphisms, and therefore canonically isomorphic.
\end{proof}

Observe that this $\dcentcat{A}_s$-tensoring also provides an action of $\dcentcat{A}_s$ on the original category $\cat{C}$. This implies that we could first have provided the $\dcentcat{A}_s$-action, and enriched along this in a fashion similar to Definition \ref{enrich}. For the case $\cat{C}=(\dcentcat{A},\otimes_c)$, this $\dcentcat{A}_s$-action agrees with $\otimes_s$.

\begin{lem}\label{unitobjidentity}
	The functor $\mathbb{I}_s\cdot-:\cat{C}\rar \cat{C}$ induced by the $\dcentcat{A}$-tensoring is naturally isomorphic to the identity on $\cat{C}$ along the natural tranformation $\tau$ with components
	$$
	\tau_c=
	\hbox{
		\begin{tikzpicture}[baseline=(current  bounding  box.center)]
		\node (inc) at (0,0.5){$\pi(\mathbb{I}_s, c)$};
		\node (bo) at (0,3){$c$};
		\node (tr) at (0,1.1) {$\bigtriangledown$};
		\coordinate (trd) at (0,1);
		\coordinate (truw) at (-0.15,1.2);
		\coordinate (true) at (0.1,1.2);
		\coordinate (trum) at (-0.1,1.2);
		
		\coordinate (ctrl) at (0.175,2.3);
		
		\begin{knot}[clip width=4,clip radius=3pt]
			\strand [thick] (inc) to (trd);
			\strand [thick,blue] (trum) to [out=120, in=0] (ctrl) to [out=180,in=120] (truw);
			\strand [thick] (true) to [out=60,in =-90] (bo);
			\flipcrossings{2}
		\end{knot}
	\end{tikzpicture}
},
	$$
	where the triangle denotes the inclusion $\pi(\mathbb{I}_s,c)\hookrightarrow \mathbb{I}_sc$.
\end{lem}

\begin{proof}
	The proof is analogous to the proof of \cite[Lemma 20]{Wasserman2017}.
\end{proof}

For $\cat{A}\subset \dcentcat{A}_s$, the tensoring given here relates to the original action of $\cat{A}$ on $\cat{C}$ via the tensor product in $\cat{C}$ in the following way:

\begin{lem}\label{tensorvstensoring}
	Let $a\in \cat{A}$, viewed as the object $(a,s)\in \dcentcat{A}$, then for any $c\in\cat{C}$ we have:
	$$
	a\otimes_\cat{C} c \cong\pi((a,s)\otimes_c \mathbb{I}_s,c).
	$$
\end{lem}

\begin{proof}
	Consider the co-Yoneda embedding of $\pi((a,s),c)$
	\begin{align*}
	\cat{C}(\pi((a,s)\otimes_c \mathbb{I}_s,c),c')	&\cong \dcentcat{A}((a,s)\otimes_c \mathbb{I}_s,\underleftarrow{\underline{\cat{C}}}(c,c'))\\
													&\cong \dcentcat{A}(\mathbb{I}_s,(a,s)^*\otimes_c\underleftarrow{\underline{\cat{C}}}(c,c'))\\
													&\cong \cat{A}(\mathbb{I}_\cat{A},a^*\otimes_{\cat{A}}\underleftarrow{\cat{C}}(c,c')\\
													&\cong \cat{A}(\mathbb{I}_\cat{A},\underleftarrow{\cat{C}}(a\otimes_{\cat{C}}c,c'))\\
													&\cong \cat{C}(a\otimes_{\cat{C}}c,c'),
	\end{align*}
	where the first isomorphism is Proposition \ref{zatensoring}, the third \cite[Lemma 13]{Wasserman2017b}, the fourth Equation \eqref{aintohomiso}, and the final isomorphism comes from the defining property of $\underleftarrow{\cat{C}}(a\otimes_{\cat{C}}c,c')$, Definition \ref{enrichedhomdef}. We see that $a\otimes_{\cat{C}}c$ and $\pi((a,s),c)$ have canonically isomorphic co-Yoneda embeddings, so are canonically isomorphic.
\end{proof}

The $\dcentcat{A}_s$-tensoring allows us to describe morphisms in $\underleftarrow{\underline{\cat{C}}}$ as follows.

\begin{df}\label{assocmorphisms}
	Suppose $f\colon c\rar_z c'$ is a morphism in $\underleftarrow{\underline{\cat{C}}}$, then its \emph{$\dcentcat{A}$-mate} is the morphism in $\cat{C}$
	$$
	\tilde{f}\colon \pi(z,c)\rar c',
	$$
	obtained by applying the isomorphism from Proposition \ref{zatensoring} to $f$.
	
	Given a morphism $f\colon c \rar_a c'$ in $\underleftarrow{\cat{C}}$, we get the \emph{associated $\dcentcat{A}$-enriched morphism} $\underline{f}:c \rar_{(a,s)\otimes_c \mathbb{I}_s}c$ in $\underleftarrow{\underline{\cat{C}}}$ by applying the isomorphism from Proposition \ref{zatensoring} to the mate $\bar{f}\colon ac \rar c'$ composed with the isomorphism $ac \cong \pi((a,s)\otimes_c\mathbb{I}_s,c)$ from Lemma \ref{tensorvstensoring}.
\end{df}

\subsubsection{Associated $\dcentcat{A}$-enriched functors and transformations}
We want to extend our assignment $\cat{C}\mapsto\underleftarrow{\underline{\cat{C}}}$ to a 2-functor. For this, we need to know where to send functors between braided tensor categories containing $\cat{A}$:
\begin{lem}\label{assocZAfunct}
	For any morphism $F\in \BTCA$, the associated $\cat{A}$-enriched functor $\underleftarrow{F}$ from Definition \ref{associatedAfunctor} can be viewed as a $\dcentcat{A}_s$-enriched functor $\underleftarrow{\underline{F}}$.
\end{lem}
\begin{proof}
	As $F$ is braided, the morphism $\underleftarrow{F}_{c,c'}$ will be compatible with the braiding and therefore define a morphism $\underleftarrow{\underline{F}}_{c,c'}$ in $\dcentcat{A}$.
\end{proof}

For natural transformations, we have:

\begin{lem}\label{assocZAnattrafo}
	Let $\eta:(F,\mu_{-1},\mu_0,\mu_1)\rar (G,\nu_{-1},\nu_0,\nu_1)$ be a monoidal natural transformation between functors on $\cat{C}\supset\cat{A}$ with $\nu_{0}\eta|_\cat{A}=\mu_{0}$. Setting $\underleftarrow{\underline{\eta}}_c$ to be the associated $\dcentcat{A}$-enriched morphism (Definition \ref{assocmorphisms}) for $\eta_c:F(c)\rar G(c)$. Then $\underleftarrow{\underline{\eta}}:\underleftarrow{\underline{F}}\Rightarrow \underleftarrow{{G}}$ is a natural transformation between the associated $\dcentcat{A}$-enriched functors.
\end{lem}

\begin{proof}
	Naturality follows from Lemma \ref{assnattrafonatural}.
\end{proof}

\begin{prop}
	The assignment $\underleftarrow{\underline{(-)}}$ on $\BTCA$ with values in the 2-category of $\dcentcat{A}_s$-enriched categories is functorial.
\end{prop}

\begin{proof}
	We have to check that this assignment preserves composition of functors and of natural transformations. The composite of mates is the mate of composites for degree $\mathbb{I}_s$-morphisms, and similarly, as the action of a composition of functors on hom-objects is by the composition of the maps the functors induce on hom-objects, the image of the composition under $\underleftarrow{\underline{(-)}}$ will be the composite of the images.
\end{proof}

\subsubsection{Monoidal structure}
In this section we examine the sense in which $\underleftarrow{\underline{\cat{C}}}$ is monoidal. The monoidal structure on $\cat{C}$ will give rise to a monoidal structure on $\underleftarrow{\underline{\cat{C}}}$, however, this monoidal structure will not factor over the $(\dcentcat{A},\otimes_s)$-product. Rather, it will factor through the convolution tensor product (Definition \ref{ZAprodsdef}) of $\dcentcat{A}$-enriched categories, where we use the tensor product $\otimes_c$ on the $\dcentcat{A}$-hom objects. That is, it will be a $\dcentcat{A}$-crossed tensor category (Definition \ref{zacrossedtensordef}). We first establish the existence of a functor that acts as the unit for the crossed tensor structure.

\begin{lem}\label{unitforZAenrich}
	Let $\cat{C}$ be a braided tensor category containing $\cat{A}$. Then the associated $\dcentcat{A}_s$-enriched functor for the inclusion $\cat{A}\subset\cat{C}$ is a functor
	$$
	\mathbb{I}_\cat{C}:\cat{A}_\cat{Z}\rar \underleftarrow{\underline{\cat{C}}}.
	$$
	The category $\cat{A}_\cat{Z}$ was defined in Definition \ref{ZAprodsdef}.
\end{lem}

\begin{proof}
	We simply observe that $\underleftarrow{\underline{\cat{A}}}=\cat{A}_\cat{Z}$.
\end{proof}

The reader can also think of $\cat{A}_\cat{Z}$ as being defined by $\underleftarrow{\underline{\cat{A}}}=\cat{A}_\cat{Z}$. 

\begin{prop}\label{cunderlliscrossedtensor}
	If $\cat{C}$ is a braided tensor category containing a symmetric spherical fusion category $\cat{A}$, then $\underleftarrow{\underline{\cat{C}}}$ is $\dcentcat{A}$-crossed tensor (see Definition \ref{zacrossedtensordef}), with monoidal structure from Definition \ref{inducedatensor} lifted to the $\dcentcat{A}_s$-enriched category, with unit functor $\mathbb{I}_\cat{C}$ from Lemma \ref{unitforZAenrich}, and associators and unitors given by the mates to the associators and unitors in $\cat{C}$.
\end{prop}

\begin{proof}
	As the monoidal structure from Definition \ref{inducedatensor} is compatible with the composition, and the composition in $\underleftarrow{\underline{\cat{C}}}$ is a restriction of this, the lift of the monoidal structure will be compatible with composition. We still need to show that the morphisms
	$$
	\underleftarrow{\cat{C}}(c_1,c_1')\otimes_c \underleftarrow{\cat{C}}(c_2,c_2') \xrightarrow{\otimes} \underleftarrow{\cat{C}}(c_1c_2,c_1'c_2')
	$$
	are compatible with the braiding, so that they lift to $\dcentcat{A}$. In $\underleftarrow{\underline{\cat{C}}}\cattens{c} \underleftarrow{\underline{\cat{C}}}$, the left hand object will be equipped with the consecutive braiding on both factors, while the braiding on the right hand side comes from the braiding monodromy of $c_1c_2$. Comparing these braidings with respect to some $a\in\cat{A}$ in terms of mates for $f_1\colon c_1\rar_{a_1}c_1'$ and $f_2\colon c_2\rar_{a_2}c_2'$ gives
	$$
	
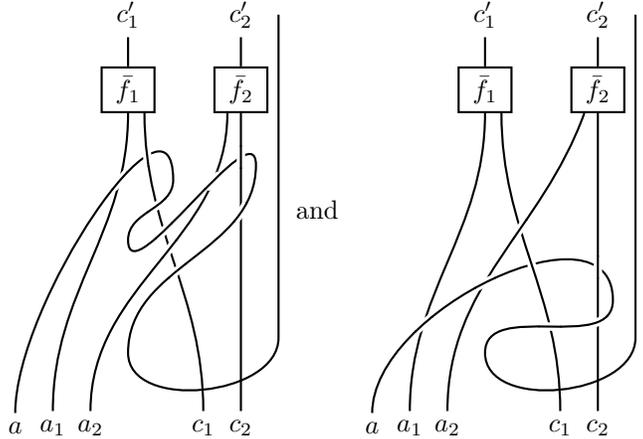
\hbox{
		\begin{tikzpicture}[baseline=(current  bounding  box.center)]
		\node (C) at (0,2)[draw,minimum width=20pt,minimum height=10pt,thick]{$\bar{f}_1$};
		\node (D) at (1.5,2)[draw,minimum width=20pt,minimum height=10pt,thick]{$\bar{f}_2$};
		\node (b3) at (-1,-2.5){$a_1$};
		\node (b4) at (-0.5,-2.5){$a_2$};
		\node (b5) at (1,-2.5){$c_1$};
		\node (b6) at (1.5,-2.5){$c_2$};
		\node (c1) at (0,3){$c_1'$};
		\node (c2) at (1.5,3){$c_2'$};
		
		\node (a) at (-1.5,-2.5){$a$};
		\coordinate (ac1) at (0.6, 0.8);
		\coordinate (ac2) at (0,0);
		\coordinate (ac3) at (1.7,1);
		\coordinate (ac4) at (0,-1.5);
		\coordinate (ac5) at (2,-1.3);
		\coordinate (ao) at (2,3);
		
		\begin{knot}[clip width=4,clip radius =4pt]
			\strand [thick] (c1) to (C);
			\strand [thick] (c2) to (D);
			\strand [thick] (C) to [out=-90,in=90] (b3);
			\strand [thick] (D.-120) to [out=-90,in=90] (b4);
			\strand [thick] (C.-55)	to [out=-90,in=90] (b5);
			\strand [thick] (D)	to [out=-90,in=90] (b6);
			\strand [thick] (a) to [out=90,in=90] (ac1) to [out=-90,in=90] (ac2) to [out=-90,in=90] (ac3) to [out=-90,in=90] (ac4) to [out=-90,in=-90] (ac5) to [out=90,in=-90] (ao);
			\flipcrossings{1,5,6,3,9,7}
		\end{knot}
	\end{tikzpicture}
}
\tn{ and }
\hbox{
	\begin{tikzpicture}[baseline=(current  bounding  box.center)]
	\node (C) at (0,2)[draw,minimum width=20pt,minimum height=10pt,thick]{$\bar{f}_1$};
	\node (D) at (1.5,2)[draw,minimum width=20pt,minimum height=10pt,thick]{$\bar{f}_2$};
	\node (b3) at (-1,-2.5){$a_1$};
	\node (b4) at (-0.5,-2.5){$a_2$};
	\node (b5) at (1,-2.5){$c_1$};
	\node (b6) at (1.5,-2.5){$c_2$};
	\node (c1) at (0,3){$c_1'$};
	\node (c2) at (1.5,3){$c_2'$};
	
	\node (a) at (-1.5,-2.5){$a$};
	\coordinate (ac3) at (1.7,-0.8);
	\coordinate (ac4) at (0,-1.5);
	\coordinate (ac5) at (2,-1.3);
	\coordinate (ao) at (2,3);
	
	\begin{knot}[clip width=4,clip radius =4pt]
	\strand [thick] (c1) to (C);
	\strand [thick] (c2) to (D);
	\strand [thick] (C) to [out=-90,in=90] (b3);
	\strand [thick] (D.-120) to [out=-110,in=90] (b4);
	\strand [thick] (C.-55)	to [out=-90,in=90] (b5);
	\strand [thick] (D)	to [out=-90,in=90] (b6);
	\strand [thick] (a) to [out=90,in=90] (ac3) to [out=-90,in=90] (ac4) to [out=-90,in=-90] (ac5) to [out=90,in=-90] (ao);
	\flipcrossings{1,5,3,7}
	\end{knot}
	\end{tikzpicture}
},
$$
for first braiding and then applying $\otimes$ and vice versa, respectively. These are indeed equal.

The associators satisfy the pentagon equations by virtue of the associators in $\cat{C}$ satisfying the pentagon equations. To define the unitors in $\underleftarrow{\underline{\cat{C}}}$, observe that we have natural isomorphisms between the tensor product composed with the unit functor and the unitor (see Definition \ref{ZAprodsdef}) for $\cattens{c}$ with components
$$
\mathbb{I}_\cat{C}(a)\otimes c \xrightarrow{\cong} \pi((a,s)\otimes_c \mathbb{I}_s, c),
$$
coming from Lemma \ref{tensorvstensoring}. A similar construction to the one done in Lemma \ref{zatensoring} gives a right action of $\dcentcat{A}_s$ on $\underleftarrow{\underline{\cat{C}}}$, and this action will commute with the left action. Using a result similar to Lemma \ref{tensorvstensoring} for this action allows us to define the right unitor. These unitors will satisfy the triangle equation by virtue of the left and right actions commuting.
\end{proof}

\subsubsection{Associated $\dcentcat{A}$-crossed tensor functors}
It turns out that the associated enriched functor $\underleftarrow{\underline{F}}$ to a functor $F\colon \cat{C}\rar \cat{D}$ of braided tensor categories containing $\cat{A}$ from Lemma \ref{assocZAfunct} carries naturally the structure of a $\dcentcat{A}$-crossed tensor functor (see Definition \ref{tensorfunctsnattrafos}). The $\dcentcat{A}$-crossed tensor structure on $\underleftarrow{\underline{F}}$ is obtained as follows:

\begin{lem}\label{assocZAXTfunc}
	The $\dcentcat{A}_s$-enriched functor $\underleftarrow{\underline{F}}$ is $\dcentcat{A}$-crossed tensor with as structure morphisms $\underline{\mu}_0$ and $\underline{\mu}_1$ the enriched natural tranformations with components the $\dcentcat{A}$-mates to the components of $\mu_0$ and $\mu_1$, respectively.
\end{lem}

\begin{proof}
	Observe that $\underline{\mu}_1$ defined in this way is indeed a $\dcentcat{A}_s$-enriched natural isomorphism between $F(-\otimes-)$ and $F(-)\otimes F(-)$. As the associators come from the associators in the braided tensor categories containing $\cat{A}$, this natural isomorphism will be compatible with the associators. 
	
	Denoting the inclusions of $\cat{A}$ into $\cat{C}$ and $\cat{D}$ by $i_\cat{C}$ and $i_\cat{D}$, respectively, the natural isomorphism $\mu_{0}$ takes $F\circ i_\cat{C}$ to $\cat{D}$, monoidally. This implies that $\underline{\mu}_{0}$ gives a natural isomorphism between $\underleftarrow{\underline{F}}\circ\mathbb{I}_\cat{C}$ and $\mathbb{I}_\cat{D}$, as desired.
\end{proof}

For natural transformations, we use the following:
\begin{lem}
	The associated $\dcentcat{A}$-natural transformation $\underleftarrow{\underline{\eta}}$ is a monoidal natural transformation of $\dcentcat{A}$-crossed tensor functors (see Definition \ref{zaxbtdef}).
\end{lem}
\begin{proof}
	There are two conditions in Definition \ref{zaxbtdef}. The first (which is analogous to the usual monoidality axiom) is satisfied by the fact that $\eta$ is monoidal. The second asks for commutativity of
		\begin{center}
		\begin{tikzcd}[row sep=tiny]
		\underleftarrow{\underline{F}}\circ \mathbb{I}_\cat{C}	\arrow[dd,"\underleftarrow{\underline{\eta}}\circ\mathbb{I}_\cat{C}"',Rightarrow] \arrow[rd,"\underleftarrow{\underline{\mu}}_0",Rightarrow]	&						\\
		&\mathbb{I}_\cat{D}\\
		\underleftarrow{\underline{G}}\circ \mathbb{I}_\cat{C}	\arrow[ru,Rightarrow,"\underleftarrow{\underline{\nu}}_0"']
		\end{tikzcd},
	\end{center}
	which follows from the fact that $\eta$ preserves the inclusion of $\cat{A}$ into $\cat{C}$ and $\cat{D}$.
\end{proof}

\subsubsection{$\dcentcat{A}$-crossed braiding}
We will now show that the braiding for $\cat{C}$ gives rise to a $\dcentcat{A}$-crossed braiding, see Definition \ref{zacrossedtensordef}. That is, the braiding will be a monoidal natural isomorphism between $\otimes_{\underleftarrow{\underline{\cat{C}}}}$ and $\otimes_{\underleftarrow{\underline{\cat{C}}}}\circ B$, where $B$ is the swap map for $\cattens{c}$ which uses the braiding on the hom-objects. The first step is to examine what the braiding functor $B$ from Definition \ref{ZAprodsdef} is on $\underleftarrow{\underline{\cat{C}}}\cattens{c}\underleftarrow{\underline{\cat{C}}}$. 

\begin{lem}\label{bisbeta}
	Let $\underleftarrow{\underline{\cat{C}}}$ be as above. On a the underlying object in $\cat{A}$ of a hom-object $\underleftarrow{\underline{\cat{C}}}\cattens{c}\underleftarrow{\underline{\cat{C}}}(c_1\boxtimes c_2, c_1'\boxtimes c_2')$ the braiding functor $B$ acts as 
	$$
	\underleftarrow{\cat{C}}(c_1,c_1')\underleftarrow{\cat{C}}(c_2,c_2')\xrightarrow{\beta^{-2}}\underleftarrow{\cat{C}}(c_1,c_1')\underleftarrow{\cat{C}}(c_2,c_2')\xrightarrow{s}\underleftarrow{\cat{C}}(c_2,c_2')\underleftarrow{\cat{C}}(c_1,c_1'),
	$$
	the composite of the functor $\beta^{-2}$ (Definition \ref{betasquaredinv}) with the symmetry in $\cat{A}$.
\end{lem}

\begin{proof}
	This is immediate from the definition of the half-braidings on the hom-objects (Definition \ref{halfbraidequip}).
\end{proof}

A braiding for a $\dcentcat{A}$-crossed tensor category $\underleftarrow{\underline{\cat{C}}}$ is by Definition \ref{zacrossedtensordef} a natural transformation from $\otimes_{\underleftarrow{\underline{\cat{C}}}}$ to $\otimes_{\underleftarrow{\underline{\cat{C}}}}\circ B$. So, our next step is to compute the composite of $B$ with the monoidal structure. It turns out that the resulting functor can be viewed as the monoidal structure $\bar{\otimes}$ on $\underleftarrow{\cat{C}}$ from Definition \ref{othertensor}. A similar argument to the proof of Proposition \ref{cunderlliscrossedtensor} shows that $\bar{\otimes}$ defines a monoidal structure on $\underleftarrow{\underline{\cat{C}}}$.

\begin{prop}\label{cunllbraid}
	Let $\cat{C}$ be a braided tensor category containing a spherical symmetric fusion category $\cat{A}$. Then the category $\underleftarrow{\underline{\cat{C}}}$ is a $\dcentcat{A}$-crossed braided tensor category.
\end{prop}

\begin{proof}
	We have already shown in Proposition \ref{cunderlliscrossedtensor} that $\underleftarrow{\underline{\cat{C}}}$ is $\dcentcat{A}$-crossed tensor. We have to show that the braiding for $\cat{C}$ gives a natural transformation between the tensor structure and the composite of $B$ with the tensor structure. We start by computing this composite. On objects, it is just the monoidal structure of $\cat{C}$ composed with the switch map. To see what $\otimes_{\underleftarrow{\underline{\cat{C}}}}\circ B$ is on hom-objects, observe the following. By Lemma \ref{bisbeta}, we know that $B$ acts on the underlying objects in $\cat{A}$ of the hom-objects as $\beta^{-2}\circ \tn{Switch}_\cat{A}$, where $\tn{Switch}_\cat{A}$ is the switch functor for $\cattens{\cat{A}}$. So we have the following equality of morphisms in $\cat{A}$:
	$$
	\Phi(\otimes_{\underleftarrow{\underline{\cat{C}}}}\circ B) =\otimes_{\underleftarrow{\cat{C}}}\circ\beta^{-2}\circ \tn{Switch}_\cat{A}=\bar{\otimes}\circ\tn{Switch}_\cat{A},
	$$
	where the last equality is Lemma \ref{otimesbetaisotimesafterbeta}, and $\Phi$ is the forgetful functor. By Lemma \ref{braidingbetaotimes}, the braiding in $\cat{C}$ induces a natural transformation between the last functor and $\otimes_{\underleftarrow{\cat{C}}}$. Lifting to $\dcentcat{A}$, this implies that the braiding gives a natural isomorphism between $\otimes_{\underleftarrow{\underline{\cat{C}}}}\circ B$ and $\otimes_{\underleftarrow{\underline{\cat{C}}}}$. Furthermore, the hexagon equations will be satisfied by virtue of them being satisfied in $\cat{C}$.
\end{proof}

\subsubsection{Associated braided $\dcentcat{A}$-crossed tensor functors}
Now that we have equipped our $\dcentcat{A}$-crossed tensor categories with a braiding, we can ask whether the functor from Lemma \ref{assocZAXTfunc} is braided in the sense of Definition \ref{tensorfunctsnattrafos}.
\begin{lem}
	Let $(F,\mu_{-1},\mu_{0},\mu_1):(\cat{C},\otimes_\cat{C},\beta_\cat{C})\rar (\cat{D},\otimes_\cat{D},\beta_\cat{D})$ be a morphism in $\BTCA$, then the associated $\dcentcat{A}$-crossed tensor functor $(\underleftarrow{\underline{F}},\underline{\mu}_0,\underline{\mu}_1):\underleftarrow{\underline{\cat{C}}}\rar \underleftarrow{\underline{\cat{D}}}$ from Lemma \ref{assocZAXTfunc} is $\dcentcat{A}$-crossed braided.
\end{lem}

\begin{proof}	
	This boils down to checking $\mu_1F(\beta_\cat{C})=\beta_\cat{D}\mu_1$, which holds by virtue of $F$ being a braided functor.
\end{proof}

In summary, we have shown:

\begin{prop}
	The assignment $\underleftarrow{\underline{(-)}}$ defines a bifunctor
	$$
	\BTCA\rar \ZAXBT.
	$$
\end{prop}

\subsubsection{Enriching the commutant of $\cat{A}$}
Let $\underleftarrow{\underline{\cat{C}}}$ be obtained by the enriching procedure above. Recall (from \cite[Definition 16]{Wasserman2017b}) that its neutral subcategory $\underleftarrow{\underline{\cat{C}}}_\cat{A}$ is the subcategory for which the Yoneda embedding  $\underleftarrow{\underline{\cat{C}}}(-,c)$ factors through $\cat{A}\hookrightarrow\dcentcat{A}$. We will now give a characterisation of this subcategory in terms of the so-called braided commutant of $\cat{A}$ in $\cat{C}$.

\begin{df}\label{commutantdef}
	Let $\cat{C}$ be a braided tensor category with braiding $\beta$ and let $\cat{B}$ be a braided monoidal subcategory. Then the \emph{braided commutant of $\cat{B}$ in $\cat{C}$} is the full subcategory with objects
	$$
	\cat{Z}_2(\cat{B},\cat{C})=\{c\in \cat{C}|\beta_{c,b}\circ\beta_{b,c}=\id_{bc}\quad \forall b \in \cat{B}\}.
	$$
	When $\cat{B}=\cat{C}$, we will denote this subcategory by $\cat{Z}_2(\cat{C})$, and call it the \emph{M\"uger centre of $\cat{C}$}.
\end{df}

When $\cat{A}$ is a symmetric subcategory of $\cat{C}$ the commutant $\cat{Z}_2(\cat{A},\cat{C})$ contains $\cat{A}$. 

\begin{prop}\label{yonedacommutant}
	Denote by $\underleftarrow{\underline{\MC{A}{C}}}\subset\underleftarrow{\underline{\cat{C}}}$ the full subcategory on the objects of $\MC{A}{C}$. Then:
	$$
	\underleftarrow{\underline{\MC{A}{C}}}=\underleftarrow{\underline{\cat{C}}}_\cat{A}.
	$$
\end{prop}

\begin{proof}
	As this is a statement about small full subcategories, it suffices to show that $\underleftarrow{\underline{\MC{A}{C}}}\subset\cat{C}_\cat{A}$ and $\underleftarrow{\underline{\MC{A}{C}}}\supset\cat{C}_\cat{A}$ at the level of objects. 
	
	The inclusion $\underleftarrow{\underline{\MC{A}{C}}}\subset\cat{C}_\cat{A}$ follows directly from the way the half-braidings on $\underleftarrow{\underline{\cat{C}}}(c,c')$ are defined in Definition \ref{halfbraidequip}: in Equation \eqref{halfbraidingonhomobjects} the morphism $(\beta_{a,c'}^{-1}\beta_{c',a}^{-1})_*$ is just the identity, so the composite becomes the symmetry in $\cat{A}$ between $\underleftarrow{\cat{C}}(c,c')$ and $a$. 
	
	For the reverse inclusion, suppose that $c$ is such that its Yoneda embedding $\underleftarrow{\underline{\cat{C}}}(-,c)$ factors through $\cat{A}$. This means that for each $c'\in \cat{C}$, the hom-object $\underleftarrow{\underline{\cat{C}}}(c',c)$ is $\underleftarrow{\cat{C}}(c',c)$ equipped with the symmetry in $\cat{A}$. Looking at the definition (Equation \eqref{halfbraidingonhomobjects}) of the half-braiding, we see that this implies that $(\beta_{a,c'}^{-1}\beta_{c',a}^{-1})_*$ is the identity on $\underleftarrow{\cat{C}}(c',c)$ for all $c'$. By the Yoneda lemma this means that $\beta_{a,c'}^{-1}\beta_{c',a}^{-1}$ is the identity on $ac$, which is what we wanted to show.
\end{proof}

Recall (see for example \cite[Section A.1.3]{Wasserman2017b}) that one can use lax monoidal functors between enriching categories to pass from categories enriched in one to categories enriched in the other. This is known as base-change and is done by applying the functors to the hom-objects. We observe the following which is immediate from the above proposition combined with the fact that the composite $\cat{A}\hookrightarrow \dcentcat{A}\xrightarrow{\Phi} \cat{A}$ of the forgetful functor with the inclusion functor is the identity on $\cat{A}$:

\begin{cor}\label{centralenrich}
	Let $\cat{C}$ be a braided tensor category containing $\cat{A}$, and assume that $\MC{A}{C}=\cat{C}$. Then:
	$$
	\overline{\underleftarrow{\underline{\cat{C}}}}=\underleftarrow{\cat{C}},
	$$
	where $\overline{ \cat{K}}$ for a $\dcentcat{A}_s$-enriched category $\cat{K}$ denotes base-change along the lax monoidal forgetful functor $\dcentcat{A}_s\rar \cat{A}$.
\end{cor}

\section{De-Enriching}

We will now take a $\dcentcat{A}$-crossed braided tensor category and produce a braided tensor category. We will show that this construction defines an inverse to the enriching procedure done in the previous section. 

\subsection{The de-enriching 2-functor}
Recall that base-change along lax monoidal functors preserves monoidal categories, and, if the lax monoidal functor is braided, also braided monoidal categories. The basics of this are recalled in eg. \cite[Section A.1.3]{Wasserman2017b}. The lax monoidal functors we use come from instances of the following general statement:
\begin{lem}\label{deenrichlaxmon}
	Let $\cat{V}$ be a symmetric monoidal category, and let $\cat{Z}$ be a symmetric $\cat{V}$-monoidal $\cat{V}$-enriched category. That is, the monoidal structure $\otimes_\cat{Z}$ factors through the $\cat{V}$-enriched cartesian product and is symmetric with respect to the swap functor on $\cat{Z}\cattens{\cat{V}}\cat{Z}$ induced by the symmetry in $\cat{V}$. Denote the unit object of $\cat{Z}$ by $\mathbb{I}_\cat{Z}$. Then the functor 
	$$
	\cat{Z}(\mathbb{I}_\cat{Z},-)\colon \cat{Z}\rar \cat{V}
	$$
	is symmetric lax monoidal, with lax structure given by
	\begin{align*}
	\mu_0\colon \mathbb{I}_\cat{V}&\xrightarrow{\id_{\mathbb{I}_\cat{Z}}}\cat{Z}(\mathbb{I}_\cat{Z},\mathbb{I}_\cat{Z})\\
	\mu_{z,z'}\colon \cat{Z}(\mathbb{I}_\cat{Z},z)	\cat{Z}(\mathbb{I}_\cat{Z},z')&\xrightarrow{\otimes_\cat{Z}}	\cat{Z}(\mathbb{I}_\cat{Z}\mathbb{I}_\cat{Z},zz') \xrightarrow[\cong]{(\mathbb{I}_\cat{Z}\rar\mathbb{I}_\cat{Z}\mathbb{I}_\cat{Z})^*} 	\cat{Z}(\mathbb{I}_\cat{Z},zz'),
	\end{align*}
	for $z,z'\in \cat{Z}$.
\end{lem}

We omit the proof. We will in particular use this lemma in the case where $\cat{V}=\Vect$ and $\cat{Z}=\cat{A}$ or $\cat{Z}=\dcentcat{A}$. Using this lemma, we set:

\begin{df}
	Let $\cat{K}$ be a $\dcentcat{A}_s$-enriched category. We will write $\DE(\cat{K})$ for the $\Vect$-enriched category obtained from $\cat{K}$ by change of basis along $\dcentcat{A}(\mathbb{I}_s,-)$, and will call this the \emph{de-enrichment of $\cat{K}$}.
\end{df}

To treat $\dcentcat{A}$-crossed (braided) tensor categories, we will additionally need a version of Lemma \ref{deenrichlaxmon} for 2-fold tensor categories. 

\begin{lem}
	Let $(\cat{Z},\otimes_1,\otimes_2)$ be a braided lax 2-fold tensor category \cite[Definitions 9, 11 and 12]{Wasserman2017a}, with, for $i=1,2$, the braiding for $\otimes_i$ denoted $\beta_i$ and the unit denoted by $\mathbb{I}_i$. Then the functor
	$$
	\cat{Z}(\mathbb{I}_2,-)\colon \cat{Z}\rar \Vect
	$$
	is braided lax monoidal with respect to $\otimes_1$, with lax structure given by
	\begin{align*}
	\mu_0\colon \C&\xrightarrow{\id_{\mathbb{I}_2}}\cat{Z}(\mathbb{I}_2,\mathbb{I}_2),\\
	\mu_{z,z'}\colon \cat{Z}(\mathbb{I}_2,z)	\cat{Z}(\mathbb{I}_2,z')&\xrightarrow{\otimes_\cat{Z}}	\cat{Z}(\mathbb{I}_2\otimes_1\mathbb{I}_2,z\otimes_1z') \xrightarrow[\cong]{(\mathbb{I}_2\rar\mathbb{I}_2\otimes_1\mathbb{I}_2)^*} 	\cat{Z}(\mathbb{I}_2,z\otimes_1z'),
	\end{align*}
	for $z,z'\in \cat{Z}$. Here the morphism $\mathbb{I}_2\rar\mathbb{I}_2\otimes_1\mathbb{I}_2$ is one of the structure morphisms from \cite[Definition 9]{Wasserman2017a}.
\end{lem}

We omit the proof of this lemma, it is a routine adaptation of the proof of Lemma \ref{deenrichlaxmon}.

Using this Lemma, one can prove, with a slight adjustment of the original proof, the following variation of the statement \cite[Proposition A.27]{Wasserman2017b} that base-change preserves (braided) monoidal categories. The phrasing here is specialised to the case $\cat{Z}=\dcentcat{A}$ that we actually need.

\begin{prop}
	Let $\cat{K}$ be a $\dcentcat{A}$-crossed braided tensor category. Then the linear category obtained from $\cat{K}$ by change of basis along $\dcentcat{A}(\mathbb{I}_s,-)$ is a braided tensor category.
\end{prop}

Being additionally tensored over $\dcentcat{A}_s$ will ensure the resulting category is tensored over $\Vect$, by standard arguments. We additionally note that, by definition, the de-enrichment of a $\dcentcat{A}_s$-linear category \cite[Definition 4]{Wasserman2017b} is a linear category, and the de-enrichment is semi-simple if the $\dcentcat{A}_s$-enriched category is. We have the following immediate corollary of this proposition and the preceding lemmas.

\begin{cor}
	The de-enrichment of any $\dcentcat{A}_s$-fusion category \cite[Definition 36]{Wasserman2017b} is a fusion category.
\end{cor}

It turns out that whenever we start with a $\dcentcat{A}$-crossed braided tensor category, the compatibility between the $\dcentcat{A}_s$-tensoring, the crossed tensor structure and the braiding will ensure the following:

\begin{prop}
	Let $\cat{K}$ be a $\dcentcat{A}$-crossed braided tensor category. Then the image of the unit functor $\mathbb{I}:\cat{A}_\cat{Z}\rar \cat{K}$ under change of basis along $\dcentcat{A}(\mathbb{I}_s,-)$ is a braided monoidal faithful functor from $\cat{A}$ into $\DE(\cat{K})$.
\end{prop}

\begin{proof}
	Denote the crossed monoidal structure on $\cat{K}$ by $\otimes_\cat{K}$, the $\dcentcat{A}_s$ tensoring by $\cdot$, and the crossed monoidal structure on $\cat{A}_\cat{Z}$ by $\otimes_\cat{A}$. By definition, the unit functor is such that $\mathbb{I}(a)\otimes_\cat{K} k= a\cdot k$, and the tensoring satisfies $(a\otimes_\cat{A} a' )\cdot k= a \cdot (a'\cdot k)$, so, taking $k=\mathbb{I}(\mathbb{I}_\cat{A})$, we have that $\mathbb{I}(a\otimes_\cat{A} a')=\mathbb{I}(a)\otimes_{\cat{K}}\mathbb{I}(a')$. The usual argument for compatibility between the unitors and the braiding \cite[Proposition 1]{Joyal1986} implies that the unit functor is also braided. This means that the image of $\mathbb{I}$ under de-enrichment,
	$$
	\DE(\mathbb{I}):\DE(\cat{A}_\cat{Z})=\cat{A}\rar \DE(\cat{K}),
	$$
	is a braided linear functor. Any linear monoidal functor on a fusion category is faithful, monoidality forces non-zero objects to be send to non-zero objects, so this finishes the proof.
\end{proof}

\subsection{Equivalence between braided categories containing $\cat{A}$ and $\dcentcat{A}$-crossed braided categories}
The goal of this section is to show that $\underleftarrow{\underline{(-)}}$ outlined in Section \ref{enrichingsection} is an equivalence of 2-categories between the 2-category $\BTCA$ of braided tensor categories (Definition \ref{bfcdef}) containing $\cat{A}$ and $\ZAXBT$ (see Definition \ref{zaxbtdef}), with inverse given by $\DE(-)$. 

\begin{thm}\label{bfcazaxbtev}
	The bifunctors $\DE(-)$ and $\underleftarrow{\underline{(-)}}$ are mutually inverse.
\end{thm}

\begin{proof}
	At the level of the objects of the categories that are the objects of the 2-categories, both $\DE(-)$ and $\underleftarrow{\underline{(-)}}$ are constructions that leave the objects of the categories invariant. So their composites also leave the objects invariant. This means that the components of the natural transformations witnessing that these bifunctors are mutually inverse are be functors that are the identity on objects. In order to prove the theorem, we in particular need to show that these component functors are equivalences, and for this it is then enough to show they induce isomorphisms on the hom-objects.
	
	Consider first the composite $\DE(-)\circ\underleftarrow{\underline{(-)}}$. Let $\cat{C}$ be an object of $\BTCA$, then the category $\DE(\underleftarrow{\underline{\cat{C}}})$ has hom-spaces:
	$$
	\dcentcat{A}(\mathbb{I}_s,\underleftarrow{\underline{\cat{C}}}(c,c'))\cong \cat{A}(\mathbb{I}_\cat{A},\underleftarrow{\cat{C}}(c,c'))\cong \cat{C}(c,c'),
	$$
	where we have used that $\dcentcat{A}(\mathbb{I}_s,-)\cong\cat{A}(\mathbb{I}_\cat{A},-)\circ\mathbf{Forget}$ (see \cite[Lemma 13]{Wasserman2017b}), and the definitions of the hom-objects of $\underleftarrow{\underline{\cat{C}}}$ and $\underleftarrow{\cat{C}}$. Taking $\mathbf{H}_\cat{C}\colon\DE(\underleftarrow{\underline{\cat{C}}})\rar\cat{C}$ to be the functor that is the identity on objects and the above isomorphism on hom-objects gives us an equivalence $\DE(\underleftarrow{\underline{\cat{C}}})\cong \cat{C}$ for each $\cat{C}\in \BTCA$. To see that these functors combine to a natural transformation, we simply observe that the above isomorphism is the inverse to the isomorphism used to define the action of the functor $\underleftarrow{\underline{F}}$ on the hom-objects. In particular, naturality is satisfied on the nose, there is no need for a ``naturator'' natural isomorphism between the linear functors $F\mathbf{H}_\cat{C}$ and $\mathbf{H}_\cat{D}\DE(\underleftarrow{\underline{F}})$.

	For the other composite, we observe that, for $\cat{K}\in \ZAXBT$, the underlying objects in $\cat{A}$ of the hom-objects of $\underleftarrow{\underline{\DE({ \cat{K}})}}$ are characterised by:
	\begin{align}\begin{split}\label{enrichafterde}
		\dcentcat{A}(z,\underleftarrow{\underline{\DE({ \cat{K}})}}(k,k'))&\cong \DE({\cat{K}})(z\cdot k,k')\\
	&= \dcentcat{A}(\mathbb{I}_s,\cat{K}(z\cdot k,k'))\\
	&\cong \dcentcat{A}(z,\cat{K}(k,k')),
	\end{split}
	\end{align}
	where the first isomorphism is the $\dcentcat{A}$-tensoring (Proposition \ref{zatensoring}) of $\underleftarrow{\underline{\DE({\cat{K}})}}$, the second the definition of $\DE$, the third is the $\dcentcat{A}$-tensoring of $\cat{K}$. These isomorphisms are all natural, so will combine to an equivalence $\mathbf{J}_\cat{K}$ between $\underleftarrow{\underline{\DE({ \cat{K}})}}$ and $\cat{K}$, which is the identity on objects. Similarly to before, the functors induced by de-enrichment are defined using the isomorphisms above, so naturality of the natural transformation defined by the $\mathbf{J}_\cat{K}$ is automatic. 
\end{proof}

\section{The Reduced Tensor Product}
We have now gathered all we need to define the reduced tensor product of braided tensor categories containing $\cat{A}$, by pushing $\cattens{s}$ along the equivalence from Theorem \ref{bfcazaxbtev}. After giving this definition, we will establish some basic properties that this product has.

\subsection{Definition of the reduced tensor product}

\begin{df}
	The \emph{reduced tensor product} $\rt$ on the category $\BFCA$ (see Definition \ref{bfcdef}) is defined by
	\begin{center}
		\begin{tikzcd}[column sep= large]
		\BFCA\times 	\BFCA \arrow[d,"\rt"]\arrow[r,"\underleftarrow{\underline{(-)}}\times\underleftarrow{\underline{(-)}}"] &\ZAXBT \times  \ZAXBT  \arrow[d,"\cattens{s}"]\\
		 \BFCA&\ZAXBT  \arrow[l,"\DE"], 
		\end{tikzcd}
	\end{center}
	where $\cattens{s}$ was defined in Definition \ref{ZAprodsdef}. 
\end{df}

To make $\BTCA$ into a symmetric monoidal 2-category for this reduced tensor product, we need to specify associators. We will do this by using the associators $\tilde{\alpha}$ for $\cattens{s}$, which in turn are induced from the associators $\alpha^s$ for $\otimes_s$. 

\begin{df}
	Let $\cat{C},\cat{D},\cat{E}\in \BTCA$. Then the \emph{associator for $\cat{C},\cat{D}$ and $\cat{E}$} is the functor
	$$
	A_{\cat{C},\cat{D},\cat{E}}\colon (\cat{C}\rt \cat{D})\rt\cat{E} \rar \cat{C}\rt (\cat{D}\rt \cat{E}),
	$$
	given by the composite
	\begin{align*}
	 (\cat{C}\rt \cat{D})\rt\cat{E} 	&= \DE((\underleftarrow{\underline{\cat{C}\rt \cat{D}}})\cattens{s}\underleftarrow{\underline{\cat{E}}})\\
	 									&= \DE(\underleftarrow{\underline{\DE(\underleftarrow{\underline{\cat{C}}}\cattens{s}\underleftarrow{\underline{\cat{D}}})}}\cattens{s}\underleftarrow{\underline{\cat{E}}})\\
	 									&\xrightarrow{\mathbf{J}_{\underleftarrow{\underline{\cat{C}}}\cattens{s}\underleftarrow{\underline{\cat{D}}}}} \DE((\underleftarrow{\underline{\cat{C}}}\cattens{s}\underleftarrow{\underline{\cat{D}}})\cattens{s}\underleftarrow{\underline{\cat{E}}})\\
	 									&\xrightarrow{\tilde{\alpha}}\DE(\underleftarrow{\underline{\cat{C}}}\cattens{s}(\underleftarrow{\underline{\cat{D}}}\cattens{s}\underleftarrow{\underline{\cat{E}}}))\\
	 									& \xrightarrow{\mathbf{J}^{-1}_{\underleftarrow{\underline{\cat{D}}}\cattens{s}\underleftarrow{\underline{\cat{E}}}}}\DE(\underleftarrow{\underline{\cat{C}}}\cattens{s}\underleftarrow{\underline{\DE(\underleftarrow{\underline{\cat{D}}}\cattens{s}\underleftarrow{\underline{\cat{E}}})}})\\
	 									&=  \cat{C}\rt (\cat{D}\rt\cat{E}). 
	\end{align*}
\end{df}

These associators need to satisfy the pentagon equations, possibly up to a invertible 2-cell. Fortunately, we can take this this 2-cell to be the identity, as we will show below. To simplify the argument, we first prove the following lemma:

\begin{lem}\label{DEaltdesc}
	Let $\cat{C}\in \BTCA$. Then
	$$
	\underleftarrow{\underline{\dcentcat{A}}}(\mathbb{I}_s, \underleftarrow{\underline{\cat{C}}}(c,c'))\cong \underleftarrow{\underline{\DE(\underleftarrow{\underline{\cat{C}}})}}(c,c'),
	$$
	canonically.
\end{lem}

\begin{proof}
	Using Equation \eqref{enrichafterde}, it suffices to observe
	\begin{equation}\label{doublede}
	\dcentcat{A}(z, \underleftarrow{\underline{\dcentcat{A}}}(\mathbb{I}_s, \underleftarrow{\underline{\cat{C}}}(c,c')))\cong \dcentcat{A}(z,\underleftarrow{\underline{\cat{C}}}(c,c')),
	\end{equation}
	which follows from the definition of $\underleftarrow{\underline{\dcentcat{A}}}$.
\end{proof}

Observe that all functors involved in the definition of the associators are the identity on objects, so it will suffice to consider what happens on morphisms. The lemma gives an alternative expression for the associators:

\begin{lem}\label{assonhom}
	Let $\cat{C},\cat{D},\cat{E}\in \BTCA$ and, for $i=1,2$, let $c_i,d_i,e_i$ be objects of the respective categories. We adopt the shorthand $\cat{C}(c_1,c_2)=\cat{C}_{12}$, with the obvious extension to the other categories, their $\dcentcat{A}$-enriched versions and their products. We will further suppress $\otimes_s$ from the notation for this lemma. Note that the action of the associators on the hom-objects between $c_i,d_i$ and $e_i$ is a morphism
	$$
	A_{\cat{C},\cat{D},\cat{E}}^{12}\colon \dcentcat{A}(\mathbb{I}_s, \underleftarrow{\underline{\DE(\underleftarrow{\underline{\cat{C}}}\underleftarrow{\underline{\cat{D}}})}}_{12}\underleftarrow{\underline{\cat{E}}}_{12})\rar \dcentcat{A}(\mathbb{I}_s, \underleftarrow{\underline{\cat{C}}}_{12}\underleftarrow{\underline{\DE(\underleftarrow{\underline{\cat{D}}}_{12}\underleftarrow{\underline{\cat{E}}}_{12})}}).
	$$
	This morphism is equal to the composite
	\begin{align*}
	 \dcentcat{A}(\mathbb{I}_s, \underleftarrow{\underline{\DE(\underleftarrow{\underline{\cat{C}}}\underleftarrow{\underline{\cat{D}}})}}_{12}\underleftarrow{\underline{\cat{E}}}_{12})&\cong \dcentcat{A}(\mathbb{I}_s, \underleftarrow{\underline{\dcentcat{A}}}(\mathbb{I}_s,\underleftarrow{\underline{\cat{C}}}_{12}\underleftarrow{\underline{\cat{D}}}_{12})\underleftarrow{\underline{\cat{E}}}_{12})\\
	 &\cong \dcentcat{A}(\mathbb{I}_s, (\underleftarrow{\underline{\cat{C}}}_{12}\underleftarrow{\underline{\cat{D}}}_{12})\underleftarrow{\underline{\cat{E}}}_{12})\\
	 &\xrightarrow{\alpha^s} \dcentcat{A}(\mathbb{I}_s, \underleftarrow{\underline{\cat{C}}}_{12}(\underleftarrow{\underline{\cat{D}}}_{12}\underleftarrow{\underline{\cat{E}}}_{12}))\\
	 &\cong \dcentcat{A}(\mathbb{I}_s, \underleftarrow{\underline{\cat{C}}}_{12}\underleftarrow{\underline{\DE(\underleftarrow{\underline{\cat{D}}}_{12}\underleftarrow{\underline{\cat{E}}}_{12})}}),
	\end{align*}
	where the first isomorphism is Lemma \ref{DEaltdesc}, the second Equation \eqref{aintohomiso} combined with the definition of $\underleftarrow{\underline{\dcentcat{A}}}$, and the final isomorphism is the inverse of the corresponding isomorphisms.
\end{lem}

We are now in a position to verify the pentagon equations.

\begin{prop}
	The associators $A$ satisfy the pentagon equations on the nose. 
\end{prop}

\begin{proof}
	Let $\cat{C},\cat{D},\cat{E}\in \BTCA$ and, for $i=1,2$, let $c_i,d_i,e_i$ be objects of the respective categories, and adapt the notation from Lemma \ref{assonhom}. It is enough to check the pentagon equation on hom-objects. Using Lemma \ref{assonhom} we get, for each edge of the pentagon, a commutative square, which we will spell out for one edge
	\begin{center}
		\begin{tikzcd}
		(((\cat{C}\rt \cat{D})\rt \cat{E})\rt \cat{F})_{12} \arrow[rr,"A_{\cat{C}\rt \cat{D},\cat{E},\cat{F}}"]\arrow[d,"\cong"] &\quad& ((\cat{C}\rt \cat{D})\rt (\cat{E}\rt \cat{F}))_{12} \arrow[d, "\cong"]\\
		\dcentcat{A}(\mathbb{I}_s, ((\underleftarrow{\underline{\cat{C}\rt \cat{D}}}_{12})\underleftarrow{\underline{\cat{E}}}_{12})\underleftarrow{\underline{\cat{F}}}_{12}) \arrow[rr, "(\alpha^s_{\underleftarrow{\underline{\cat{C}\rt \cat{D}}}_{12},\underleftarrow{\underline{\cat{E}}}_{12},\underleftarrow{\underline{\cat{F}}}_{12}})"] \arrow[d," \mathbf{J}_{\underleftarrow{\underline{\cat{C}}}_{12}\underleftarrow{\underline{\cat{D}}}_{12}}"]&\quad&	\dcentcat{A}(\mathbb{I}_s, (\underleftarrow{\underline{\cat{C}\rt \cat{D}}}_{12})(\underleftarrow{\underline{\cat{E}}}_{12}\underleftarrow{\underline{\cat{F}}}_{12})) \arrow[d," \mathbf{J}_{\underleftarrow{\underline{\cat{C}}}_{12}\underleftarrow{\underline{\cat{D}}}_{12}}"]\\
		\dcentcat{A}(\mathbb{I}_s, ((\underleftarrow{\underline{\cat{C}}}_{12} \underleftarrow{\underline{\cat{D}}}_{12})\underleftarrow{\underline{\cat{E}}}_{12})\underleftarrow{\underline{\cat{F}}}_{12})\arrow[rr, "(\alpha^s_{\underleftarrow{\underline{\cat{C}}}_{12} \underleftarrow{\underline{\cat{D}}}_{12},\underleftarrow{\underline{\cat{E}}}_{12},\underleftarrow{\underline{\cat{F}}}_{12}})"] &\quad& \dcentcat{A}(\mathbb{I}_s, (\underleftarrow{\underline{\cat{C}}}_{12} \underleftarrow{\underline{\cat{D}}}_{12})(\underleftarrow{\underline{\cat{E}}}_{12}\underleftarrow{\underline{\cat{F}}}_{12})),
		\end{tikzcd}
	\end{center}
	where the top square is commutative by definition, while the bottom square commutes by naturality of $\alpha^s$. The squares associated in this way to adjacent edges in the pentagon share a vertical edge, so we get a diagram that looks schematically like
	$$
	
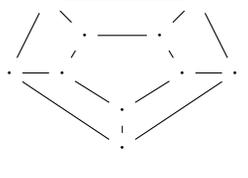
\begin{tikzpicture}
	\node (o1) at (-1.5,0){$\cdot$};
	\node (ro1) at (1.5,0){$\cdot$};
	\node (i1) at (-0.8,0){$\cdot$};
	\node (ri1) at (0.8,0){$\cdot$};
	
	\node(o2) at (-1,1){$\cdot$};
	\node(ro2) at (1,1){$\cdot$};
	\node(i2) at (-0.5,0.5){$\cdot$};
	\node(ri2) at (0.5,0.5){$\cdot$};
	
	\node (o3) at (0,-1){$\cdot$};
	\node (i3) at (0,-0.5){$\cdot$};
	
	\draw (o1) to (i1);
	\draw (o2) to (i2);
	\draw (o3) to (i3);
	\draw (ro1) to (ri1);
	\draw (ro2) to (ri2);
	
	\draw (o1) to (o2);
	\draw (o2) to (ro2);
	\draw (ro2) to (ro1);
	\draw (ro1) to (o3);
	\draw (o3) to (o1);

	\draw (i1) to (i2);
	\draw (i2) to (ri2);
	\draw (ri2) to (ri1);
	\draw (ri1) to (i3);
	\draw (i3) to (i1);
	
	\end{tikzpicture},
	$$
	where the squares correspond to the outside of the one constructed above, the outer pentagon is the pentagon we want to show commutes, and the inner pentagon is the image under $\dcentcat{A}(\mathbb{I}_s,-)$ of a pentagon equation for $\alpha^s$. All these faces commute, so we conclude that the outer pentagon commutes.
\end{proof}

\subsubsection{The Unit for the Reduced Tensor Product}

\begin{prop}\label{rtunit}
	Let $\cat{C}\in \BTCA$. Then we have an equivalence
	$$
	\dcentcat{A}\rt\cat{C}\cong \cat{C},
	$$
	given by the image of the unitor for $\cattens{s}$, as defined in Definition \ref{enrichedcartprod}. We similarly have a right unitor, and together these satisfy the triangle equality.
\end{prop} 

\begin{proof}
	It is clear that the functors defined above are equivalences, so we are left with checking the triangle equality. Recall that the unitor for $\cattens{s}$ is defined using the $\dcentcat{A}_s$-tensoring. On objects, the two routes along the triangle agree up to the natural isomorphism with components coming from Equation \eqref{tensoronprod}. On hom-objects, we observe that Equation \eqref{tensoronprod} implies
	$$
	\underleftarrow{\underline{\cat{C}}}\cattens{s}\underleftarrow{\underline{\cat{D}}}(z_1\cdot c_1\boxtimes d_1, z_2 \cdot c_2\boxtimes d_2)\cong\underleftarrow{\underline{\cat{C}}}\cattens{s}\underleftarrow{\underline{\cat{D}}}( c_1\boxtimes z_1\cdot d_1,c_2\boxtimes z_1\cdot d_2),
	$$
	so we are done.
\end{proof}

In summary:

\begin{thm}\label{rttheorem}
	The reduced tensor product $\rt$ defines a symmetric monoidal structure on the 2-category $\BTCA$.
\end{thm}

\subsection{Basic properties of the reduced tensor product}
We will now establish some basic properties of the reduced tensor product, and compute it in some examples. In our computations, the following result will be used frequently:

\begin{prop}\label{tensofneutralpart}
	Let $\cat{K}$ and $\cat{L}$ be $\dcentcat{A}_s$-enriched and tensored categories. Then
	$$
	(\cat{K}\cattens{s}\cat{L})_\cat{A}\cong\cat{K}\cattens{s}\cat{L}_\cat{A}\cong\cat{K}_\cat{A}\cattens{s}\cat{L}\cong\overline{\cat{K}_\cat{A}}\cattens{\cat{A}}\overline{\cat{L}_\cat{A}},
	$$
	where we view the $\cat{A}$-enriched and tensored category on the right as $\dcentcat{A}_s$-enriched and tensored category by using the symmetric strong monoidal inclusion functor $\cat{A}\hookrightarrow \dcentcat{A}$.
\end{prop}

This proposition is \cite[Proposition 18]{Wasserman2017b}.

\subsubsection{Reduced tensor product and the commutant of $\cat{A}$}
It is interesting to examine what the reduced tensor product becomes on the commutant (Definition \ref{commutantdef}) of $\cat{A}$ in $\cat{C}$. When taking the reduced tensor product, this commutant behaves nicely. We will use the following bit of notation:

\begin{notation}
	Let $\cat{C}$ and $\cat{D}$ be braided tensor categories containing $\cat{A}$. The symbol $\bxc{A}$, with slight abuse of notation, denotes
	$$
	\cat{C}\bxc{A}\cat{D}=\DE(\underleftarrow{\cat{C}}\bxc{A}\underleftarrow{\cat{D}}),
	$$
	where the use of $\bxc{A}$ on the right hand side denotes the $\cat{A}$-product introduced in Definition \ref{enrichedcartprod}, and $\DE$ denotes change of basis along $\cat{A}(\mathbb{I}_\cat{A},-)$.
\end{notation}

\begin{prop}\label{commutantrt}
	Let $\cat{C},\cat{D}\in \BTCA$. Then the commutant of $\cat{A}$ in $\cat{C}\rt\cat{D}$ satisfies
	$$
	\cat{Z}_2(\cat{A},\cat{C}\rt\cat{D})\cong \MC{A}{C}\bxc{A}\MC{A}{D}.
	$$
\end{prop}

\begin{proof}
	Using Proposition \ref{yonedacommutant}, this follows directly from \cite[Proposition 18]{Wasserman2017b}.
\end{proof}

\subsubsection{Examples}
To give the reader some intuition for the reduced tensor product, we compute some examples.

\begin{ex}
	Let $\cat{C}$ be a braided tensor category containing a symmetric fusion category $\cat{A}$. Then the reduced tensor product over $\cat{A}$ of $\cat{C}$ with $\cat{A}$ is given by
	$$
	\cat{C}\rt \cat{A} \cong \cat{Z}_2(\cat{A},\cat{C})\bxc{A}\cat{A}\cong \cat{Z}_2(\cat{A},\cat{C}).
	$$
	To see this, we observe that the neutral subcategory of $\cat{A}$ enriched over itself is all of $\underleftarrow{\underline{\cat{A}}}$. Now apply Proposition \cite[Proposition 18]{Wasserman2017b} to get
	$$
	\cat{C}\rt \cat{A} \cong \DE(\underleftarrow{\underline{\cat{C}}}_\cat{A}\cattens{s}\underleftarrow{\underline{\cat{A}}})\cong \DE(\underleftarrow{\MC{A}{C}}\cattens{\cat{A}}\underleftarrow{\cat{A}})\cong\MC{A}{C}.
	$$
	Here, we have used Corollary \ref{centralenrich} and that $\underleftarrow{\cat{A}}$ is the unit for $\bxc{A}$.
\end{ex}

\begin{ex}
	Let $\cat{C}$ and $\cat{D}$ be braided tensor categories containing $\cat{A}$, and assume that $\cat{D}=\MC{A}{D}$. Then
	$$
	\cat{C}\rt\cat{D}\cong \MC{A}{C}\bxc{A}\cat{D}.
	$$
	The assumption on $\cat{D}$ means that we have $\underleftarrow{\underline{\cat{D}}}_\cat{A}=\underleftarrow{\underline{\cat{D}}}$, by Proposition \ref{yonedacommutant}. Using Proposition \ref{tensofneutralpart} and Corollary \ref{centralenrich}, we get:
	$$
	\cat{C}\rt\cat{D}\cong\DE(\underleftarrow{\underline{\cat{C}}}_\cat{A}\cattens{s}\underleftarrow{\underline{\cat{D}}})\cong \DE(\underleftarrow{\MC{A}{C}}\bxc{A}\underleftarrow{\cat{C}})\cong \MC{A}{C}\bxc{A}\cat{D}.
	$$
\end{ex}

\begin{ex}
	Let $\cat{C}$ and $\cat{D}$ be braided tensor categories containing $\cat{A}$, and assume that $\cat{C}=\MC{A}{C}$ and that $\cat{D}=\MC{A}{D}$. Then
	$$
	\cat{C}\rt\cat{D}\cong \cat{C}\bxc{A}\cat{D}.
	$$
	The assumption on $\cat{C}$ and $\cat{D}$ means that we have $\underleftarrow{\underline{\cat{C}}}_\cat{A}=\underleftarrow{\underline{\cat{C}}}$ and $\underleftarrow{\underline{\cat{D}}}_\cat{A}=\underleftarrow{\underline{\cat{D}}}$, by Proposition \ref{yonedacommutant}. Using Proposition \ref{tensofneutralpart} and Corollary \ref{centralenrich}, we get:
	\begin{align*}
	\cat{C}\rt\cat{D}&\cong \DE(\underleftarrow{\underline{\cat{C}}}\cattens{s}\underleftarrow{\underline{\cat{D}}})\\
	&\cong \DE(\underleftarrow{\cat{C}}\bxc{A}\underleftarrow{\cat{D}})\\
	&\cong \cat{C}\bxc{A}\cat{D}.
	\end{align*}
\end{ex}

\subsubsection{Minimal modular extensions}
In this section we will show that the reduced tensor product between so called minimal modular extensions is again a minimal modular extension. We first recall the definition of a minimal modular extension.

\begin{df}[\cite{Mueger2002}]
	Let $\cat{C}\in \BTCA$, then a \emph{minimal modular extension of $\cat{C}$ over $\cat{A}$} is a braided tensor category $\cat{M}$ containing $\cat{C}$ with $\MC{M}{M}=\Vect$ and $\MC{A}{M}=\cat{C}$. The (possibly empty) set of minimal modular extensions of $\cat{C}$ over $\cat{A}$ will be denoted by $\MME{C}$.
\end{df}

\begin{rmk}
	In \cite{Mueger2002}, a minimal modular extension of $\cat{C}$ over $\cat{A}$ is defined to be a \emph{modular} tensor category $\cat{M}$ such that 
	$$
	\dim \cat{M}=  \dim \MC{C}{C}\cdot \dim \cat{C},
	$$
	wher $\dim \cat{M}$ denotes the global dimension of $\cat{M}$. Our notion of minimal modular extension agrees with this under the additional assumption that one works with ribbon categories over $\C$, rather than braided tensor categories. By \cite{Mueger2002} a pre-modular category $\cat{N}$ (over $\C$, or more generally with $\dim \cat{N}\neq 0$) is modular if and only if $\MC{N}{N}=\Vect$, and further
	$$
	\dim \cat{N}=\dim \cat{A} \cdot \dim \MC{A}{N}.
	$$
	Now suppose that $\cat{N}$ is a modular tensor category containing $\cat{C}$ and that $\cat{A}=\MC{C}{C}$. Observe that $\cat{C}\subset \MC{A}{N}$. So $\MC{A}{N}=\cat{C}$ is equivalent to
	$$
	\dim \cat{N}=  \dim \cat{A}\cdot \dim \cat{C}.
	$$
\end{rmk}

The reduced tensor product works particularly well with minimal modular extensions:

\begin{prop}
	Let $\cat{M}\in \MME{C}$ and $\cat{N}\in \MME{D}$ for $\cat{C},\cat{D}\in \BTCA$ with $\MC{A}{C}\cong \MC{A}{D}=\cat{A}$. Then $\cat{M}\rt \cat{N}\in \mme(\cat{C}\cattens{\cat{A}}\cat{D},\cat{A})$.
\end{prop}

\begin{proof}
	We observe that $\MC{\cat{A}}{\cat{M}\rt \cat{N}}=\cat{C}\cattens{\cat{A}}\cat{D}$ is immediate from Proposition \ref{commutantrt}. This leaves showing that $\cat{M}\rt \cat{N}$ has $\Vect$ as its subcategory of transparent objects. To see this, observe that by the double commutant theorem \cite{Mueger2002}, we have that 
	$$\cat{Z}_2(\MC{A}{M},\cat{M})=\MC{A}{A}=\cat{A}$$ and $$\cat{Z}_2(\MC{A}{N},\cat{N})=\MC{A}{A}=\cat{A}.$$ We then have
	\begin{align*}
		\cat{Z}_2(\cat{Z}_2(\cat{A},\cat{M}\rt \cat{N}),\cat{M}\rt \cat{N})&=\cat{Z}_2(\cat{Z}_2(\cat{A},\cat{M})\cattens{\cat{A}}\MC{A}{N},\cat{M}\rt \cat{N})\\&=\cat{A}\cattens{\cat{A}}\cat{A}=\cat{A},
	\end{align*}
	as the braiding on $\cat{M}\rt \cat{N}$ is componentwise. We further have that $$\cat{Z}_2(\cat{Z}_2(\cat{A},\cat{M}\rt \cat{N}))\supset \cat{Z}_(\cat{M}\rt\cat{N},\cat{M}\rt\cat{N}),$$ so we are left with establishing which objects $a$ in $\cat{A}$ are transparent for all of $\cat{M}\rt\cat{N}$. This is detected by the $\dcentcat{A}_s$-tensoring restricted to the non-$\cat{A}$ objects of the subcategory $\cat{Z}_s(\langle a \rangle,\dcentcat{A})$, where $\langle a \rangle$ denotes the subcategory spanned by $a$: the object $a$ is transparent if and only if these objects annihilate $\cat{M}\rt \cat{N}$. As, by modularity, no such set of objects of $\dcentcat{A}$ annihilates $\cat{M}$ or $\cat{N}$, and therefore no non-unit objects of $\cat{A}$ are transparent in $\cat{M}\rt\cat{N}$, and we conclude that $\cat{M}\rt \cat{N}$ is modular.
\end{proof}

\begin{rmk}
	We observe that the reduced tensor product gives $\MME{A}$ the structure of an abelian group. For the case $\cat{A}=\Rep(G)$ this abelian group was, in \cite{Lan2016}, identified with the group $H^3(G,U(1))$, and the reduced tensor product corresponds to the pairing given there between sets of minimal modular extensions. The advantage of the approach given here is that the constructions are done purely in terms of the modular structure of the categories involved, using only the braidings and the fusion rules.
\end{rmk}


\begin{thebibliography}{BDSPV15}
	
	\bibitem[AK02]{KirillovJr.2002}
	Jr. Alexander~Kirillov.
	\newblock Modular categories and orbifold models.
	\newblock {\em Commun. Math. Phys.}, 229:309--335, 2002.
	
	\bibitem[BDSPV15]{Bartlett2015a}
	Bruce Bartlett, Christopher~L. Douglas, Christopher~J. Schommer-Pries, and
	Jamie Vicary.
	\newblock Modular categories as representations of the 3-dimensional bordism
	2-category.
	\newblock {\em ArXiv:1509.06811}, 2015.
	
	
	\bibitem[BGH{\etalchar{+}}17]{Bruillard2016}
	Paul Bruillard, César Galindo, Tobias Hagge, Siu-Hung Ng, Julia~Yael Plavnik,
	Eric~C. Rowell, and Zhenghan Wang.
	\newblock Fermionic modular categories and the 16-fold way.
	\newblock {\em J. Math. Phys.}, 58:041704, 2017.
	
	\bibitem[Bru00]{Bruguieres2000}
	Alain Bruguières.
	\newblock Catégories prémodulaires, modularisations et invariants de
	variétés de dimension 3.
	\newblock {\em Math. Annalen}, 316:215--236, 2000.
	
	\bibitem[Del90]{Deligne1990}
	Pierre Deligne.
	\newblock Catégories Tannakiennes.
	\newblock {\em The Grothendieck Festschrift}, II:111--195, 1990.
	
	\bibitem[Del02]{Deligne2002}
	Pierre Deligne.
	\newblock Catégories tensorielles.
	\newblock {\em Mosc. Math. J.}, 2:227--248, 2002.
	
	\bibitem[DGNO10]{Drinfeld2009}
	Vladimir Drinfeld, Shlomo Gelaki, Dmitri Nikshych, and Victor Ostrik.
	\newblock On braided fusion categories I.
	\newblock {\em Selecta Math. (N.S.)}, 16:1--119, 2010.
	
	\bibitem[Dri]{Drinfeld2009a}
	Vladimir Drinfeld.
	\newblock Reduced tensor product.
	\newblock {\em Unpublished work, note in circulation.}
	
	\bibitem[DSPS14]{Douglas2014a}
	Christopher~L. Douglas, Christopher Schommer-Pries, and Noah Snyder.
	\newblock The balanced tensor product of module categories.
	\newblock {\em ArXiv:1406.4204v3}, 2014.
	
	\bibitem[DSPS19]{Douglas2014}
	Christopher~L. Douglas, Christopher Schommer-Pries, and Noah Snyder.
	\newblock The balanced tensor product.
	\newblock {\em Kyoto J. Math.}, 59:167--179, 2019.
	
	\bibitem[ENO10]{Etingof2009}
	Pavel Etingof, Dmitri Nikshych, and Viktor Ostrik.
	\newblock Fusion categories and homotopy theory.
	\newblock {\em Quantum Topol.}, 1:209--273, 8 2010.
	
	\bibitem[JS86]{Joyal1986}
	André Joyal and Ross Street.
	\newblock Braided monoidal categories.
	\newblock {\em Macquarie Math. Reports}, 1986.
	
	\bibitem[LKW17a]{Lan2016a}
	Tian Lan, Liang Kong, and Xiao-Gang Wen.
	\newblock Classification of (2+1)-dimensional topological order and
	symmetry-protected topological order for bosonic and fermionic systems with
	on-site symmetries.
	\newblock {\em Phys. Rev. B}, 95:235140, 2017.
	
	\bibitem[LKW17b]{Lan2016}
	Tian Lan, Liang Kong, and Xiao-Gang Wen.
	\newblock Modular extensions of unitary braided fusion categories and 2+1d
	topological/spt orders with symmetries.
	\newblock {\em Commun. Math. Phys.}, 351:709--739, 2017.
	
	\bibitem[Maj91]{Majid1991}
	Shahn Majid.
	\newblock Representations, duals and quantum doubles of monoidal categories.
	\newblock volume II.-Supple, pages 197--206. Circolo Matematico di
	Palermo(Palermo), 1991.
	
	\bibitem[Mü03]{Mueger2002}
	Michael Müger.
	\newblock On the structure of modular categories.
	\newblock {\em Proc. London Math. Soc. (3)}, 87:291--308, 2003.
	
	\bibitem[Mü04]{Muger2004}
	Michael Müger.
	\newblock Galois extensions of braided tensor categories and braided crossed
	G-categories.
	\newblock {\em J. Algebra}, 277:256--281, 7 2004.
	
	\bibitem[Mü10]{Turaev2010a}
	Michael Müger.
	\newblock On the structure of braided crossed G-categories.
	\newblock {\em Homotopy Quantum Field Theory, Appendix}, pages 221--235, 2010.
	
	\bibitem[Ost03]{Ostrik2003}
	Victor Ostrik.
	\newblock Module categories, weak hopf algebras and modular invariants.
	\newblock {\em Transform. Groups}, 8:177--206, 2003.
	
	\bibitem[SW20]{Schweigert2018}
	Christoph Schweigert and Lukas Woike.
	\newblock Extended homotopy quantum field theories and their orbifoldization.
	\newblock {\em J. Pure Appl. Algebra}, 224:106213, 8 2020.
	\newblock Extended homotopy quantum field theories and their orbifoldization.
	
	\bibitem[Was17]{Wasserman2017e}
	Thomas~A. Wasserman.
	\newblock {\em A Reduced Tensor Product of Braided Fusion Categories over a
		Symmetric Fusion Category}.
	\newblock PhD thesis, University of Oxford, 2017.
	
	\bibitem[Was20a]{Wasserman2017a}
	Thomas~A. Wasserman.
	\newblock The Drinfeld centre of a symmetric fusion category is 2-fold
	monoidal.
	\newblock {\em Adv. Math.}, 366:107090, 2020.
	
	\bibitem[Was20b]{Wasserman2017}
	Thomas~A. Wasserman.
	\newblock The symmetric tensor product on the Drinfeld centre of a symmetric
	fusion category.
	\newblock {\em J. Pure Appl. Algebra}, 224:106348, 2020.
	
	\bibitem[Was21]{Wasserman2017b}
	Thomas~A. Wasserman.
	\newblock Drinfeld centre-crossed braided tensor categories.
	\newblock {\em High. Struct.}, 5:204--243, 2021.
	
\end{thebibliography}
\end{document}